\documentclass[a4paper,12pt,intlimits,oneside]{amsart}

\usepackage{amsmath}
\usepackage{amssymb}
\usepackage{amsfonts}
\usepackage{amsthm,amscd}

\newtheorem{thm}{Theorem}[section]
\newtheorem{prop}[thm]{Proposition}
\newtheorem{cor}[thm]{Corollary}
\newtheorem{lem}[thm]{Lemma}
\newtheorem{defn}[thm]{Definition}
\newtheorem{remark}[thm]{Remark}

\theoremstyle{definition}
\newcommand{\comment}[1]{}

\numberwithin{equation}{section}

\def\lsim{\raisebox{-1ex}{$~\stackrel{\textstyle <}{\sim}~$}}

\theoremstyle{definition}

\begin{document}
\title[Dual of $\mathcal{H}_{\mathrm{loc}}^{(q,p)}$ and Pseudo-differential operators]{Duals of Hardy-amalgam spaces $\mathcal{H}_{\mathrm{loc}}^{(q,p)}$ and Pseudo-differential operators}
\author[Z.V.P. Abl\'e]{Zobo Vincent de Paul Abl\'e}
\address{Laboratoire de Math\'ematiques Fondamentales, UFR Math\'ematiques et Informatique, Universit\'e F\'elix Houphou\"et-Boigny Abidjan-Cocody, 22 B.P 582 Abidjan 22. C\^ote d'Ivoire}
\email{{\tt vincentdepaulzobo@yahoo.fr}}
\author[J. Feuto]{Justin Feuto}
\address{Laboratoire de Math\'ematiques Fondamentales, UFR Math\'ematiques et Informatique, Universit\'e F\'elix Houphou\"et-Boigny Abidjan-Cocody, 22 B.P 1194 Abidjan 22. C\^ote d'Ivoire}
\email{{\tt justfeuto@yahoo.fr}}


\date{}

\begin{abstract}
In this paper, we carry on with the study of the Hardy-Amalgam spaces $\mathcal{H}_{\mathrm{loc}}^{(q,p)}$ spaces introduced in \cite{AbFt}. We investigate their dual spaces and establish some results of boundedness of pseudo-differential operators in these spaces.  
\end{abstract}

\maketitle

\section{Introduction}
The celebrated paper of C. Fefferman and E. Stein \cite{FS} has been crucial in recent developments of the real variable theory of Hardy spaces. Given a function $\varphi\in\mathcal C^\infty(\mathbb R^d)$ with support on $B(0,1)$ such that $\int_{\mathbb R^d}\varphi dx=1$, where $B(0,1)$ is the unit open ball centered at $0$ and $\mathcal C^\infty(\mathbb R^d)$ denotes the space of infinitely differentiable complex values functions on $\mathbb R^d$, and 
$t>0$, we denote by $\varphi_t$ the dilated function $\varphi_t(x)=t^{-d}\varphi(x/t)$, $x\in\mathbb R^d$. The Hardy space $\mathcal H^q:=\mathcal H^q(\mathbb R^d)$ is defined as the space of tempered distributions $f$ such that the maximal function
\begin{equation}
\mathcal M_{\varphi}(f):=\sup_{t>0}|f\ast\varphi_t| \label{maximal}
\end{equation}
is in $L^q(\mathbb R^d)$, while its local version $\mathcal{H}_{\mathrm{loc}}^q:=\mathcal{H}_{\mathrm{loc}}^q(\mathbb R^d)$ introduced by D. Goldberg in \cite{DGG} is that for which ${\mathcal{M}_{\mathrm{loc}}}_{_{\varphi}}(f)\in L^q(\mathbb R^d)$, where ${\mathcal{M}_{\mathrm{loc}}}_{_{\varphi}}(f)$ is defined as in (\ref{maximal}) with $0<t\leq 1$. These spaces have been the subject of several studies and generalizations (see \cite{AT}, \cite{MBOW}, \cite{MBBLDYYZ}, \cite{DCUSLW}, \cite{Gc}, \cite{JHWH}, \cite{Ky}, \cite{RL}, \cite{NEYS}, \cite{NEYS1}, \cite{SAY}, \cite{LT}, \cite{YDYS} and \cite{YDYSB}). Also, following the maximal function approach and using Amalgams, we introduced in \cite{AbFt} the spaces $\mathcal{H}^{(q,p)}$ and $\mathcal{H}_{\mathrm{loc}}^{(q,p)}$ which generalize $\mathcal H^q$ and $\mathcal{H}_{\mathrm{loc}}^q$ spaces.

For recall, Amalgams arise naturally in harmonic analysis and were introduced by N. Wiener \cite{NW} in 1926. For a systematic study of these spaces and their role in Fourier analysis, we refer to \cite{BDD}, \cite{FSTW} and \cite{FH}. For $0<q,p\leq+\infty$, the amalgam of $L^{q} := L^{q}(\mathbb R^{d})$ and $L^{p} := L^{p}(\mathbb R^{d})$ is the space $(L^{q},\ell^{p}) :=(L^{q},\ell^{p})(\mathbb R^{d})$ of measurable functions $f :\mathbb R^{d}\rightarrow\mathbb C$ which are locally in $L^{q}$ and such that the sequence $\left\{\left\|f\chi_{Q_{k}}\right\|_{q}\right\}_{k\in\mathbb Z^{d}}$ belongs to $\ell^{p}:=\ell^{p}(\mathbb Z^{d})$, where $Q_k:=k+\left[0,1\right)^d=\prod_{i=1}^d[k_i,k_i+1)$, $\chi_{Q_{k}}$ denotes the characteristic function of $Q_{k}$ and $\left\|f\right\|_{q}:=\left(\int_{\mathbb R^{d}}\left|f(x)\right|^{q}dx\right)^{\frac{1}{q}}$, with the usual modification when $q=+\infty$. In \cite{AbFt}, we defined the Hardy-amalgam spaces $\mathcal{H}^{(q,p)}$ and $\mathcal{H}_{\mathrm{loc}}^{(q,p)}$, for $0<q,p<+\infty$, by taking the Wiener-amalgam "norm" in the definition of Hardy space $\mathcal{H}^q$ and local Hardy space $\mathcal{H}_{\mathrm{loc}}^q$ instead of the Lebesgue "norm". Next, in \cite{AbFt2}, we characterized the dual space of $\mathcal{H}^{(q,p)}$, whenever $0<q\leq p\leq 1$, and obtained some results of boundedness of some classical linear operators such as Calder\'on-Zygmund, convolution and Riesz potential operators, when $0<q\leq 1$ and $q\leq p<+\infty$.

The aim of this paper is to extend some well known results for $\mathcal{H}_{\mathrm{loc}}^q$ spaces to Hardy-amalgam spaces $\mathcal{H}_{\mathrm{loc}}^{(q,p)}$, when $0<q\leq 1$ and $q\leq p<+\infty$, as we did in \cite{AbFt2} for $\mathcal{H}^{(q,p)}$. This article is organized as follows. 

In Section 2, we recall some properties of Wiener amalgam spaces $(L^q,\ell^p)$ and Hardy-amalgam spaces $\mathcal{H}^{(q,p)}$ and $\mathcal{H}_{\mathrm{loc}}^{(q,p)}$ obtained in \cite{AbFt} we will need. In Section 3, we establish the atomic and molecular decompositions of $\mathcal{H}_{\mathrm{loc}}^{(q,p)}$ spaces, when $0<q\leq 1$ and $q\leq p<+\infty$, which correspond to the local versions of those obtained in \cite{AbFt} for $\mathcal{H}^{(q,p)}$.  

Next, in Section 4, with some results of Sections 2 and 3, we characterize the dual space of $\mathcal{H}_{\mathrm{loc}}^{(q,p)}$, whenever $0<q\leq p\leq 1$. Lastly, in Section 5, as applications of the results of Sections 3 and 4, we study the boundedness of pseudo-differential operators in $\mathcal{H}_{\mathrm{loc}}^{(q,p)}$ spaces, when $0<q\leq 1$ and $q\leq p<+\infty$. 

Throughout the paper, we let $\mathbb{N}=\left\{1,2,\ldots\right\}$ and $\mathbb{Z}_{+}=\left\{0,1,2,\ldots\right\}$. We use $\mathcal S := \mathcal S(\mathbb R^{d})$ will denote the Schwartz class of rapidly decreasing smooth functions endowed with the topology defined by the family of norms $\left\{\mathcal{N}_{m}\right\}_{m\in\mathbb{Z}_{+}}$, where for all $m\in\mathbb{Z}_{+}$ and $\psi\in\mathcal{S}$, $$\mathcal{N}_{m}(\psi):=\underset{x\in\mathbb R^{d}}\sup(1 + |x|)^{m}\underset{|\beta|\leq m}\sum|{\partial}^\beta \psi(x)|$$ with $|\beta|=\beta_1+\ldots+\beta_d$, ${\partial}^\beta=\left(\partial/{\partial x_1}\right)^{\beta_1}\ldots\left(\partial/{\partial x_d}\right)^{\beta_d}$ for all $\beta=(\beta_1,\ldots,\beta_d)\in\mathbb{Z}_{+}^d$ and $|x|:=(x_1^2+\ldots+x_d^2)^{1/2}$. The dual space of $\mathcal S$ is the space of tempered distributions denoted by $\mathcal S':= \mathcal S'(\mathbb R^{d})$ equipped with the weak-$\ast$ topology. If $f\in\mathcal{S'}$ and $\theta\in\mathcal{S}$, we denote the evaluation of $f$ on $\theta$ by $\left\langle f,\theta\right\rangle$. The letter $C$ will be used for non-negative constants independent of the relevant variables that may change from one occurrence to another. When a constant depends on some important parameters $\alpha,\gamma,\ldots$, we denote it by $C(\alpha,\gamma,\ldots)$. Constants with subscript, such as $C_{\alpha,\gamma,\ldots}$, do not change in different occurrences but depend on the parameters mentioned in it. We propose the following abbreviation $\mathrm{\bf A}\lsim \mathrm{\bf B}$ for the inequalities $\mathrm{\bf A}\leq C\mathrm{\bf B}$, where $C$ is a positive constant independent of the main parameters. If $\mathrm{\bf A}\lsim \mathrm{\bf B}$ and $\mathrm{\bf B}\lsim \mathrm{\bf A}$, then we write $\mathrm{\bf A}\approx \mathrm{\bf B}$. For any given (quasi-) normed spaces $\mathcal{A}$ and $\mathcal{B}$ with the corresponding (quasi-) norms $\left\|\cdot\right\|_{\mathcal{A}}$ and $\left\|\cdot\right\|_{\mathcal{B}}$, the symbol $\mathcal{A}\hookrightarrow\mathcal{B}$ means that for all $f\in\mathcal{A}$, then $f\in\mathcal{B}$ and $\left\|f\right\|_{\mathcal{B}}\lsim\left\|f\right\|_{\mathcal{A}}$.

For $\lambda>0$ and a cube $Q\subset\mathbb R^{d}$ (by a cube we mean a cube whose edges are parallel to the coordinate axes), we write $\lambda Q$ for the cube with same center as $Q$ and side-length $\lambda$ times side-length of $Q$, while $\left\lfloor \lambda \right\rfloor$ stands for the greatest integer less or equal to $\lambda$. Also, for $x\in\mathbb R^{d}$ and $\ell>0$, $Q(x,\ell)$ will denote the cube centered at $x$ and side-length $\ell$. We use the same notations for balls. We denote by  $E^{c}$ the set $\mathbb R^d\backslash{E}$ and $\mathrm{diam}(E):=\sup_{x,y\in E}|x-y|$. For a measurable set $E\subset\mathbb R^d$, we denote by $\chi_{_{E}}$ the characteristic function of $E$ and $\left|E\right|$ for the Lebesgue measure. Also, for any set $A$, we denote by $\# A:=\mathrm{card}(A)$ its cardinality.
 
Throughout this paper, without loss of generality and unless otherwise specified, we assume that cubes are closed and denote by $\mathcal{Q}$ the set of all cubes.

\section{Recalls on Wiener amalgam spaces $(L^q,\ell^p)$ and Hardy-amalgam spaces $\mathcal{H}^{(q,p)}$ and $\mathcal{H}_{\mathrm{loc}}^{(q,p)}$} 

\subsection{Wiener amalgam spaces $(L^q,\ell^p)$} 

For $f\in(L^q,\ell^p)$, we set $$\left\|f\right\|_{q,p}:=\left\|\left\{\left\|f\chi_{_{Q_k}}\right\|_{q}\right\}_{k\in\mathbb{Z}^d}\right\|_{\ell^{p}}.$$ Endowed with the (quasi)-norm $\left\|\cdot\right\|_{q,p}$, the amalgam space $(L^q,\ell^p)$ is a complete space. It is well known that it is a Banach space if $1\leq q, p\leq+\infty$ and its dual is $(L^{q'},\ell^{p'})$ for $1\leq q, p<+\infty$, where $\frac{1}{q}+\frac{1}{q'}=1$ and $\frac{1}{p}+\frac{1}{p'}=1$ with the usual conventions (see \cite{BDD}, \cite{RBHS}, \cite{FSTW}, \cite{FH} and \cite{JSTW}). The following result is also well known (see \cite{BDD}, \cite{RBHS}, \cite{FSTW}, \cite{FH} and \cite{JSTW}).

\begin{prop}\label{ratotute} Let $0<q,q_1,p,p_1\leq +\infty$. We have
\begin{enumerate}
\item $(L^q,\ell^q)=L^q$ with $\left\|\cdot\right\|_{q,q}=\left\|\cdot\right\|_{q}$.
\item $\left\|f\right\|_{q,p_1}\leq\left\|f \right\|_{q,p}$ , if $p\leq p_1$ and $f\in (L^q,\ell^p)$.
\item $\left\|f\right\|_{q,p}\leq\left\|f \right\|_{q_1,p}$ , if $q\leq q_1$ and $f\in (L^{q_1},\ell^p)$.
\end{enumerate}
\end{prop}

Also, we have the following reverse Minkowski's inequality for $(L^q, \ell^p)$, with $0<q<1$ and $0<p\leq 1$. 

\begin{prop}[\cite{AbFt2}, Proposition 2.2] \label{InverseHoldMink3}
Let $0<q<1$ and $0<p\leq 1$. For all finite sequence $\left\{f_n\right\}_{n=0}^m$ of elements of $(L^q, \ell^p)$, we have
\begin{eqnarray}
\sum_{n=0}^m\left\|f_n\right\|_{q,p}\leq\left\|\sum_{n=0}^m|f_n|\right\|_{q,p}. \label{InverseHoldMink4}
\end{eqnarray}
\end{prop}

\begin{prop}[\cite{AbFt}, (2.4) and (2.5), p. 1902] \label{okaido}
We have $\mathcal{S}\hookrightarrow(L^{q},\ell^{p})$, for $0<q, p\leq+\infty$. Also, $(L^{q},\ell^{p})\hookrightarrow\mathcal{S'}$, for $ 1\leq q, p\leq+\infty$.
\end{prop}

\begin{remark}\label{remqpourcompl}
For $0<q,p\leq+\infty$, the convergence in $\mathcal{S}$ implies the convergence in $(L^{q},\ell^{p})$.

Likewise, for $1\leq q, p\leq+\infty$, the convergence in $(L^{q},\ell^{p})$ implies the convergence in $\mathcal{S'}$. Consequently, for all $1\leq q\leq+\infty$ and $0<p\leq+\infty$, the convergence in $(L^{q},\ell^{p})$ implies the convergence in $\mathcal{S'}$, more precisely $(L^{q},\ell^{p})\hookrightarrow\mathcal{S'}$. 
\end{remark}

Another important property of amalgam spaces is the boundedness of the Hardy-Littlewood maximal operator $\mathfrak{M}$. Let $f$ be a locally integrable function. The centralized Hardy-Littlewood maximal function $\mathfrak{M}(f)$ is defined by
\begin{eqnarray*}
\mathfrak{M}(f)(x)=\underset{r>0}\sup\ |B(x,r)|^{-1}\int_{B(x,r)}|f(y)|dy,\ \ x\in\mathbb{R}^d.
\end{eqnarray*} 
Carton-Lebrun and al. proved in \cite{CLHS}, Theorems 4.2 and 4.5, the following result.

\begin{prop} \label{operamaximacalebr}
Let $1<q,p<+\infty$. Then, for any locally integrable function $f$,
\begin{eqnarray*}
\left\|\mathfrak{M}(f)\right\|_{q,p}\lsim\left\|f\right\|_{q,p}.
\end{eqnarray*}
\end{prop}

A very useful generalization of this result is the following 

\begin{prop}[\cite{LSUYY}, Proposition 11.12] \label{operamaxima}
Let $1<p,u\leq+\infty$ and $1<q<+\infty$. Then, for all sequence of measurable functions $\left\{f_n\right\}_{n\geq 0}$,
\begin{eqnarray*}
\left\|\left(\underset{n=0}{\overset{+\infty}\sum}\left[\mathfrak{M}(f_n)\right]^{u}\right)^{\frac{1}{u}}\right\|_{q,p}\approx\left\|\left(\underset{n=0}{\overset{+\infty}\sum}|f_n|^{u}\right)^{\frac{1}{u}}\right\|_{q,p}, 
\end{eqnarray*}
with the implicit equivalent positive constants independent of $\left\{f_n\right\}_{n\geq 0}$ .
\end{prop}

\subsection{Hardy-amalgam spaces $\mathcal{H}^{(q,p)}$ and $\mathcal{H}_{\mathrm{loc}}^{(q,p)}$} 

Let $0<q,p<+\infty$ and $\varphi\in\mathcal{S}$ with $\text{supp}(\varphi)\subset B(0,1)$ and $\int_{\mathbb{R}^d}\varphi(x)dx=1$. In \cite{AbFt}, we defined the space $\mathcal{H}^{(q,p)}$ as the space of tempered distributions $f$ having their maximal function $\mathcal{M}_{\varphi}(f)$ in $(L^q,\ell^p)$. For $f\in\mathcal{H}^{(q,p)}$, we set $$\left\|f\right\|_{\mathcal{H}^{(q,p)}}:=\left\|\mathcal{M}_{\varphi}(f)\right\|_{q,p}.$$ Similarly, we defined the space $\mathcal{H}_{\mathrm{loc}}^{(q,p)}$ by replacing the maximal operator $\mathcal{M}_{\varphi}$ by its local version ${\mathcal{M}_{\mathrm{loc}}}_{_{\varphi}}$. Also, $\left\|\cdot\right\|_{\mathcal{H}^{(q,p)}}$ and $\left\|\cdot\right\|_{\mathcal{H}_{\mathrm{loc}}^{(q,p)}}$ respectively define norms on $\mathcal{H}^{(q,p)}$ and $\mathcal{H}_{\mathrm{loc}}^{(q,p)}$, whenever $1\leq q,p<+\infty$, and quasi-norms otherwise. Moreover,
\begin{eqnarray}
\left\|f\right\|_{\mathcal{H}_{\mathrm{loc}}^{(q,p)}}\leq\left\|f\right\|_{\mathcal{H}^{(q,p)}}, \label{depous}
\end{eqnarray}
for all $f\in\mathcal{H}^{(q,p)}$. We obtained the following result in \cite{AbFt}.

\begin{thm}[\cite{AbFt}, Theorem 3.2] \label{gto} Let $1\leq q, p<+\infty$. 
\begin{enumerate}
\item If $1< q, p<+\infty$, then $\mathcal{H}^{(q,p)}=\mathcal{H}_{\mathrm{loc}}^{(q,p)}=\left(L^q,\ell^p\right)$, with norms equivalences. 
\item If $q=1$, then $\mathcal{H}^{(1,p)}\subset\mathcal{H}_{\mathrm{loc}}^{(1,p)}\subset\left(L^1,\ell^p\right)$. Furthermore, 
\begin{eqnarray*}
\left\|f\right\|_{1,p}\leq\left\|f\right\|_{\mathcal{H}_{\mathrm{loc}}^{(1,p)}}\leq\left\|f\right\|_{\mathcal{H}^{(1,p)}}, 
\end{eqnarray*}
for all $f\in\mathcal{H}^{(1,p)}$.
\end{enumerate}
\end{thm} 

\begin{cor}[\cite{AbFt}, Remark 3.3] \label{bravoooheiner}
Let $1\leq q<+\infty$ and $0<p<+\infty$. Then, $\mathcal{H}^{(q,p)}\subset\mathcal{H}_{\mathrm{loc}}^{(q,p)}\subset L_{\mathrm{loc}}^1$.
\end{cor}

Also, in \cite{AbFt}, we showed that $\mathcal{H}^{(q,p)}$ and $\mathcal{H}_{\mathrm{loc}}^{(q,p)}$ can be characterized with other maximal functions and do not depend on the choice of the function $\varphi$. Endeed, given $a, b>0$, $\Phi\in\mathcal{S}$ and $f\in\mathcal{S'}$, we define the non-tangential maximal function $\mathcal{M}_{\Phi,a}^{\ast}(f)$ of $f$ with respect to $\Phi$ as $$\mathcal{M}_{\Phi,a}^{\ast}(f)(x):=\underset{t>0}\sup\left\{\underset{\left|x-y\right|\leq at}\sup\left|(f\ast\Phi_t)(y)\right|\right\},\ x\in\mathbb{R}^d$$ and the auxiliary maximal function $\mathcal{M}_{\Phi,b}^{\ast\ast}(f)$ of $f$ with respect to $\Phi$ as $$\mathcal{M}_{\Phi,b}^{\ast\ast}(f)(x):=\underset{t>0}\sup\left\{\underset{y\in\mathbb{R}^d}\sup\frac{\left|(f\ast\Phi_t)(x-y)\right|}{(1+t^{-1}|y|)^b}\right\},\ x\in\mathbb{R}^d.$$ Moreover, given a positive integer $N$, we define $$\mathfrak{N}_N(\phi):=\int_{\mathbb{R}^d}(1+|x|)^N\left(\sum_{|\alpha|\leq N+1}|\partial^{\alpha}\phi(x)|\right)dx,\ \ \phi\in\mathcal{S}$$ and $$\mathcal{F}_N:=\left\{\phi\in\mathcal{S}:\ \mathfrak{N}_N(\phi)\leq 1\right\}.$$ The non-tangential grand maximal function of $f$ is defined as $$\mathcal{M}_{\mathcal{F}_N}(f)(x):=\underset{\psi\in\mathcal{F}_N}\sup\mathcal{M}_{\psi,1}^{\ast}(f)(x),\ x\in\mathbb{R}^d$$ and the radial grand maximal function of $f$ as $$\mathcal{M}_{\mathcal{F}_N}^0(f)(x):=\underset{\psi\in\mathcal{F}_N}\sup\mathcal{M}_{\psi}(f)(x),\ x\in\mathbb{R}^d.$$ 

Similarly, we define the local versions of these maximal functions by taking $t$ in $(0;1]$. We have the following theorem (see \cite{AbFt}, Theorem 3.7, p. 1910).

\begin{thm}\label{bravoooh}
Let $0<q,p<+\infty$ and $\Phi\in\mathcal{S}$ with $\int_{\mathbb{R}^d}\Phi(x)dx\neq 0$. Then, for any $f\in\mathcal{S'}$, the following assertions are equivalent: 
\begin{enumerate}
\item $\mathcal{M}_{\Phi}(f)\in(L^q,\ell^p)$.
\item For all $a>0$, $\mathcal{M}_{\Phi,a}^{\ast} (f)\in(L^q,\ell^p)$.
\item For all  $b>\max\left\{\frac{d}{q},\frac{d}{p}\right\}$, $\mathcal{M}_{\Phi,b}^{\ast\ast} (f)\in(L^q,\ell^p)$. 
\item There exists an integer $N>\max\left\{\frac{d}{q},\frac{d}{p}\right\}$ such that $\mathcal{M}_{\mathcal{F}_N} (f)\in(L^q,\ell^p)$.
\end{enumerate}
Moreover, for any $N\geq\max\left\{\left\lfloor \frac{d}{q}\right\rfloor,\left\lfloor \frac{d}{p}\right\rfloor\right\}+1$ and $\max\left\{\frac{d}{q},\frac{d}{p}\right\}<b<N$ with $N=\left\lfloor b\right\rfloor+1$, 
\begin{eqnarray*}
\left\|\mathcal{M}_{\mathcal{F}_N}^0(f)\right\|_{q,p}\approx\left\|\mathcal{M}_{\mathcal{F}_N}(f)\right\|_{q,p}\approx\left\|\mathcal{M}_{\Phi}(f)\right\|_{q,p}\approx\left\|\mathcal{M}_{\Phi,b}^{\ast\ast}(f)\right\|_{q,p}\approx\left\|\mathcal{M}_{\Phi,a}^{\ast}(f)\right\|_{q,p}.
\end{eqnarray*}
\end{thm}

\begin{remark}
Theorem \ref{bravoooh} holds when the maximal functions are replaced by their local versions. Thus, we obtain a variety of maximal characterizations for $\mathcal{H}^{(q,p)}$ and $\mathcal{H}_{\mathrm{loc}}^{(q,p)}$ and their independence of the choice of the function $\varphi$. 
\end{remark}
 
For the sequel, we retain three equivalences in Theorem \ref{bravoooh} and their local versions; more precisely, for any $N\geq\max\left\{\left\lfloor \frac{d}{q}\right\rfloor,\left\lfloor \frac{d}{p}\right\rfloor\right\}+1$, 
\begin{eqnarray}
\left\|\mathcal{M}_{\mathcal{F}_N}^0(f)\right\|_{q,p}\approx\left\|\mathcal{M}_{\mathcal{F}_N}(f)\right\|_{q,p}\approx\left\|\mathcal{M}_{\Phi}(f)\right\|_{q,p}. \label{jaifito}
\end{eqnarray}

We have the following remark (see \cite{AbFt}, (3.12) and (3.13), pp. 1911-1912).

\begin{remark} \label{remarqconver}
Let $f\in\mathcal{H}^{(q,p)}$ and $\Phi\in\mathcal{S}$. We have  
\begin{eqnarray}
|f\ast\Phi(x)|\leq C(\Phi,d,q,p)\left\|\mathcal{M}_{\mathcal{F}_N}(f)\right\|_{q,p}\ , \label{converop}
\end{eqnarray}
for all $x\in\mathbb{R}^d$, with $C(\Phi,d,q,p):=C(d,q,p)\mathfrak{N}_N(\Phi)$ a nonnegative constant independent of $x$ and $f$. Thus, $f$ is a bounded distribution,
\begin{eqnarray}
|\left\langle f,\Phi\right\rangle|\leq C(d,q,p)\mathfrak{N}_N(\Phi)\left\|\mathcal{M}_{\mathcal{F}_N}(f)\right\|_{q,p}\label{conver}
\end{eqnarray}
and $\mathcal{H}^{(q,p)}\hookrightarrow\mathcal{S'}$. Also,  
\begin{eqnarray}
\left\|f\ast\Phi\right\|_{q,p}\leq\mathfrak{N}_N(\Phi)\left\|\mathcal{M}_{\mathcal{F}_N}(f)\right\|_{q,p}. \label{0conver1}
\end{eqnarray}
\end{remark}

In Remark \ref{remarqconver}, we can replace $\mathcal{H}^{(q,p)}$ by $\mathcal{H}_{\mathrm{loc}}^{(q,p)}$ and $\mathcal{M}_{\mathcal{F}_N}$ by its local version $\mathcal{M}_{{\mathrm{loc}}_{_{\mathcal{F}_N}}}$.

\begin{cor} [\cite{AbFt}, Proposition 3.8] \label{priscatun}
The convergence in $\mathcal{H}^{(q,p)}$ and $\mathcal{H}_{\mathrm{loc}}^{(q,p)}$ implies the convergence in $\mathcal{S'}$. Furthermore, $\mathcal{H}^{(q,p)}$ and $\mathcal{H}_{\mathrm{loc}}^{(q,p)}$ are Banach spaces whenever $1\leq q,p<+\infty$ and quasi-Banach spaces otherwise. 
\end{cor}

We give a very useful lemma in passing from $\mathcal{H}^{(q,p)}$ to $\mathcal{H}_{\mathrm{loc}}^{(q,p)}$.  

\begin{lem} \label{lemmdepass}
Let  $0<q,p<+\infty$ and $f\in\mathcal{H}_{\mathrm{loc}}^{(q,p)}$. Then, there exist $u\in\mathcal{H}^{(q,p)}$ and $v\in\mathcal{H}_{\mathrm{loc}}^{(q,p)}\cap\mathcal{C}^{\infty}(\mathbb{R}^d)$ such that $f=u+v$ and 
\begin{eqnarray*}
\left\|f\right\|_{\mathcal{H}_{\mathrm{loc}}^{(q,p)}}\approx\left\|u\right\|_{\mathcal{H}^{(q,p)}}+\left\|v\right\|_{\mathcal{H}_{\mathrm{loc}}^{(q,p)}},
\end{eqnarray*}
with the implicit equivalent positive constants independent of $f$.
\end{lem}
\begin{proof} 
The idea of our proof is the one of the classical case, namely in passing from $\mathcal{H}^{q}$ to $\mathcal{H}_{\mathrm{loc}}^{q}$ (see \cite{GH}, Lemma 2.2.5, p. 55). One can also see \cite{DGG}, Lemma 4. Fix an integer $N\geq\max\left\{\left\lfloor \frac{d}{q}\right\rfloor,\left\lfloor \frac{d}{p}\right\rfloor\right\}+1$ with $N>d$. Let $\psi\in\mathcal{C}^{\infty}(\mathbb{R}^d)$ with $\text{supp}(\psi)\subset B(0,2)$, $0\leq\psi\leq 1$ and $\psi\equiv1$ on $B(0,1)$. Set $\theta(x):=\left(\mathcal{F}^{-1}(\psi)\right)(x)$, $v(x):=(f\ast\theta)(x)$, for all $x\in\mathbb{R}^d$, and $u:=f-f\ast\theta$, where $\mathcal{F}^{-1}$ denotes the inverse Fourier transform $\mathcal{F}$ defined by $\left(\mathcal{F}^{-1}(\psi)\right)(x)=\left(\mathcal{F}(\psi)\right)(-x)=:\widehat{\psi}(-x)$ and $$\left(\mathcal{F}(\psi)\right)(x)=:\widehat{\psi}(x)=\int_{\mathbb{R}^d}e^{-2\pi ix\xi}\psi(\xi)d\xi,\ \ x\in\mathbb{R}^d,$$ with $i^2=-1$ (see \cite{LGkos}, Definitions 2.2.8 and 2.2.13). Clearly, we have $v\in\mathcal{C}^{\infty}(\mathbb{R}^d)\cap\mathcal{S'}$, since $\theta\in\mathcal{S}$ by its  definition and $f\in\mathcal{S'}$. Fix $\phi\in\mathcal{S}$ with $\text{supp}(\widehat{\phi})\subset B(0,1)$, $0\leq\widehat{\phi}\leq 1$ and $\widehat{\phi}\equiv1$ on $B(0,1/2)$. Then, there exists a constant $C>0$ such that 
\begin{eqnarray}
C(\theta\ast\phi_t)\in\mathcal{F}_N, \label{lemmdepasspour1}
\end{eqnarray} 
for all $0<t\leq 1$. Endeed, since $\int_{\mathbb{R}^d}\phi(x)dx=\widehat{\phi}(0)=1$, we have $\theta\ast\phi_t\rightarrow\theta$ in $\mathcal{S}$ as $t\rightarrow0$. Therefore, for all multi-indexes $\beta$ and all $0<t\leq 1$, we have 
\begin{eqnarray*}
\left|\partial^\beta(\theta\ast\phi_t)(x)\right|\leq C_{\beta,N}(1+|x|)^{-2N}, 
\end{eqnarray*}
for all $x\in\mathbb R^d$, where $C_{\beta,N}>0$ is a constant independent of $t$ and $x$. Thus,
\begin{eqnarray*}
\mathfrak{N}_N(\theta\ast\phi_t)&=&\int_{\mathbb{R}^d}(1+|x|)^{N}\left(\sum_{|\beta|\leq N+1}\left|\partial^\beta(\theta\ast\phi_t)(x)\right|\right)dx\\
&\leq&C_{N}\int_{\mathbb{R}^d}(1+|x|)^{-N}dx=C(N)<+\infty,
\end{eqnarray*}
since $N>d$. Hence $C(N)^{-1}(\theta\ast\phi_t)\in\mathcal{F}_N$. This establishes (\ref{lemmdepasspour1}).  

From (\ref{lemmdepasspour1}), it follows that 
\begin{eqnarray*}
\left|C((f\ast\theta)\ast\phi_t)(x)\right|=\left|\left((C(\theta\ast\phi_t))_{1}\ast f\right)(x)\right|\leq\mathcal{M}_{{\mathrm{loc}}_{_{\mathcal{F}_N}}}^0(f)(x),
\end{eqnarray*}
for all $x\in\mathbb{R}^d$ and all $0<t\leq 1$. Hence  
\begin{eqnarray*}
\mathcal{M}_{\mathrm{loc}_\phi}(v)(x)=\sup_{0<t\leq 1}|((f\ast\theta)\ast\phi_t)(x)|\leq C\mathcal{M}_{{\mathrm{loc}}_{_{\mathcal{F}_N}}}^0(f)(x),
\end{eqnarray*}
for all $x\in\mathbb{R}^d$. Thus,   
\begin{equation}
\left\|\mathcal{M}_{\mathrm{loc}_\phi}(v)\right\|_{q,p}\leq C\left\|\mathcal{M}_{{\mathrm{loc}}_{_{\mathcal{F}_N}}}^0(f)\right\|_{q,p}\approx\left\|f\right\|_{\mathcal{H}_{\mathrm{loc}}^{(q,p)}}, \label{lemmdepass1}
\end{equation}
by the local version of (\ref{jaifito}). Since $\int_{\mathbb{R}^d}\phi(x)dx=\widehat{\phi}(0)=1\neq 0$, it follows from (\ref{lemmdepass1}) that $v\in\mathcal{H}_{\mathrm{loc}}^{(q,p)}$ and 
\begin{eqnarray}
\left\|v\right\|_{\mathcal{H}_{\mathrm{loc}}^{(q,p)}}\approx\left\|\mathcal{M}_{\mathrm{loc}_\phi}(v)\right\|_{q,p}\leq C\left\|f\right\|_{\mathcal{H}_{\mathrm{loc}}^{(q,p)}}, \label{lemmdepass2} 
\end{eqnarray}
by the local version of (\ref{jaifito}). Moreover, 
\begin{eqnarray*}
(u\ast\phi_t)(x)&=&\mathcal{F}^{-1}(\widehat{u\ast\phi_t})(x)=\mathcal{F}^{-1}(\widehat{u}\widehat{(\phi_t)})(x)=\mathcal{F}^{-1}(\widehat{u}[\widehat{(\phi)}(t.)])(x)\\
&=&\mathcal{F}^{-1}((\widehat{f}-\widehat{f\ast\theta})[\widehat{(\phi)}(t.)])(x)=\mathcal{F}^{-1}((\widehat{f}-\widehat{f}\widehat{\theta})[\widehat{(\phi)}(t.)])(x)\\
&=&\mathcal{F}^{-1}(\widehat{f}(1-\widehat{\theta})[\widehat{(\phi)}(t.)])(x)=\mathcal{F}^{-1}(\widehat{f}(1-\widehat{\mathcal{F}^{-1}(\psi)})[\widehat{(\phi)}(t.)])(x)\\
&=&\mathcal{F}^{-1}(\widehat{f}(1-\psi)[\widehat{(\phi)}(t.)])(x)
\end{eqnarray*}
and $(1-\psi)[\widehat{(\phi)}(t.)]\equiv0$ if $t>1$, since $\psi\equiv1$ on $B(0,1)$ and $\text{supp}(\widehat{(\phi)}(t.))\subset B(0,1/t)\subset B(0,1)$. Hence $(u\ast\phi_t)(x)=0$, for all $t>1$. Thus,
\begin{eqnarray*}
\mathcal{M}_\phi(u)(x)=\sup_{t>0}|(u\ast\phi_t)(x)|=\sup_{0<t\leq 1}|(u\ast\phi_t)(x)|=\mathcal{M}_{\mathrm{loc}_\phi}(u)(x) 
\end{eqnarray*}
and  
\begin{eqnarray*}
\left\|\mathcal{M}_\phi(u)\right\|_{q,p}=\left\|\mathcal{M}_{\mathrm{loc}_\phi}(u)\right\|_{q,p}=\left\|\mathcal{M}_{\mathrm{loc}_\phi}(f-v)\right\|_{q,p}\approx\left\|f-v\right\|_{\mathcal{H}_{\mathrm{loc}}^{(q,p)}}, 
\end{eqnarray*}
since $f,v\in\mathcal{H}_{\mathrm{loc}}^{(q,p)}$ and $\int_{\mathbb{R}^d}\phi(x)dx\neq 0$. It follows that $u\in\mathcal{H}^{(q,p)}$,
\begin{eqnarray}
\left\|u\right\|_{\mathcal{H}^{(q,p)}}\approx\left\|\mathcal{M}_\phi(u)\right\|_{q,p}=\left\|\mathcal{M}_{\mathrm{loc}_\phi}(u)\right\|_{q,p}\approx\left\|u\right\|_{\mathcal{H}_{\mathrm{loc}}^{(q,p)}} \label{lemmdepass2bis}
\end{eqnarray}
and
\begin{eqnarray*}
\left\|u\right\|_{\mathcal{H}^{(q,p)}}\approx\left\|u\right\|_{\mathcal{H}_{\mathrm{loc}}^{(q,p)}}=\left\|f-v\right\|_{\mathcal{H}_{\mathrm{loc}}^{(q,p)}}\leq C(q,p)\left(\left\|f\right\|_{\mathcal{H}_{\mathrm{loc}}^{(q,p)}}+\left\|v\right\|_{\mathcal{H}_{\mathrm{loc}}^{(q,p)}}\right),
\end{eqnarray*}
by (\ref{jaifito}) and its local version. Hence 
\begin{eqnarray*}
\left\|u\right\|_{\mathcal{H}^{(q,p)}}+\left\|v\right\|_{\mathcal{H}_{\mathrm{loc}}^{(q,p)}}\leq C\left\|f\right\|_{\mathcal{H}_{\mathrm{loc}}^{(q,p)}}+C\left\|v\right\|_{\mathcal{H}_{\mathrm{loc}}^{(q,p)}}\leq C\left\|f\right\|_{\mathcal{H}_{\mathrm{loc}}^{(q,p)}},
\end{eqnarray*}
by (\ref{lemmdepass2}). Conversely, since $v,u\in\mathcal{H}_{\mathrm{loc}}^{(q,p)}$, we have
\begin{eqnarray*}
\left\|f\right\|_{\mathcal{H}_{\mathrm{loc}}^{(q,p)}}=\left\|u+v\right\|_{\mathcal{H}_{\mathrm{loc}}^{(q,p)}}\leq C\left(\left\|u\right\|_{\mathcal{H}^{(q,p)}}+\left\|v\right\|_{\mathcal{H}_{\mathrm{loc}}^{(q,p)}}\right),
\end{eqnarray*}
by (\ref{lemmdepass2bis}). This completes the proof of Lemma \ref{lemmdepass}.
\end{proof} 

From now on, for simplicity, we denote $\mathcal{M}_{\mathcal{F}_N}^0$ by $\mathcal{M}^0$ and assume that $0<q\leq 1$ and $q\leq p<+\infty$. We recall the following definition which is also the one of an atom for $\mathcal{H}^{q}$.

\begin{defn}\label{defhqpatom}
Let $1<r\leq+\infty$ and $s\geq\left\lfloor d\left(\frac{1}{q}-1\right)\right\rfloor$ be an integer. A function $\textbf{a}$ is a $(q,r,s)$-atom on $\mathbb{R}^d$ for $\mathcal{H}^{(q,p)}$ if there exist a cube $Q$ such that  
\begin{enumerate}
\item $\text{supp}(\textbf{a})\subset Q$ ,
\item $\left\|\textbf{a}\right\|_r\leq|Q|^{\frac{1}{r}-\frac{1}{q}}$ ; \label{defratom1}
\item $\int_{\mathbb{R}^d}x^{\beta}\textbf{a}(x)dx=0$, for all multi-indexes $\beta$ with $|\beta|\leq s$ . \label{defratom2}
\end{enumerate}
\end{defn}

Condition (\ref{defratom2}) is called the vanishing condition or the vanishing moment. We denote by $\mathcal{A}(q,r,s)$ the set of all $(\textbf{a},Q)$ such that $\textbf{a}$ is a $(q,r,s)$-atom and $Q$ is the associated cube (with respect to Definition \ref{defhqpatom}). We obtained the following result in \cite{AbFt}. 

\begin{thm} [\cite{AbFt}, Theorem 4.4] \label{fondamentaltheo}
Let $\delta\geq\left\lfloor d\left(\frac{1}{q}-1\right)\right\rfloor$ be an integer. For all $f\in\mathcal{H}^{(q,p)}$, there exist a sequence $\left\{(\textbf{a}_n, Q^n)\right\}_{n\geq 0}$ in $\mathcal{A}(q,\infty,\delta)$ and a sequence of scalars $\left\{{\lambda}_n\right\}_{n\geq 0}$ such that  
\begin{eqnarray*}
f=\sum_{n\geq 0}{\lambda}_n\textbf{a}_n\ \ \text{in}\ \mathcal{S'}\ \text{and}\ \mathcal{H}^{(q,p)}, 
\end{eqnarray*}
and, for any real $\eta>0$,
\begin{eqnarray*}
\left\|\sum_{n\geq 0}\left(\frac{|\lambda_n|}{\left\|\chi_{_{Q^n}}\right\|_{q}}\right)^{\eta}\chi_{_{Q^{n}}}\right\|_{\frac{q}{\eta},\frac{p}{\eta}}^{\frac{1}{{\eta}}}\lsim\left\|f\right\|_{\mathcal{H}^{(q,p)}}.
\end{eqnarray*}
\end{thm}

Notice that in Theorem \ref{fondamentaltheo} the sequence of elements of $\mathcal{A}(q,\infty,\delta)$ can be replaced by a sequence of elements of $\mathcal{A}(q,r,\delta)$, with $1<r<+\infty$, since $\mathcal{A}(q,\infty,\delta)\subset\mathcal{A}(q,r,\delta)$. 

Also, we have the following result (see \cite{AbFt}, Proposition 4.5, p. 1920) which will be very useful for the next sections. 

\begin{prop} \label{propofonda}
Let $1<u,v<s\leq+\infty$. Then, for every sequence of scalars $\left\{\delta_n\right\}_{n\geq 0}$ and every sequence $\left\{\mathfrak b_n\right\}_{n\geq 0}$ in $L^s$ with the support of each $\mathfrak b_n$ embedded in a cube $Q^n$ and $\left\|\mathfrak b_n\right\|_s\leq|Q^n|^{\frac{1}{s}-\frac{1}{u}}$, we have 
\begin{eqnarray*}
\left\|\sum_{n\geq 0}|\delta_n \mathfrak b_n|\right\|_{u,v}\lsim\left\|\sum_{n\geq 0}\frac{|\delta_n|}{\left\|\chi_{_{Q^n}}\right\|_{u}}\chi_{_{Q^n}}\right\|_{u,v}, 
\end{eqnarray*}
with implicit constant independent of $\left\{\mathfrak b_n\right\}_{n\geq 0}$ and $\left\{\delta_n\right\}_{n\geq 0}$.
\end{prop}

\section{Atomic and molecular Decompositions of $\mathcal{H}_{\mathrm{loc}}^{(q,p)}$} 

In this section, our aim is to give atomic and molecular decompositions of $\mathcal{H}_{\mathrm{loc}}^{(q,p)}$ spaces, namely reconstruction and decomposition theorems, when $0<q\leq 1$ and $q\leq p<+\infty$. In fact, in our earlier paper \cite{AbFt}, p. 1929, the idea of such decompositions for $\mathcal{H}_{\mathrm{loc}}^{(q,p)}$ spaces has been mentioned without more details. In the sequel, unless otherwise specified, we assume that $0<q\leq 1$ and $q\leq p<+\infty$. Also, for simplicity, we adopt the following notations: $$\mathcal{M}_{\mathrm{loc}}:=\mathcal{M}_{{\mathrm{loc}}_{_{\mathcal{F}_N}}}\ ,\ \mathcal{M}_{\mathrm{loc}}^0:=\mathcal{M}_{{\mathrm{loc}}_{_{\mathcal{F}_N}}}^0\ \text{ and }\ \mathcal{M}_{{\mathrm{loc}}_{_{0}}}:={\mathcal{M}_{\mathrm{loc}}}_{_{\varphi}}.$$

\subsection{Atomic Decomposition} 

Our definition of atom on $\mathbb{R}^d$ for $\mathcal{H}_{\mathrm{loc}}^{(q,p)}$ is the one of $\mathcal{H}^{(q,p)}$ except that, when $|Q|\geq1$, the vanishing condition (\ref{defratom2}) in Definition \ref{defhqpatom} is not requiered. Likewise, we denote by $\mathcal{A}_{\mathrm{loc}}(q,r,s)$ the set of all $(\textbf{a},Q)$ such that $\textbf{a}$ is a $(q,r,s)$-atom on $\mathbb{R}^d$ for $\mathcal{H}_{\mathrm{loc}}^{(q,p)}$ and $Q$ is the associated cube. It follows from this definition the following remark. 

\begin{remark}\label{fondamentalremarloc}
Let $1<r\leq v\leq+\infty$. Then,
\begin{enumerate}
\item $(\textbf{a},Q)\in\mathcal{A}_{\mathrm{loc}}(q,v,s)\ \Rightarrow\ (\textbf{a},Q)\in\mathcal{A}_{\mathrm{loc}}(q,r,s)$. 
\item $\mathcal{A}(q,r,s)\subset\mathcal{A}_{\mathrm{loc}}(q,r,s)$.
\end{enumerate}
\end{remark}

Also, it is easy to see that, for all $(\textbf{a},Q)\in\mathcal{A}_{\mathrm{loc}}(q,r,s)$, 
\begin{eqnarray}
\left\|\textbf{a}\right\|_{\mathcal{H}_{\mathrm{loc}}^{(q,p)}}\leq C, \label{locfondamentalpropo}
\end{eqnarray}
where $C>0$ is a constant independent of the $(q,r,s)$-atom $\textbf{a}$. 

We have the following reconstruction theorems.    

\begin{thm} \label{thafondaloc} 
Let $0<\eta\leq 1$ and $\delta\geq\left\lfloor d\left(\frac{1}{q}-1\right)\right\rfloor$ be an integer. Then, for all sequences $\left\{(\textbf{a}_n, Q^n)\right\}_{n\geq 0}$ in $\mathcal{A}_{\mathrm{loc}}(q,\infty,\delta)$ and all sequences of scalars $\left\{{\lambda}_n\right\}_{n\geq 0}$ such that 
\begin{eqnarray}
\left\|\sum_{n\geq 0}\left(\frac{|\lambda_n|}{\left\|\chi_{_{Q^n}}\right\|_{q}}\right)^{\eta}\chi_{_{Q^{n}}}\right\|_{\frac{q}{\eta},\frac{p}{\eta}}<+\infty, \label{inversetheoloc}
\end{eqnarray}
the series $f:=\sum_{n\geq 0}{\lambda}_n\textbf{a}_n$ converges in $\mathcal{S'}$ and $\mathcal{H}_{\mathrm{loc}}^{(q,p)}$, with 
\begin{eqnarray*}
\left\|f\right\|_{\mathcal{H}_{\mathrm{loc}}^{(q,p)}}\lsim\left\|\sum_{n\geq 0}\left(\frac{|\lambda_n|}{\left\|\chi_{_{Q^n}}\right\|_{q}}\right)^{\eta}\chi_{_{Q^{n}}}\right\|_{\frac{q}{\eta},\frac{p}{\eta}}^{\frac{1}{\eta}}.
\end{eqnarray*}

\end{thm}
\begin{proof}
For any $(\textbf{a}_n, Q^n)\in\mathcal{A}_{\mathrm{loc}}(q,\infty,\delta)$, we have 
\begin{eqnarray}
\mathcal{M}_{{\mathrm{loc}}_{_{0}}}(\textbf{a}_n)(x)\lsim\left[\mathfrak{M}(\chi_{_{Q^n}})(x)\right]^{\frac{d+\delta+1}{d}}\left\|\chi_{_{Q^{n}}}\right\|_q^{-1}\ , \label{inversetheo6loc}
\end{eqnarray}
for all $x\in\mathbb{R}^d$. For the proof of (\ref{inversetheo6loc}), we distinguish two cases: $|Q^{n}|<1$ and $|Q^{n}|\geq1$. Denote by $x_n$ and $\ell_n$ respectively the center and the side-lenght of $Q^{n}$. When $|Q^{n}|<1$, we know that the atom $\textbf{a}_n$ satisfies the vanishing condition. Thus, setting $\widetilde{Q^{n}}:=2\sqrt{d}Q^n$, we obtain, as in the proof of \cite{AbFt}, Theorem 4.3, pp. 1914-1915,  
\begin{eqnarray*}
\mathcal{M}_{{\mathrm{loc}}_{_{0}}}(\textbf{a}_n)(x)\lsim\left[\mathfrak{M}(\chi_{_{Q^n}})(x)\right]^{\frac{d+\delta+1}{d}}\left\|\chi_{_{Q^{n}}}\right\|_q^{-1}\ , 
\end{eqnarray*}
for all $x\in\mathbb{R}^d$. When $|Q^{n}|\geq1$, the vanishing condition of the atom $\textbf{a}_n$ not being requiered, we consider $\widetilde{Q^{n}}:=4\sqrt{d}Q^n$. Then, it is straigthforward to see that 
\begin{eqnarray*}
\mathcal{M}_{{\mathrm{loc}}_{_{0}}}(\textbf{a}_n)(x)=0,
\end{eqnarray*}
for all $x\notin\widetilde{Q^{n}}$. Thus, 
\begin{eqnarray*}
\mathcal{M}_{{\mathrm{loc}}_{_{0}}}(\textbf{a}_n)(x)&=&\mathcal{M}_{{\mathrm{loc}}_{_{0}}}(\textbf{a}_n)(x)\chi_{_{\widetilde{Q^{n}}}}(x)\\
&\lsim&\mathfrak{M}(\textbf{a}_n)(x)\chi_{_{\widetilde{Q^{n}}}}(x)\\
&\lsim&\left\|\textbf{a}_n\right\|_{\infty}\chi_{_{\widetilde{Q^{n}}}}(x)\\
&\lsim&\frac{\ell_n^{d+\delta+1}}{\ell_n^{d+\delta+1}+|x-x_n|^{d+\delta+1}}\left\|\chi_{_{Q^{n}}}\right\|_q^{-1}\\
&\lsim&\left[\mathfrak{M}\chi_{_{Q^{n}}}(x)\right]^{\frac{d+\delta+1}{d}}\left\|\chi_{_{Q^{n}}}\right\|_q^{-1},
\end{eqnarray*}
for all $x\in\mathbb{R}^d$. Combining these two cases, we obtain (\ref{inversetheo6loc}).  

With (\ref{inversetheo6loc}), we end as in the proof of \cite{AbFt}, Theorem 4.3.  
\end{proof}

\begin{thm} \label{thafondamloc}
Let $\max\left\{p,1\right\}<r<+\infty$, $0<\eta<q$ and $\delta\geq\left\lfloor d\left(\frac{1}{q}-1\right)\right\rfloor$ be an integer. Then, for all sequences $\left\{(\textbf{a}_n, Q^n)\right\}_{n\geq 0}$ in $\mathcal{A}_{\mathrm{loc}}(q,r,\delta)$ and all sequences of scalars $\left\{{\lambda}_n\right\}_{n\geq 0}$ such that 
\begin{eqnarray}
\left\|\sum_{n\geq 0}\left(\frac{|\lambda_n|}{\left\|\chi_{_{Q^n}}\right\|_{q}}\right)^{\eta}\chi_{_{Q^{n}}}\right\|_{\frac{q}{\eta},\frac{p}{\eta}}<+\infty, \label{inversetheor2loc}
\end{eqnarray}
the series $f:=\sum_{n\geq 0}{\lambda}_n\textbf{a}_n$ converges in $\mathcal{S'}$ and $\mathcal{H}_{\mathrm{loc}}^{(q,p)}$, with 
\begin{eqnarray*}
\left\|f\right\|_{\mathcal{H}_{\mathrm{loc}}^{(q,p)}}\lsim\left\|\sum_{n\geq 0}\left(\frac{|\lambda_n|}{\left\|\chi_{_{Q^n}}\right\|_{q}}\right)^{\eta}\chi_{_{Q^{n}}}\right\|_{\frac{q}{\eta},\frac{p}{\eta}}^{\frac{1}{\eta}}.
\end{eqnarray*}
\end{thm}
\begin{proof} 
For any $(\textbf{a}_n, Q^n)\in\mathcal{A}_{\mathrm{loc}}(q,r,\delta)$, we have
\begin{eqnarray}
\mathcal{M}_{{\mathrm{loc}}_{_{0}}}(\textbf{a}_n)(x)\lsim\mathfrak{M}(\textbf{a}_n)(x)\chi_{_{\widetilde{Q^n}}}(x)+\frac{\left[\mathfrak{M}(\chi_{_{Q^{n}}})(x)\right]^{\mu}}{\left\|\chi_{_{Q^{n}}}\right\|_q}\ , \label{inversetheor5loc}
\end{eqnarray}
for all $x\in\mathbb{R}^d$, where $\widetilde{Q^n}:=4\sqrt{d}Q^n$ and $\mu=\frac{d+\delta+1}{d}\cdot$ To prove (\ref{inversetheor5loc}), denote by $x_n$ and $\ell_n$ respectively the center and the side-lenght of $Q^n$. It's clear that 
\begin{eqnarray}
\mathcal{M}_{{\mathrm{loc}}_{_{0}}}(\textbf{a}_n)(x)\lsim
\mathfrak{M}(\textbf{a}_n)(x), \label{inversetheor3loc}
\end{eqnarray}
for all $x\in\widetilde{Q^n}$. Also, 
\begin{eqnarray}
\mathcal{M}_{{\mathrm{loc}}_{_{0}}}(\textbf{a}_n)(x)\lsim
\left[\mathfrak{M}(\chi_{_{Q^{n}}})(x)\right]^{\frac{d+\delta+1}{d}}\left\|\chi_{_{Q^{n}}}\right\|_q^{-1}, \label{inversetheor4loc}
\end{eqnarray}
for all $x\notin\widetilde{Q^n}$. Endeed, when $|Q^n|<1$, by proceeding as in the proof of \cite{AbFt}, Theorem 4.3, pp. 1914-1915, we obtain (\ref{inversetheor4loc}). And when $|Q^n|\geq1$, (\ref{inversetheor4loc}) is clearly  satisfied since
\begin{eqnarray*}
\mathcal{M}_{{\mathrm{loc}}_{_{0}}}(\textbf{a}_n)(x)=0,
\end{eqnarray*}
for all $x\notin\widetilde{Q^{n}}$. Thus, combining (\ref{inversetheor3loc}) and (\ref{inversetheor4loc}), we obtain (\ref{inversetheor5loc}).

With (\ref{inversetheor5loc}), we end as in the proof of \cite{AbFt}, Theorem 4.6, p. 1921.  
\end{proof}

For our decomposition theorem, we need some lemmas. Also, we borrow some ideas from \cite{LT} and \cite{YDYS}. But before, we have the useful following remark.

\begin{remark} \label{remarqdenseyloc}
In our earlier paper \cite{AbFt}, Lemma 4.1 has been very useful to establish the decomposition theorem  of $\mathcal{H}^{(q,p)}$ spaces (\cite{AbFt}, Theorem 4.4, p. 1916). Also, in the same paper (see p. 1929), we noticed that, in Lemma 4.1, the maximal operator $\mathcal{M}^0$ can be replaced by its local version $\mathcal{M}_{\mathrm{loc}}^0$, more precisely with $b_k=(f-c_k)\eta_k$, if $|Q_\ast^k|<1$, and $b_k=f\eta_k$, if $|Q_\ast^k|\geq1$, for each integer $k\geq0$. That allowed us to show that $L_{\mathrm{loc}}^1\cap\mathcal{H}_{\mathrm{loc}}^{(q,p)}$ is dense in $\mathcal{H}_{\mathrm{loc}}^{(q,p)}$ in the quasi-norm $\left\|\cdot\right\|_{\mathcal{H}_{\mathrm{loc}}^{(q,p)}}$ (see \cite{AbFt}, Proposition 4.11, p. 1929). However, we discovered a mistake in the proof of \cite{AbFt}, Proposition 4.11, p. 1929. Here, we correct this error. 
\end{remark}

\begin{prop}\label{denseyloc} Suppose that $0<q\leq1$ and $0<p<+\infty$. Let $p\leq s\leq+\infty$ and $0<r\leq+\infty$. Then, $L_{\mathrm{loc}}^1\cap\mathcal{H}_{\mathrm{loc}}^{(q,p)}$, $(L^q,\ell^s)\cap\mathcal{H}_{\mathrm{loc}}^{(q,p)}$ and $(L^r,\ell^{\infty})\cap\mathcal{H}_{\mathrm{loc}}^{(q,p)}$ are dense subspaces of $\mathcal{H}_{\mathrm{loc}}^{(q,p)}$. 
\end{prop}
\begin{proof}  
Let $f\in\mathcal{H}_{\mathrm{loc}}^{(q,p)}$ and $\delta\geq\max\left\{\left\lfloor d\left(\frac{1}{q}-1\right)\right\rfloor,\left\lfloor d\left(\frac{1}{p}-1\right)\right\rfloor\right\}$ be an integer. For each $j\in\mathbb{Z}_{+}$, we consider the level set $$\mathcal{O}^j:=\left\{x\in\mathbb{R}^d:\ \mathcal{M}_{\mathrm{loc}}^0(f)(x)>2^j\right\}.$$ By Remark \ref{remarqdenseyloc}, $f$ can be decomposed as follows: $f=g_j+b_j$ with $b_j=\sum_{n\geq 0}b_{j,n}$ and $$b_{j,n}=\left\{\begin{array}{lll}(f-c_{j,n})\eta_{_{j,n}},&\text{ if }&|Q_{\ast}^{j,n}|<1\\ f\eta_{_{j,n}},&\text{ if }&|Q_{\ast}^{j,n}|\geq1,\end{array}\right.$$ where each $\eta_{j,n}$ is supported in $Q_{\ast}^{j,n}$, with 
\begin{eqnarray}
\bigcup_{n\geq0}Q_{\ast}^{j,n}=\mathcal{O}^j, \sum_{n\geq 0}\chi_{Q_{\ast}^{j,n}}\lsim 1; \sum_{n\geq 0}\eta_{_{j,n}}=\chi_{\mathcal{O}^j}; \ 0\leq\eta_{_{j,n}}\leq1; \label{caldzygm1loc0}
\end{eqnarray}
\begin{equation}
\mathcal{M}_{\mathrm{loc}}^0(b_{j,n})(x)\lsim\mathcal{M}_{\mathrm{loc}}^0(f)(x)\chi_{_{Q_{\ast}^{j,n}}}(x)+\frac{2^j \ell_{j,n}^{d+\delta+1}}{|x-x_{j,n}|^{d+\delta+1}}\chi_{_{\mathbb{R}^d\backslash{Q_{\ast}^{j,n}}}}(x) \label{caldzygm2loc}
\end{equation}
and
\begin{equation}
\mathcal{M}_{\mathrm{loc}}^0(g_j)(x)\lsim\mathcal{M}_{\mathrm{loc}}^0(f)(x)\chi_{{\mathbb{R}^d\backslash{\mathcal{O}^j}}}(x)+2^j\sum_{n\geq 0}\frac{\ell_{j,n}^{d+\delta+1}}{(\ell_{j,n}+ |x-x_{j,n}|)^{d+\delta+1}}\ , \label{caldzygm1loc}
\end{equation}
for all $x\in\mathbb{R}^d$, where $x_{j,n}$ and $\ell_{j,n}$ denote respectively the center and the side-length of $Q_{\ast}^{j,n}$ and $\mathbb{R}^d\backslash{\mathcal{O}^j}=\left\{x\in\mathbb{R}^d:\ \mathcal{M}_{\mathrm{loc}}^0(f)(x)\leq 2^j\right\}$. Set $u=\frac{d+\delta+1}{d}\cdot$ We have $1<u$ and $1<\min\left\{qu,pu\right\}$, since $\delta\geq\max\left\{\left\lfloor d\left(\frac{1}{q}-1\right)\right\rfloor,\left\lfloor d\left(\frac{1}{p}-1\right)\right\rfloor\right\}\cdot$ Thus, with (\ref{caldzygm2loc}) and (\ref{caldzygm1loc0}), by arguing as in the proof of \cite{AbFt}, Theorem 4.4, p. 1916, we obtain $$\left\|b_j\right\|_{\mathcal{H}_{\mathrm{loc}}^{(q,p)}}\lsim\left\|\chi_{_{\mathcal{O}^j}}\mathcal{M}_{\mathrm{loc}}^0(f)\right\|_{q,p},$$ for all $j\in\mathbb{Z}_{+}$, so that 
\begin{eqnarray}
g_j=f-b_j\in\mathcal{H}_{\mathrm{loc}}^{(q,p)}, \label{caldzygm3loc}
\end{eqnarray}
for all $j\in\mathbb{Z}_{+}$ and $\left\|f-g_j\right\|_{\mathcal{H}_{\mathrm{loc}}^{(q,p)}}=\left\|b_j\right\|_{\mathcal{H}_{\mathrm{loc}}^{(q,p)}}\lsim\left\|\chi_{_{\mathcal{O}^j}}\mathcal{M}_{\mathrm{loc}}^0(f)\right\|_{q,p}\rightarrow 0$ as $j\rightarrow+\infty$. Hence the sequence $\left\{g_j\right\}_{j\geq 0}\subset\mathcal{H}_{\mathrm{loc}}^{(q,p)}$ converges to $f$ in $\mathcal{H}_{\mathrm{loc}}^{(q,p)}$. Moreover,
\begin{eqnarray*}
\left\|\mathcal{M}_{\mathrm{loc}}^0(g_j)\right\|_{1,\frac{p}{q}}&\lsim&\left\|\mathcal{M}_{\mathrm{loc}}^0(f)\chi_{{\mathbb{R}^d\backslash{\mathcal{O}^j}}}+2^j\sum_{n\geq 0}\left[\mathfrak{M}(\chi_{_{Q_{\ast}^{j,n}}})\right]^{\frac{d+\delta+1}{d}}\right\|_{1,\frac{p}{q}}\\
&\lsim&\left\|\mathcal{M}_{\mathrm{loc}}^0(f)\chi_{{\mathbb{R}^d\backslash{\mathcal{O}^j}}}\right\|_{1,\frac{p}{q}}
+2^j\left\|\sum_{n\geq 0}\left[\mathfrak{M}(\chi_{_{Q_{\ast}^{j,n}}})\right]^{\frac{d+\delta+1}{d}}\right\|_{1,\frac{p}{q}},
\end{eqnarray*}
for all $j\in\mathbb{Z}_{+}$, by (\ref{caldzygm1loc}). But, since $0<q\leq1$, we have
\begin{eqnarray*}
\left\|\mathcal{M}_{\mathrm{loc}}^0(f)\chi_{{\mathbb{R}^d\backslash{\mathcal{O}^j}}}\right\|_{1,\frac{p}{q}}\leq 2^{j(1-q)}\left\|\mathcal{M}_{\mathrm{loc}}^0(f)\right\|_{q,p}^q\approx 
2^{j(1-q)}\left\|f\right\|_{\mathcal{H}_{\mathrm{loc}}^{(q,p)}}^q.
\end{eqnarray*}
Also, since $0<q\leq1<pu$, we have $\frac{p}{q}u>1$. Thus, by Proposition \ref{operamaxima} and (\ref{caldzygm1loc0}),   
\begin{eqnarray*}
\left\|\sum_{n\geq 0}\left[\mathfrak{M}(\chi_{_{Q_{\ast}^{j,n}}})\right]^{\frac{d+\delta+1}{d}}\right\|_{1,\frac{p}{q}}&=&\left\|\left(\sum_{n\geq 0}\left[\mathfrak{M}(\chi_{_{Q_{\ast}^{j,n}}})\right]^u\right)^{\frac{1}{u}}\right\|_{u,\frac{p}{q}u}^u\\
&\lsim&\left\|\sum_{n\geq 0}\chi_{_{Q_{\ast}^{j,n}}}\right\|_{1,\frac{p}{q}}\lsim\left\|\chi_{_{\mathcal{O}^j}}\right\|_{1,\frac{p}{q}}
\end{eqnarray*}
and 
\begin{eqnarray*}
\left\|\chi_{_{\mathcal{O}^j}}\right\|_{1,\frac{p}{q}}=\left[\sum_{k\in\mathbb{Z}^d}\left(\int_{Q_k}\chi_{_{\mathcal{O}^j}}(x)dx\right)^{\frac{p}{q}}\right]^{\frac{q}{p}}&\leq&\left[2^{-jp}\left\|\mathcal{M}_{\mathrm{loc}}^0(f)\right\|_{q,p}^p\right]^{\frac{q}{p}}\\
&\approx&2^{-jq}\left\|f\right\|_{\mathcal{H}_{\mathrm{loc}}^{(q,p)}}^q.
\end{eqnarray*}
Hence
\begin{eqnarray*}
\left\|\sum_{n\geq 0}\left[\mathfrak{M}(\chi_{_{Q_{\ast}^{j,n}}})\right]^{\frac{d+\delta+1}{d}}\right\|_{1,\frac{p}{q}}\lsim 2^{-jq}\left\|f\right\|_{\mathcal{H}_{\mathrm{loc}}^{(q,p)}}^q.
\end{eqnarray*}
Thus,
\begin{eqnarray*}
\left\|\mathcal{M}_{\mathrm{loc}}^0(g_j)\right\|_{1,\frac{p}{q}}\lsim 2^{j(1-q)}\left\|f\right\|_{\mathcal{H}_{\mathrm{loc}}^{(q,p)}}^q<+\infty,
\end{eqnarray*}
for all $j\in\mathbb{Z}_{+}$. Therefore, for all $j\in\mathbb{Z}_{+}$, 
\begin{eqnarray}
g_j\in\mathcal{H}_{\mathrm{loc}}^{(1,\frac{p}{q})}\subset L_{\mathrm{loc}}^1,\label{caldzygm4loc}
\end{eqnarray}
by Corollary \ref{bravoooheiner}. Combining (\ref{caldzygm3loc}) and (\ref{caldzygm4loc}), we obtain 
\begin{eqnarray}
g_j\in L_{\mathrm{loc}}^1\cap\mathcal{H}_{\mathrm{loc}}^{(q,p)},
\end{eqnarray}
for all $j\in\mathbb{Z}_{+}$. Thus, $L_{\mathrm{loc}}^1\cap\mathcal{H}_{\mathrm{loc}}^{(q,p)}$ is dense in $\mathcal{H}_{\mathrm{loc}}^{(q,p)}$. 

Also, since $g_j\in L_{\mathrm{loc}}^1$ for all $j\in\mathbb{Z}_{+}$, we have $|g_j(x)|\leq\mathcal{M}_{{\mathrm{loc}}_{_{0}}}(g_j)(x)$, for almost all $x\in\mathbb{R}^d$ and all $j\in\mathbb{Z}_{+}$. It follows from this estimate that  
\begin{eqnarray*}
\left\|g_j\right\|_{q,s}\leq\left\|\mathcal{M}_{{\mathrm{loc}}_{_{0}}}(g_j)\right\|_{q,s}\leq\left\|\mathcal{M}_{{\mathrm{loc}}_{_{0}}}(g_j)\right\|_{q,p}=\left\|g_j\right\|_{\mathcal{H}_{\mathrm{loc}}^{(q,p)}}<+\infty,
\end{eqnarray*}
 all $j\in\mathbb{Z}_{+}$. Hence $g_j\in(L^q,\ell^s)\cap\mathcal{H}_{\mathrm{loc}}^{(q,p)}$, for all $j\in\mathbb{Z}_{+}$. Thus, $(L^q,\ell^s)\cap\mathcal{H}_{\mathrm{loc}}^{(q,p)}$ is dense in $\mathcal{H}_{\mathrm{loc}}^{(q,p)}$.

For the proof of the density of $(L^r,\ell^{\infty})\cap\mathcal{H}_{\mathrm{loc}}^{(q,p)}$ in $\mathcal{H}_{\mathrm{loc}}^{(q,p)}$, our approach in the proof of \cite{AbFt}, Proposition 4.11 is not correct, since the estimate $\mathcal{M}_{{\mathrm{loc}}_{_{0}}}(g_j)(x)\lsim 2^j$, for almost all $x\in\mathbb{R}^d$, used in \cite{AbFt} is not valid when the tempered distribution $f\notin L_{\mathrm{loc}}^1$. Here, we close this gap by showing that for any $\epsilon>0$, there exists $f_\epsilon\in(L^r,\ell^{\infty})\cap\mathcal{H}_{\mathrm{loc}}^{(q,p)}$ such that  
\begin{eqnarray}
\left\|f-f_\epsilon\right\|_{\mathcal{H}_{\mathrm{loc}}^{(q,p)}}\leq\epsilon, \label{1correctif01}
\end{eqnarray}
which implies that $(L^r,\ell^{\infty})\cap\mathcal{H}_{\mathrm{loc}}^{(q,p)}$ is dense in $\mathcal{H}_{\mathrm{loc}}^{(q,p)}$. For the proof of (\ref{1correctif01}), we distinguish two cases. 

Case 1: Assume that $f\in L_{\mathrm{loc}}^1$. Then, 
\begin{eqnarray}
|g_j(x)|\leq\mathcal{M}_{{\mathrm{loc}}_{_{0}}}(g_j)(x)\lsim 2^j
, \label{1correctif2}
\end{eqnarray}
for almost all $x\in\mathbb{R}^d$ and all $j\in\mathbb{Z}_{+}$. Hence $g_j\in L^{\infty}\subset(L^r,\ell^{\infty})$, for all $j\in\mathbb{Z}_{+}$. Thus, $g_j\in(L^r,\ell^{\infty})\cap\mathcal{H}_{\mathrm{loc}}^{(q,p)}$, for all $j\in\mathbb{Z}_{+}$. Furthermore, by the above results, we know that $\lim_{j\rightarrow+\infty}\left\|f-g_j\right\|_{\mathcal{H}_{\mathrm{loc}}^{(q,p)}}=0.$ Hence for any $\epsilon>0$, there exists an integer $j_\epsilon\geq0$ such that $g_{j_\epsilon}\in(L^r,\ell^{\infty})\cap\mathcal{H}_{\mathrm{loc}}^{(q,p)}$ and
\begin{eqnarray*}
\left\|f-g_{j_\epsilon}\right\|_{\mathcal{H}_{\mathrm{loc}}^{(q,p)}}\leq\epsilon.
\end{eqnarray*}

Case 2: Assume that $f\notin L_{\mathrm{loc}}^1$. Then, we have no longer $g_j\in L^{\infty}$, since (\ref{1correctif2}) does not hold. However, by the above results, we know that $g_j\in L_{\mathrm{loc}}^1\cap\mathcal{H}_{\mathrm{loc}}^{(q,p)}$, for all $j\in\mathbb{Z}_{+}$, and 
\begin{eqnarray}
\lim_{j\rightarrow+\infty}\left\|f-g_j\right\|_{\mathcal{H}_{\mathrm{loc}}^{(q,p)}}=0.\label{1correctif3} 
\end{eqnarray}
Since $g_j\in L_{\mathrm{loc}}^1\cap\mathcal{H}_{\mathrm{loc}}^{(q,p)}$, we can apply the preceeding argument to each $g_j$. Thus, for each $g_j$, there exists a sequence $\left\{g_k^j\right\}_{k\geq0}\subset(L^r,\ell^{\infty})\cap\mathcal{H}_{\mathrm{loc}}^{(q,p)}$ with
\begin{eqnarray}
\lim_{k\rightarrow+\infty}\left\|g_j-g_k^j\right\|_{\mathcal{H}_{\mathrm{loc}}^{(q,p)}}=0. \label{1correctif4}
\end{eqnarray}
Let $\epsilon>0$ and $C_{q,p}\geq1$ the modulus of concavity of the quasi-norm $\left\|\cdot\right\|_{\mathcal{H}_{\mathrm{loc}}^{(q,p)}}$. Then, by (\ref{1correctif3}), there exists an integer $j_\epsilon\geq0$ such that $\left\|f-g_{j_\epsilon}\right\|_{\mathcal{H}_{\mathrm{loc}}^{(q,p)}}\leq\frac{\epsilon}{2C_{q,p}}\cdot$ Also, by (\ref{1correctif4}), there exists an integer $k_{j_\epsilon}\geq0$ such that $g_{k_{j_\epsilon}}^{j_\epsilon}\in(L^r,\ell^{\infty})\cap\mathcal{H}_{\mathrm{loc}}^{(q,p)}$ and $\left\|g_{j_\epsilon}-g_{k_{j_\epsilon}}^{j_\epsilon}\right\|_{\mathcal{H}_{\mathrm{loc}}^{(q,p)}}\leq\frac{\epsilon}{2C_{q,p}}\cdot$ Thus, $g_{k_{j_\epsilon}}^{j_\epsilon}\in(L^r,\ell^{\infty})\cap\mathcal{H}_{\mathrm{loc}}^{(q,p)}$ and
\begin{eqnarray*}
\left\|f-g_{k_{j_\epsilon}}^{j_\epsilon}\right\|_{\mathcal{H}_{\mathrm{loc}}^{(q,p)}}\leq C_{q,p}\left(\left\|f-g_{j_\epsilon}\right\|_{\mathcal{H}_{\mathrm{loc}}^{(q,p)}}+\left\|g_{j_\epsilon}-g_{k_{j_\epsilon}}^{j_\epsilon}\right\|_{\mathcal{H}_{\mathrm{loc}}^{(q,p)}}\right)\leq\epsilon.
\end{eqnarray*}

Combining the cases 1 and 2, we obtain (\ref{1correctif01}). This completes the proof of Proposition \ref{denseyloc}.
\end{proof}

Now, consider $f\in\mathcal{H}_{\mathrm{loc}}^{(q,p)}$ and an integer $\delta\geq0$. For each $j\in\mathbb{Z}$, we set $\mathcal{O}^j:=\left\{x\in\mathbb{R}^d:\ \mathcal{M}_{\mathrm{loc}}^0(f)(x)>2^j\right\}$. By Remark \ref{remarqdenseyloc}, $f$ can be decomposed as follows: $f=g_j+b_j$, with $b_j=\sum_{k\geq 0} b_{j,k}$ , $b_{j,k}=(f-c_{j,k})\eta_{_{j,k}}$ , if $|Q_{\ast}^{j,k}|<1$, and $b_{j,k}=f\eta_{_{j,k}}$ , if $|Q_{\ast}^{j,k}|\geq1$, where each function $\eta_{_{j,k}}$ is supported on $Q_{\ast}^{j,k}$. Also, $c_{j,k}\in\mathcal{P_{\delta}}:=\mathcal{P_{\delta}}(\mathbb{R}^d)$ (the space of polynomial functions of degree at most $\delta$) satisfies  
\begin{eqnarray}
\int_{\mathbb{R}^d}[f(x)-c_{j,k}(x)]\eta_{j,k}(x)\mathfrak{p}(x)dx=0, \label{operamaxi40re1}
\end{eqnarray}
for all $\mathfrak{p}\in\mathcal{P_{\delta}}$; namely $c_{j,k}=P_{k}^{j}(f)$, where $P_{k}^{j}$ is the projection operator from $\mathcal{S'}$ onto $\mathcal{P_{\delta}}$ with respect to the norm 
\begin{eqnarray}
\left\|\mathfrak{p}\right\|=\left(\int_{\mathbb{R}^d}|\mathfrak{p}(x)|^2\tilde{\eta}_{j,k}(x)dx\right)^{1/2},\ \ \mathfrak{p}\in\mathcal{P_{\delta}}, \label{normeprecisbis}
\end{eqnarray}
with $\tilde{\eta}_{j,k}=\frac{\eta_{j,k}}{\int_{Q_{\ast}^{j,k}}\eta_{j,k}(x)dx}\cdot$ (\ref{operamaxi40re1}) is understood as $$\left\langle f\ ,\  \eta_{j,k}\mathfrak{p}\right\rangle=\int_{\mathbb{R}^d}c_{j,k}(x)\eta_{j,k}(x)\mathfrak{p}(x)dx.$$ We denote by $c_{k}^{\ell}$ the unique polynomial of $\mathcal{P_{\delta}}$ which satisfies
\begin{eqnarray}
\left\langle \left(f-c_{j+1,\ell}\right)\eta_{j,k}\ ,\ \mathfrak{p}\eta_{j+1,\ell}\right\rangle=\int_{\mathbb{R}^d}c_{k}^{\ell}(x)\mathfrak{p}(x)\eta_{j+1,\ell}(x)dx, \label{operamaxi40re3}
\end{eqnarray}
for all $\mathfrak{p}\in\mathcal{P_{\delta}}$; namely $c_{k}^{\ell}=P_{\ell}^{j+1}\left((f-c_{j+1,\ell})\eta_{j,k}\right)$, where $P_{\ell}^{j+1}$ is the projection operator from $\mathcal{S'}$ onto $\mathcal{P_{\delta}}$ with respect to the norm  
\begin{eqnarray}
\left\|\mathfrak{p}\right\|=\left(\int_{\mathbb{R}^d}|\mathfrak{p}(x)|^2\tilde{\eta}_{j+1,\ell}(x)dx\right)^{1/2},\ \ \mathfrak{p}\in\mathcal{P_{\delta}}. \label{normeprecis}
\end{eqnarray} 
We will often write (\ref{operamaxi40re3}) as follows:
\begin{eqnarray*}
\int_{\mathbb{R}^d}\left(f(x)-c_{j+1,\ell}(x)\right)\eta_{j,k}(x)\mathfrak{p}(x)\eta_{j+1,\ell}(x)dx=\int_{\mathbb{R}^d}c_{k}^{\ell}(x)\mathfrak{p}(x)\eta_{j+1,\ell}(x)dx.
\end{eqnarray*} 
We have
\begin{eqnarray}
c_{j,k}=\sum_{n=1}^m\left(\int_{\mathbb{R}^d}f(x)e_n(x)\tilde{\eta}_{j,k}(x)dx\right)\overline{e_n} \label{operamaxi40re31}
\end{eqnarray}
and
\begin{equation}
c_{k}^{\ell}=\sum_{n=1}^m\left(\int_{\mathbb{R}^d}\left(f(x)-c_{j+1,\ell}(x)\right)\eta_{j,k}(x)\pi_n(x)\tilde{\eta}_{j+1,\ell}(x)dx\right)\overline{\pi_n}, \label{operamaxi40re30}
\end{equation}
where $\left\{e_1,\ldots,e_m\right\}$ and $\left\{\pi_1,\ldots,\pi_m\right\}$ are orthonormal bases of $\mathcal{P}_{\delta}$ respectively with respect to the norms (\ref{normeprecisbis}) and (\ref{normeprecis}) with $m=\text{dim}\mathcal{P}_{\delta}$, and $\overline{\mathfrak{p}}$ stands for the conjugate of the polynomial $\mathfrak{p}\in\mathcal{P}_{\delta}$. The integrals in (\ref{operamaxi40re31}) and (\ref{operamaxi40re30}) are understood respectively as $\left\langle f\ ,\ e_n\tilde{\eta}_{j,k}\right\rangle$ and $\left\langle\left(f-c_{j+1,\ell}\right)\eta_{j,k}\ ,\ \pi_n\tilde{\eta}_{j+1,\ell}\right\rangle$. We have
\begin{eqnarray}
Q_{\ast}^{j,k}\cap Q_{\ast}^{j+1,\ell}=\emptyset\ \Rightarrow\ c_{k}^{\ell}=0, \label{operamaxi5}
\end{eqnarray}
by (\ref{operamaxi40re30}). We have the following lemma whose proof is identical to those of \cite{AbFt}, (4.9), (4.10) and (4.11), p. 1918.

\begin{lem}\label{lemdecompoat3}\
\begin{enumerate}
\item If $Q_{\ast}^{j,k}\cap Q_{\ast}^{j+1,\ell}\neq\emptyset$, then $\text{diam}(Q_{\ast}^{j+1,\ell})\leq C\text{diam}(Q_{\ast}^{j,k})$ and $Q_{\ast}^{j+1,\ell}\subset C_0Q_{\ast}^{j,k}$ with $1<C<C_0$ constants independent of $f$, $j$; $k$ and $\ell$. \label{operamaxi6}
\item $\ \sharp\left\{k\in\mathbb{Z}_{+}:\ Q_{\ast}^{j,k}\cap Q_{\ast}^{j+1,\ell}\neq\emptyset\right\}\lsim 1$, for each $j\in\mathbb{Z}$ and $\ell\in\mathbb{Z}_{+}$. \label{operamaxi9}
\item Let $j\in\mathbb{Z}$ and $k\in\mathbb{Z}_{+}$. For any real $\lambda\geq 9d$, $\ \lambda Q_{\ast}^{j,k}\cap(\mathbb{R}^d\backslash\mathcal{O}^j)\neq\emptyset$. \label{operamaxi8bis}
\end{enumerate}
\end{lem}

Set $E_1^j:=\left\{k\in\mathbb{Z}_{+}:\ |Q_{\ast}^{j,k}|<1\right\}$ and $E_2^j:=\left\{k\in\mathbb{Z}_{+}:\ |Q_{\ast}^{j,k}|\geq1\right\}$, for each $j\in\mathbb{Z}$. Also, $F_1^j:=\left\{k\in\mathbb{Z}_{+}:\ |Q_{\ast}^{j,k}|<1/{C_0^d}\right\}$ and $F_2^j:=\left\{k\in\mathbb{Z}_{+}:\ |Q_{\ast}^{j,k}|\geq1/{C_0^d}\right\}$, for each $j\in\mathbb{Z}$, where $C_{0}>1$ is the constant $C_{0}$ in Lemma \ref{lemdecompoat3} (\ref{operamaxi6}). We have the following lemmas whose proofs are similar to those of \cite{AbFt}, (4.12), (4.13), p. 1918 and \cite{SZLU}, Lemma 3.7. We omit details.

\begin{lem}\label{lemdecompoat4loc} Let $j\in\mathbb{Z}$ and $k,\ell\in\mathbb{Z}_{+}$. 
If $\ell^{j,k}:=\ell(Q_{\ast}^{j,k})<1$, then
\begin{eqnarray}
\sup_{x\in\mathbb{R}^d}|c_{j,k}(x)\eta_{j,k}(x)|\lsim\sup_{y\in 9d Q_{\ast}^{j,k}\cap(\mathbb{R}^d\backslash\mathcal{O}^j)}\mathcal{M}_{\mathrm{loc}}^{0}(f)(y)\lsim 2^j.\label{operamaxi8loc}
\end{eqnarray}
If $\ell^{j+1,\ell}:=\ell(Q_{\ast}^{j+1,\ell})<1$, then
\begin{equation}
\sup_{x\in\mathbb{R}^d}|c_{k}^{\ell}(x)\eta_{j+1,\ell}(x)|\lsim\sup_{y\in 9d Q_{\ast}^{j+1,\ell}\cap(\mathbb{R}^d\backslash\mathcal{O}^j)}\mathcal{M}_{\mathrm{loc}}^{0}(f)(y)\lsim 2^j.\label{operamaxi7loc}
\end{equation}
\end{lem}

\begin{lem}\label{lemdecompoat5loc}
For every $j\in\mathbb{Z}$, $$\sum_{k\geq0}\sum_{\ell\in E_1^{j+1}}c_{k}^{\ell}\eta_{j+1,\ell}=0,$$ where the series converges pointwise and in $\mathcal{S'}$.
\end{lem}

Now, we can give our decomposition theorem.

\begin{thm}\label{fondamentaltheoloc} 
Let $\delta\geq\left\lfloor d\left(\frac{1}{q}-1\right)\right\rfloor$ be an integer. For every $f\in\mathcal{H}_{\mathrm{loc}}^{(q,p)}$, there exist a sequence $\left\{(\textbf{a}_n, Q^n)\right\}_{n\geq 0}$ in $\mathcal{A}_{\mathrm{loc}}(q,\infty,\delta)$ and a sequence $\left\{{\lambda}_n\right\}_{n\geq 0}$ of scalars such that 
\begin{eqnarray*}
f=\sum_{n\geq 0}{\lambda}_n\textbf{a}_n\ \ \text{in}\ \mathcal{S'}\ \text{and}\ \mathcal{H}_{\mathrm{loc}}^{(q,p)}, 
\end{eqnarray*}
and, for all $\eta>0$,
\begin{eqnarray*}
\left\|\sum_{n\geq 0}\left(\frac{|\lambda_n|}{\left\|\chi_{_{Q^n}}\right\|_{q}}\right)^{\eta}\chi_{_{Q^{n}}}\right\|_{\frac{q}{\eta},\frac{p}{\eta}}^{\frac{1}{{\eta}}}\lsim\left\|f\right\|_{\mathcal{H}_{\mathrm{loc}}^{(q,p)}}.
\end{eqnarray*}
\end{thm}
\begin{proof}
First, we assume that $f\in L^1_{\mathrm{loc}}\cap\mathcal{H}_{\mathrm{loc}}^{(q,p)}$. For each $j\in\mathbb{Z}$, we consider the level set $\mathcal{O}^j:=\left\{x\in\mathbb{R}^d:\ \mathcal{M}_{\mathrm{loc}}^0(f)(x)>2^j\right\}$. By Remark \ref{remarqdenseyloc}, we have $f=g_j+b_j$, with $b_j=\underset{k\geq 0}\sum b_{j,k}$ and $$b_{j,k}=\left\{\begin{array}{lll}(f-c_{j,k})\eta_{_{j,k}}\ ,&\text{ if }&|Q_{\ast}^{j,k}|<1\\ f\eta_{_{j,k}}\ ,&\text{ if }&|Q_{\ast}^{j,k}|\geq1,\end{array}\right.$$ where each $\eta_{j,k}$ is supported in $Q_{\ast}^{j,k}$, with analogous  estimates described in \cite{AbFt}, Lemma 4.1, pp. 1913-1914. Thus, arguing as in the proof of Theorem \ref{fondamentaltheo} (see \cite{AbFt}, Theorem 4.4, p. 1916), we obtain
\begin{eqnarray}
f=\underset{j=-\infty}{\overset{+\infty}\sum}\left(g_{j+1}-g_j\right)\ \text{almost everywhere and in}\ \mathcal{S'}.
\label{operamaxi3loc}
\end{eqnarray} 
By Lemma \ref{lemdecompoat5loc} and the fact that $\sum_{k\geq0}\eta_{j,k}b_{j+1,\ell}=\chi_{_{\mathcal{O}^j}}b_{j+1,\ell}=b_{j+1,\ell}$ for all integer $\ell\geq0$ (since $\text{supp}(b_{j+1,\ell})\subset\text{supp}(\eta_{j+1,\ell})\subset Q_{\ast}^{j+1,\ell}\subset\mathcal{O}^{j+1}\subset\mathcal{O}^j$), we have
\begin{eqnarray*}
g_{j+1}-g_j&=&(f-\underset{\ell\geq 0}\sum b_{j+1,\ell})-(f-\underset{\ell\geq 0}\sum b_{j,\ell})\\
&=&\underset{\ell\geq 0}\sum b_{j,\ell}-\underset{\ell\geq 0}\sum b_{j+1,\ell}+\sum_{k\geq0}\sum_{\ell\in E_1^{j+1}}c_{k}^{\ell}\eta_{j+1,\ell}\\
&=&\underset{\ell\geq 0}\sum b_{j,\ell}-\underset{k\geq 0}\sum\sum_{\ell\geq0}\eta_{j,k}b_{j+1,\ell}+\sum_{k\geq0}\sum_{\ell\in E_1^{j+1}}c_{k}^{\ell}\eta_{j+1,\ell}\\
&=&\underset{k\geq 0}\sum\left[b_{j,k}-\sum_{\ell\in E_2^{j+1}}\eta_{j,k}b_{j+1,\ell}-\sum_{\ell\in E_1^{j+1}}\left(\eta_{j,k}b_{j+1,\ell}-c_{k}^{\ell}\eta_{j+1,\ell}\right)\right]\\
&=&\underset{k\geq 0}\sum A_{j,k}\ ,
\end{eqnarray*}
where all series converge almost everywhere and in $\mathcal{S'}$, and
\begin{equation}
A_{j,k}:=b_{j,k}-\sum_{\ell\in E_2^{j+1}}\eta_{j,k}b_{j+1,\ell}-\sum_{\ell\in E_1^{j+1}}\left(\eta_{j,k}b_{j+1,\ell}-c_{k}^{\ell}\eta_{j+1,\ell}\right). \label{operamaxi7ajreloc}
\end{equation}
By Lemma \ref{lemdecompoat3} (\ref{operamaxi6}) and (\ref{operamaxi5}), we see that 
\begin{eqnarray}
\text{supp}(A_{j,k})\subset C_0 Q_{\ast}^{j,k}=:\widetilde{Q}^{j,k}\ . \label{operamaxi15loc}
\end{eqnarray}
Also, we claim that 
\begin{eqnarray}
|A_{j,k}|\leq C_{1}2^j\ \text{almost everywhere}, \label{operamaxi18loc}
\end{eqnarray} 
where $C_{1}>0$ is a constant independent of $f$, $j$ and $k$. To prove (\ref{operamaxi18loc}), we distinguish two cases for $k$. 

Case 1: $k\in E_1^{j}$. Then, $b_{j,k}=(f-c_{j,k})\eta_{j,k}$. Thus, since $\underset{\ell\geq 0}\sum\eta_{j+1,\ell}=\chi_{_{\mathcal{O}^{j+1}}}$, by rewriting (\ref{operamaxi7ajreloc}), we obtain 
\begin{eqnarray*}
A_{j,k}=f\chi_{_{\mathbb{R}^d\backslash\mathcal{O}^{j+1}}}\eta_{j,k}-c_{j,k}\eta_{j,k}+\eta_{j,k}\underset{\ell\in E_1^{j+1}}\sum c_{j+1,\ell}\eta_{j+1,\ell}+\underset{\ell\in E_1^{j+1}}\sum c_{k}^{\ell}\eta_{j+1,\ell}.
\end{eqnarray*}
Furthermore, $f\in L^1_{\mathrm{loc}}$ implies that $$|f(x)|\leq\mathcal{M}_{{\mathrm{loc}}_{_{0}}}(f)(x)\lsim \mathcal{M}_{\mathrm{loc}}^0(f)(x)\lsim 2^{j+1},$$ for almost all $x\in\mathbb{R}^d\backslash\mathcal{O}^{j+1}$. This, together with Lemma \ref{lemdecompoat4loc} and $\sum_{\ell\geq0}\chi_{Q_{\ast}^{j+1,\ell}}\lsim 1$, implies that $|A_{j,k}(x)|\lsim 2^j$, for almost all $x\in\mathbb{R}^d$.

Case 2: $k\in E_2^{j}$. Then, $b_{j,k}=f\eta_{j,k}$ and we obtain 
\begin{eqnarray*}
A_{j,k}=f\chi_{_{\mathbb{R}^d\backslash\mathcal{O}^{j+1}}}\eta_{j,k}+\eta_{j,k}\underset{\ell\in E_1^{j+1}}\sum c_{j+1,\ell}\eta_{j+1,\ell}+\underset{\ell\in E_1^{j+1}}\sum c_{k}^{\ell}\eta_{j+1,\ell}.
\end{eqnarray*}
Thus, as in Case 1, we have $|A_{j,k}(x)|\lsim 2^j$, for almost all $x\in\mathbb{R}^d$. Combining these two cases, we obtain (\ref{operamaxi18loc}). 

Set
\begin{eqnarray}
\lambda_{j,k}:=C_{1}2^j|\widetilde{Q}^{j,k}|^{\frac{1}{q}}\ \text{ and }\ \mathfrak a_{j,k}:=\lambda_{j,k}^{-1}A_{j,k}. \label{operamaxi17revuajloc}
\end{eqnarray}
Then, $\text{supp}(\mathfrak a_{j,k})\subset\widetilde{Q}^{j,k}$ and $\left\|\mathfrak a_{j,k}\right\|_{\infty}\leq|\widetilde{Q}^{j,k}|^{-\frac{1}{q}}$, by (\ref{operamaxi15loc}) and (\ref{operamaxi18loc}). 

When $k\in F_2^j$, we have $$|\widetilde{Q}^{j,k}|=|C_0 Q_{\ast}^{j,k}|=C_0^d|Q_{\ast}^{j,k}|\geq C_0^d/{C_0^d}=1.$$ Hence $(\mathfrak a_{j,k},\widetilde{Q}^{j,k})\in\mathcal{A}_{\mathrm{loc}}(q,\infty,\delta)$. 

When $k\in F_1^j$, we have $$|\widetilde{Q}^{j,k}|=|C_0 Q_{\ast}^{j,k}|=C_0^d|Q_{\ast}^{j,k}|<C_0^d/{C_0^d}=1$$ and
\begin{eqnarray}
\int_{\mathbb{R}^d}A_{j,k}(x)\mathfrak{p}(x)dx=0,\ \forall\ \mathfrak{p}\in\mathcal{P_{\delta}}. \label{operamaxi17loc}
\end{eqnarray}
For the proof of (\ref{operamaxi17loc}), notice that $F_1^j\subset E_1^j$, since $1/{C_0^d}<1$ (see Lemma \ref{lemdecompoat3} (\ref{operamaxi6})). Hence $b_{j,k}=(f-c_{j,k})\eta_{j,k}$ and
\begin{eqnarray*}
A_{j,k}=(f-c_{j,k})\eta_{j,k}-\sum_{\ell\in E_2^{j+1}}f\eta_{j,k}\eta_{j+1,\ell}
-\sum_{\ell\in E_1^{j+1}}\left((f-c_{j+1,\ell})\eta_{j+1,\ell}\eta_{j,k}-c_{k}^{\ell}\eta_{j+1,\ell}\right), 
\end{eqnarray*}
by (\ref{operamaxi7ajreloc}). Moreover, for every $\ell\in E_2^{j+1}$, we have $${C^d}|Q_{\ast}^{j,k}|<|Q_{\ast}^{j+1,\ell}|,$$ since $|Q_{\ast}^{j,k}|<1/{C_0^d}<1/{C^d}$ (see Lemma \ref{lemdecompoat3} (\ref{operamaxi6})) and $1\leq|Q_{\ast}^{j+1,\ell}|$. Therefore, for all $\ell\in E_2^{j+1}$, $$\text{diam}(Q_{\ast}^{j+1,\ell})>C\text{diam}(Q_{\ast}^{j,k}),$$ which implies that $$Q_{\ast}^{j,k}\cap Q_{\ast}^{j+1,\ell}=\emptyset,$$ for all $\ell\in E_2^{j+1}$, by Lemma \ref{lemdecompoat3} (\ref{operamaxi6}). Hence 
\begin{eqnarray*}
\sum_{\ell\in E_2^{j+1}}f\eta_{j,k}\ \eta_{j+1,\ell}=0.
\end{eqnarray*}
Thus,
\begin{eqnarray*}
A_{j,k}=(f-c_{j,k})\eta_{j,k}-\sum_{\ell\in E_1^{j+1}}\left((f-c_{j+1,\ell})\eta_{j+1,\ell}\ \eta_{j,k}-c_{k}^{\ell}\eta_{j+1,\ell}\right). 
\end{eqnarray*}
Then, (\ref{operamaxi17loc}) follows from (\ref{operamaxi40re1}) and (\ref{operamaxi40re3}). Thus, $(\mathfrak a_{j,k},\widetilde{Q}^{j,k})\in\mathcal{A}_{\mathrm{loc}}(q,\infty,\delta)$.\\
Therefore, by (\ref{operamaxi3loc}),
\begin{eqnarray}
f=\sum_{j=-\infty}^{+\infty}\sum_{k\geq 0}\lambda_{j,k}\mathfrak a_{j,k} \label{operamaxi20ajloc}
\end{eqnarray}
almost everywhere and in $\mathcal{S'}$, where each $(\mathfrak a_{j,k},\widetilde{Q}^{j,k})\in\mathcal{A}_{\mathrm{loc}}(q,\infty,\delta)$. 

Let us fix a real $\eta>0$. Arguing as in the proof of Theorem \ref{fondamentaltheo}, we obtain
\begin{eqnarray}
\left\|\sum_{j=-\infty}^{+\infty}\sum_{k\geq 0}\left(\frac{|\lambda_{j,k}|}{\left\|\chi_{_{\widetilde{Q}^{j,k}}}\right\|_{q}}\right)^{\eta}\chi_{_{\widetilde{Q}^{j,k}}}\right\|_{\frac{q}{\eta},\frac{p}{\eta}}^{\frac{1}{{\eta}}}\lsim\left\|f\right\|_{\mathcal{H}_{\mathrm{loc}}^{(q,p)}} \label{operamaxi19ajloc}
\end{eqnarray}
and the convergence of (\ref{operamaxi20ajloc}) also holds in $\mathcal{H}_{\mathrm{loc}}^{(q,p)}$. Endeed, for each $(\mathfrak a_{j,k},\widetilde{Q}^{j,k})\in\mathcal{A}_{\mathrm{loc}}(q,\infty,\delta)$, we have  
\begin{eqnarray}
\mathcal{M}_{{\mathrm{loc}}_{_{0}}}(\mathfrak a_{j,k})(x)\lsim\left[\mathfrak{M}\left(\chi_{_{\widetilde{Q}^{j,k}}}\right)(x)\right]^{u}\left\|\chi_{_{\widetilde{Q}^{j,k}}}\right\|_q^{-1}, \label{operamaxi19ajlocaj}
\end{eqnarray}
for all $x\in\mathbb{R}^d$, with $u=\frac{d+\delta+1}{d}$, by (\ref{inversetheo6loc}) (see the proof of Theorem \ref{thafondaloc}). With (\ref{operamaxi19ajloc}) and (\ref{operamaxi19ajlocaj}), the convergence in $\mathcal{H}_{\mathrm{loc}}^{(q,p)}$ easily follows.

The general case ($f\in\mathcal{H}_{\mathrm{loc}}^{(q,p)}$) follows from the density of $L_{\mathrm{loc}}^1\cap\mathcal{H}_{\mathrm{loc}}^{(q,p)}$ in $\mathcal{H}_{\mathrm{loc}}^{(q,p)}$ with respect to the quasi-norm $\left\|\cdot\right\|_{\mathcal{H}_{\mathrm{loc}}^{(q,p)}}$ as in the classical case with appropriate modifications. This completes the proof of Theorem \ref{fondamentaltheoloc}.
\end{proof}

\begin{remark} \label{remqdecoatloc0}
By the definition of $\mathfrak a_{j,k}$ , it is clear that $\text{supp}(\mathfrak a_{j,k})\subset\mathcal{O}^j$. Moreover, for every $j\in\mathbb{Z}$, the family $\left\{\text{supp}(\mathfrak a_{j,k})\right\}_{k\geq0}$ has the bounded intersection property, namely,
\begin{eqnarray}
\underset{k\geq 0}\sum\chi_{_{\text{supp}(\mathfrak a_{j,k})}}\lsim 1. \label{thafondamfini300}
\end{eqnarray}
Also, in Theorem \ref{fondamentaltheoloc}, we can replace $\mathcal{A}_{\mathrm{loc}}(q,\infty,\delta)$ by $\mathcal{A}_{\mathrm{loc}}(q,r,\delta)$, for any $1<r<+\infty$, since $\mathcal{A}_{\mathrm{loc}}(q,\infty,\delta)\subset\mathcal{A}_{\mathrm{loc}}(q,r,\delta)$. 
\end{remark}

We give another proof of Theorem \ref{fondamentaltheoloc}. This approach essentially uses Lemma \ref{lemmdepass}, Theorem \ref{fondamentaltheo} and Plancherel-Polya-Nikols'kij's inequality that we recall.

\begin{lem}\label{InePPNk} (Plancherel-Polya-Nikols'kij's inequality, \cite{HT}, Theorem, p. 16).
Let $\Omega$ be a compact subset of $\mathbb{R}^d$ and $0<s<+\infty$. Then, there exist two constants $c_1>0$ and $c_2>0$ such that 
\begin{eqnarray*}
\sup_{z\in\mathbb{R}^d}\frac{|\nabla\phi(x-z)|}{1+|z|^{\frac{d}{s}}}\leq c_1\sup_{z\in\mathbb{R}^d}\frac{|\phi(x-z)|}{1+|z|^{\frac{d}{s}}}\leq c_2\left[\mathfrak{M}(|\phi|^s)(x)\right]^{\frac{1}{s}},
\end{eqnarray*}
for all $\phi\in\mathcal{S}^\Omega$ and all $x\in\mathbb{R}^d$, where $\mathcal{S}^\Omega=\left\{\phi\in\mathcal{S}:\ \text{supp}(\mathcal{F}(\phi))\subset\Omega\right\}$.
\end{lem}

One can also see \cite{HT1}, § 1.5.3, (iii) (4)-(6), p. 31, for this inequality. Now, we can give our second approach of the proof of theorem \ref{fondamentaltheoloc}.

\begin{proof}
Let $f\in\mathcal{H}_{\mathrm{loc}}^{(q,p)}$. Then, by Lemma \ref{lemmdepass}, there exist $u\in\mathcal{H}^{(q,p)}$ and $v\in\mathcal{H}_{\mathrm{loc}}^{(q,p)}\cap\mathcal{C}^{\infty}(\mathbb{R}^d)$ such that $f=u+v$ and 
\begin{eqnarray}
\left\|u\right\|_{\mathcal{H}^{(q,p)}}+\left\|v\right\|_{\mathcal{H}_{\mathrm{loc}}^{(q,p)}}\lsim\left\|f\right\|_{\mathcal{H}_{\mathrm{loc}}^{(q,p)}}. \label{autrapproch}
\end{eqnarray}
Moreover, by Theorem \ref{fondamentaltheo}, there exist a sequence $\left\{(\textbf{a}_n, Q^n)\right\}_{n\geq 0}$ in $\mathcal{A}(q,\infty,\delta)$ and a sequence of scalars $\left\{{\lambda}_n\right\}_{n\geq 0}$ such that $u=\sum_{n\geq 0}{\lambda}_n\textbf{a}_n$ in $\mathcal{S'}$ and $\mathcal{H}^{(q,p)}$, and for any real $\eta>0$,
\begin{eqnarray}
\left\|\sum_{n\geq 0}\left(\frac{|\lambda_n|}{\left\|\chi_{_{Q^n}}\right\|_{q}}\right)^{\eta}\chi_{_{Q^{n}}}\right\|_{\frac{q}{\eta},\frac{p}{\eta}}^{\frac{1}{{\eta}}}\lsim\left\|u\right\|_{\mathcal{H}^{(q,p)}}\lsim\left\|f\right\|_{\mathcal{H}_{\mathrm{loc}}^{(q,p)}},\label{autrapproch00}
\end{eqnarray}
by (\ref{autrapproch}). Also $u=\sum_{n\geq 0}{\lambda}_n\textbf{a}_n$ in $\mathcal{H}_{\mathrm{loc}}^{(q,p)}$, since $\mathcal{H}^{(q,p)}\hookrightarrow\mathcal{H}_{\mathrm{loc}}^{(q,p)}$.  

To end, it suffices to show that there exist a sequence $\left\{(\textbf{b}_n, R^n)\right\}_{n\geq 0}$ in $\mathcal{A}_{\mathrm{loc}}(q,\infty,\delta)$ and a sequence of scalars $\left\{\sigma_n\right\}_{n\geq 0}$ such that $v=\sum_{n\geq 0}{\sigma}_n\textbf{b}_n$ in $\mathcal{H}_{\mathrm{loc}}^{(q,p)}$ and for any real $\eta>0$,
\begin{eqnarray}
\left\|\sum_{n\geq 0}\left(\frac{|\sigma_n|}{\left\|\chi_{_{R^n}}\right\|_{q}}\right)^{\eta}\chi_{_{R^{n}}}\right\|_{\frac{q}{\eta},\frac{p}{\eta}}^{\frac{1}{{\eta}}}\lsim\left\|f\right\|_{\mathcal{H}_{\mathrm{loc}}^{(q,p)}}. \label{autrapproch0}
\end{eqnarray}
To prove this, we borrow some ideas from \cite{NEYS1}, Lemma 7.9. Let us fix a real $\eta>0$. For $m=(m_1,m_2,\ldots,m_d)\in\mathbb{Z}^d$, we let $$\sigma^m:=\left\|\chi_{_{R^m}}\right\|_{q}\sup_{x\in R^m}|v(x)|=|R^m|^{1/q}\sup_{x\in R^m}|(f\ast\theta)(x)|$$ and $$\textbf{b}^m(x):=\left\{\begin{array}{lll}\frac{1}{\sigma^m}\chi_{_{R^m}}(x)v(x),&\text{ if }&\sigma^m\neq0,\\
0,&\text{ if }&\sigma^m=0,\end{array}\right.$$ where $R^m:=\prod_{i=1}^d[m_i,m_i+1]$ and $\theta$ is the function defined in the proof of Lemma \ref{lemmdepass}, namely $\theta(x):=\left(\mathcal{F}^{-1}(\psi)\right)(x)$ with $\psi\in\mathcal{C}^{\infty}(\mathbb{R}^d)$, $\text{supp}(\psi)\subset B(0,2)$, $0\leq\psi\leq 1$ and $\psi\equiv1$ on $B(0,1)$, and $v(x):=(f\ast\theta)(x)$, for all $x\in\mathbb{R}^d$. Then, we clearly have $|R^m|=1$, $\text{supp}(\textbf{b}^m)\subset R^m$ and $$\left\|\textbf{b}^m\right\|_{\infty}\leq|R^m|^{-1/q}.$$ Hence $\left\{(\textbf{b}^m, R^m)\right\}_{m\in\mathbb{Z}^d}\subset\mathcal{A}_{\mathrm{loc}}(q,\infty,\delta)$. Furthermore, we claim that
\begin{eqnarray}
v=\sum_{m\in\mathbb{Z}^d}\sigma^m\textbf{b}^m\ \text{ in }\ \mathcal{S'}. \label{autrapproch1}
\end{eqnarray}
Endeed, by the definition of cubes $R^m$, it's easy to check that for every $m\in\mathbb{Z}^d$, 
\begin{eqnarray}
\sharp\left\{k\in\mathbb{Z}^d:\ R^m\cap R^k\neq\emptyset\right\}\leq 3^d. \label{autrapproch2}
\end{eqnarray}
Thus, $\sum_{m\in\mathbb{Z}^d}\sigma^m\textbf{b}^m$ defines a function on $\mathbb{R}^d$ and, for any $\Phi\in\mathcal{S}$,
\begin{eqnarray*}
\int_{\mathbb{R}^d}\left(\sum_{m\in\mathbb{Z}^d}|\sigma^m\textbf{b}^m(x)\Phi(x)|\right)dx&\leq& 3^d\int_{\mathbb{R}^d}|v(x)\Phi(x)|dx\\&\leq& 3^d\left\|v\right\|_{\infty}\left\|\Phi\right\|_1<+\infty,
\end{eqnarray*}
since $f\in\mathcal{H}_{\mathrm{loc}}^{(q,p)}$ and $\theta\in\mathcal{S}$ implie that $v:=f\ast\theta\in L^\infty$, by the local version of Remark \ref{remarqconver}. Therefore, by Fubini's theorem, 
\begin{eqnarray*}
\int_{\mathbb{R}^d}\left(\sum_{m\in\mathbb{Z}^d}\sigma^m\textbf{b}^m(x)\right)\Phi(x)dx&=&\sum_{m\in\mathbb{Z}^d}\int_{\mathbb{R}^d}\sigma^m\textbf{b}^m(x)\Phi(x)dx\\
&=&\sum_{m\in\mathbb{Z}^d}\int_{Q_m}v(x)\Phi(x)dx=\int_{\mathbb{R}^d}v(x)\Phi(x)dx.
\end{eqnarray*}
This establishes (\ref{autrapproch1}). Also, from the definition of $v$ and $\theta$, it follows that $\text{supp}(\mathcal{F}(v))\subset\text{supp}(\psi)\subset B(0,2)$. Moreover, for all $x\in R^m$, 
\begin{eqnarray*}
\sup_{y\in R^m}|v(y)|= \sup_{y\in R^m}|v(x-(x-y))|=\sup_{z\in Q(0,2)}|v(x-z)|,
\end{eqnarray*}
since $x,y\in R^m$ implie that $z=x-y\in Q(0,2)$. Therefore, for any $0<s<q$, we have 
\begin{eqnarray*}
\sum_{m\in\mathbb{Z}^d}\left(\frac{|\sigma^m|}{\left\|\chi_{_{R^m}}\right\|_{q}}\right)^{\eta}\chi_{_{R^m}}(x)&=&\sum_{m\in\mathbb{Z}^d}\left(\sup_{y\in R^m}|v(y)|\right)^{\eta}\chi_{_{R^m}}(x)\\
&=&\sum_{m\in\mathbb{Z}^d}\left(\sup_{z\in Q(0,2)}|v(x-z)|\right)^{\eta}\chi_{_{R^m}}(x)\\
&\lsim&\sum_{m\in\mathbb{Z}^d}\left(\sup_{z\in Q(0,2)}\frac{|v(x-z)|}{1+|z|^\frac{d}{s}}\right)^{\eta}\chi_{_{R^m}}(x)\\ 
&\lsim&\sum_{m\in\mathbb{Z}^d}\left(\sup_{z\in\mathbb{R}^d}\frac{|v(x-z)|}{1+|z|^\frac{d}{s}}\right)^{\eta}\chi_{_{R^m}}(x)\\
&\lsim&\left(\mathfrak{M}(|v|^s)(x)\right)^{\frac{\eta}{s}},
\end{eqnarray*}
for all $x\in\mathbb{R}^d$, by Lemma \ref{InePPNk} and (\ref{autrapproch2}). Thus,
\begin{eqnarray*}
\left\|\sum_{m\in\mathbb{Z}^d}\left(\frac{|\sigma^m|}{\left\|\chi_{_{R^m}}\right\|_{q}}\right)^{\eta}\chi_{_{R^m}}\right\|_{\frac{q}{\eta},\frac{p}{\eta}}^{\frac{1}{{\eta}}}&\lsim&\left\|\left(\mathfrak{M}(|v|^s)\right)^{\frac{\eta}{s}}\right\|_{\frac{q}{\eta},\frac{p}{\eta}}^{\frac{1}{{\eta}}}\\
&\lsim&\left\|v\right\|_{q,p}\lsim\left\|f\right\|_{\mathcal{H}_{\mathrm{loc}}^{(q,p)}},
\end{eqnarray*}
by Proposition \ref{operamaximacalebr} and the local version of Remark \ref{remarqconver} (\ref{0conver1}). We rearrange the $\textbf{b}^m$ 's, $R^m$ 's and $\sigma^m$ 's to obtain $\left\{(\textbf{b}_n, R^n)\right\}_{n\geq 0}\subset\mathcal{A}_{\mathrm{loc}}(q,\infty,\delta)$ and a sequence of scalars $\left\{\sigma_n\right\}_{n\geq 0}$ such that $v=\sum_{n\geq 0}{\sigma}_n\textbf{b}_n$ in $\mathcal{H}_{\mathrm{loc}}^{(q,p)}$ and (\ref{autrapproch0}) holds. Thus, $$f=\sum_{n\geq 0}{\lambda}_n\textbf{a}_n+\sum_{n\geq 0}{\sigma}_n\textbf{b}_n\text{ in } \mathcal{S'} \text{ and } \mathcal{H}_{\mathrm{loc}}^{(q,p)},$$ with 
\begin{eqnarray*}
&&\left\|\sum_{n\geq 0}\left(\frac{|\lambda_n|}{\left\|\chi_{_{Q^n}}\right\|_{q}}\right)^{\eta}\chi_{_{Q^{n}}}+\sum_{n\geq 0}\left(\frac{|\sigma_n|}{\left\|\chi_{_{R^n}}\right\|_{q}}\right)^{\eta}\chi_{_{R^{n}}}\right\|_{\frac{q}{\eta},\frac{p}{\eta}}^{\frac{1}{{\eta}}}\\
&\lsim&\left\|\sum_{n\geq 0}\left(\frac{|\lambda_n|}{\left\|\chi_{_{Q^n}}\right\|_{q}}\right)^{\eta}\chi_{_{Q^{n}}}\right\|_{\frac{q}{\eta},\frac{p}{\eta}}^{\frac{1}{{\eta}}}+\left\|\sum_{n\geq 0}\left(\frac{|\sigma_n|}{\left\|\chi_{_{R^n}}\right\|_{q}}\right)^{\eta}\chi_{_{R^{n}}}\right\|_{\frac{q}{\eta},\frac{p}{\eta}}^{\frac{1}{{\eta}}}\lsim\left\|f\right\|_{\mathcal{H}_{\mathrm{loc}}^{(q,p)}},
\end{eqnarray*}
by (\ref{autrapproch00}) and (\ref{autrapproch0}). This finishes the second proof of Theorem \ref{fondamentaltheoloc}.
\end{proof}

\begin{remark} \label{remqdecoatloc}
Let $0<\eta\leq 1$, $1<r\leq+\infty$ and $\delta\geq\left\lfloor d\left(\frac{1}{q}-1\right)\right\rfloor$ be an integer. For simplicity, we denote by $\mathcal{H}_{\mathrm{loc},fin}^{(q,p)}$ the subspace of $\mathcal{H}_{\mathrm{loc}}^{(q,p)}$ consisting of finite linear combinations of $(q,r,\delta)$-atoms (for $\mathcal{H}_{\mathrm{loc}}^{(q,p)}$). By Theorems \ref{thafondaloc}, \ref{thafondamloc} and \ref{fondamentaltheoloc}, we have 
\begin{enumerate}
\item If $r=+\infty$, then for all $f\in\mathcal{H}_{\mathrm{loc}}^{(q,p)}$, 
\begin{eqnarray*}
\left\|f\right\|_{\mathcal{H}_{\mathrm{loc}}^{(q,p)}}\approx\inf\left\{\left\|\sum_{n\geq 0}\left(\frac{|\lambda_n|}{\left\|\chi_{_{Q^n}}\right\|_{q}}\right)^{\eta}\chi_{_{Q^{n}}}\right\|_{\frac{q}{\eta},\frac{p}{\eta}}^{\frac{1}{\eta}}: f=\sum_{n\geq 0}{\lambda}_n\textbf{a}_n\right\},
\end{eqnarray*} 
where the infimum is taken over all decompositions of $f$ using $(q,\infty,\delta)$-atom $\textbf{a}_n$ supported on the cube $Q^n$.
\item If $\max\left\{p,1\right\}<r<+\infty$ and $0<\eta<q$, then for all $f\in\mathcal{H}_{\mathrm{loc}}^{(q,p)}$, 
\begin{eqnarray*}
\left\|f\right\|_{\mathcal{H}_{\mathrm{loc}}^{(q,p)}}\approx\inf\left\{\left\|\sum_{n\geq 0}\left(\frac{|\lambda_n|}{\left\|\chi_{_{Q^n}}\right\|_{q}}\right)^{\eta}\chi_{_{Q^{n}}}\right\|_{\frac{q}{\eta},\frac{p}{\eta}}^{\frac{1}{\eta}}: f=\sum_{n\geq 0}{\lambda}_n\textbf{a}_n\right\},
\end{eqnarray*}
where the infimum is taken over all decompositions of $f$ using $(q,r,\delta)$-atom $\textbf{a}_n$ supported on the cube $Q^n$.
\item $\mathcal{H}_{\mathrm{loc},fin}^{(q,p)}$ is dense in $\mathcal{H}_{\mathrm{loc}}^{(q,p)}$ in the quasi-norm $\left\|.\right\|_{\mathcal{H}_{\mathrm{loc}}^{(q,p)}}$.
\item $\mathcal{H}_{\mathrm{loc}}^q$ is dense in $\mathcal{H}_{\mathrm{loc}}^{(q,p)}$ in the quasi-norm $\left\|.\right\|_{\mathcal{H}_{\mathrm{loc}}^{(q,p)}}$.  
\item $\mathcal{H}_{\mathrm{loc}}^{(q,p)}\cap L^r$ is dense in $\mathcal{H}_{\mathrm{loc}}^{(q,p)}$ in the quasi-norm $\left\|.\right\|_{\mathcal{H}_{\mathrm{loc}}^{(q,p)}}$.
\end{enumerate} 
\end{remark}

\begin{cor}\label{densiesdschwloc}
The Schwartz space $\mathcal{S}$ is dense in $\mathcal{H}_{\mathrm{loc}}^{(q,p)}$ with respect to the quasi-norm $\left\|\cdot\right\|_{\mathcal{H}_{\mathrm{loc}}^{(q,p)}}$, for $0<q<+\infty$ and $q\leq p<+\infty$. 
\end{cor} 
\begin{proof}
We distinguish two cases for $q$.  

Case 1: $0<q\leq1$. It's well known that $\mathcal{S}$ is dense in $\mathcal{H}_{\mathrm{loc}}^q$ with respect to the quasi-norm $\left\|\cdot\right\|_{\mathcal{H}_{\mathrm{loc}}^q}$, by \cite{DGG}, p. 35. It follows that $\mathcal{S}$ is dense in $\mathcal{H}_{\mathrm{loc}}^q$ with respect to the quasi-norm $\left\|\cdot\right\|_{\mathcal{H}_{\mathrm{loc}}^{(q,p)}}$, since $\mathcal{H}_{\mathrm{loc}}^q\hookrightarrow\mathcal{H}_{\mathrm{loc}}^{(q,p)}$. Furthermore, $\mathcal{H}_{\mathrm{loc}}^q$ is dense in $\mathcal{H}_{\mathrm{loc}}^{(q,p)}$ with respect to the quasi-norm $\left\|\cdot\right\|_{\mathcal{H}_{\mathrm{loc}}^{(q,p)}}$, by Remark \ref{remqdecoatloc}. Therefore, $\mathcal{S}$ is dense in $\mathcal{H}_{\mathrm{loc}}^{(q,p)}$ in the quasi-norm $\left\|\cdot\right\|_{\mathcal{H}_{\mathrm{loc}}^{(q,p)}}$. 

Case 2: $1<q<+\infty$. Then, $1<q\leq p<+\infty$. Hence $\mathcal{H}_{\mathrm{loc}}^{(q,p)}=(L^q,\ell^p)$ with norms equivalence, by Theorem \ref{gto}. Denote by $\mathcal{C}_{\mathrm{comp}}(\mathbb{R}^d)$ the space of continuous complex values functions on $\mathbb{R}^d$ with compact support. It's clear that $\mathcal{C}_{\mathrm{comp}}(\mathbb{R}^d)\subset L^q\subset(L^q,\ell^p)$. Moreover, $\mathcal{C}_{\mathrm{comp}}(\mathbb{R}^d)$ is dense in $(L^q,\ell^p)$ with respct to the norm $\left\|\cdot\right\|_{q,p}$ , by \cite{BDD}, Section 7, p. 77. Hence $L^q$ is dense in $(L^q,\ell^p)$ in the norm $\left\|\cdot\right\|_{q,p}$. Also, $\mathcal{S}$ is dense in $L^q$ in the norm $\left\|\cdot\right\|_{q,p}$ , since $\mathcal{S}$ is dense in $L^q$ in the norm $\left\|\cdot\right\|_q$ and $L^q\hookrightarrow(L^q,\ell^p)$. Therefore, $\mathcal{S}$ is dense in $(L^q,\ell^p)$ in the norm $\left\|\cdot\right\|_{q,p}$. It finally follows that $\mathcal{S}$ is dense in $\mathcal{H}_{\mathrm{loc}}^{(q,p)}$ with respect to the norm $\left\|\cdot\right\|_{\mathcal{H}_{\mathrm{loc}}^{(q,p)}}$, since $\mathcal{H}_{\mathrm{loc}}^{(q,p)}=(L^q,\ell^p)$ and $\left\|\cdot\right\|_{q,p}\approx\left\|\cdot\right\|_{\mathcal{H}_{\mathrm{loc}}^{(q,p)}}$. This finishes the proof in Case 2 and hence, the proof of Corollary \ref{densiesdschwloc}. 
\end{proof}

We end this subsection by giving the local version of \cite{AbFt}, Theorem 4.9, p. 1924. For this, we define a quasi-norm on $\mathcal{H}_{\mathrm{loc},fin}^{(q,p)}$ the subspace of $\mathcal{H}_{\mathrm{loc}}^{(q,p)}$ consisting of finite linear combinations of $(q,r,\delta)$-atoms, with $1<r\leq+\infty$ and $\delta\geq\left\lfloor d\left(\frac{1}{q}-1\right)\right\rfloor$ an integer fixed.
\begin{enumerate}
\item When $r=+\infty$, we fix $0<\eta\leq 1$. For every $f\in\mathcal{H}_{\mathrm{loc},fin}^{(q,p)}$, we set 
\begin{eqnarray*}
\left\|f\right\|_{\mathcal{H}_{\mathrm{loc},fin}^{(q,p)}}:=\inf\left\{\left\|\sum_{n=0}^m\left(\frac{|\lambda_n|}{\left\|\chi_{_{Q^n}}\right\|_{q}}\right)^{\eta}\chi_{_{Q^{n}}}\right\|_{\frac{q}{\eta},\frac{p}{\eta}}^{\frac{1}{\eta}}: f=\sum_{n=0}^m{\lambda}_n\mathfrak{a}_n,\ m\in\mathbb{Z}_{+}\right\},
\end{eqnarray*}
where the infimum is taken over all finite decompositions of $f$ using $(q,\infty,\delta)$-atom $\mathfrak{a}_n$ supported on the cube $Q^n$.
\item When $\max\left\{p,1\right\}<r<+\infty$, we fix $0<\eta<q$. we set 
\begin{eqnarray*}
\left\|f\right\|_{\mathcal{H}_{\mathrm{loc},fin}^{(q,p)}}:=\inf\left\{\left\|\sum_{n=0}^m\left(\frac{|\lambda_n|}{\left\|\chi_{_{Q^n}}\right\|_{q}}\right)^{\eta}\chi_{_{Q^{n}}}\right\|_{\frac{q}{\eta},\frac{p}{\eta}}^{\frac{1}{\eta}}: f=\sum_{n=0}^m{\lambda}_n\mathfrak{a}_n,\ m\in\mathbb{Z}_{+}\right\},
\end{eqnarray*}
where the infimum is taken over all finite decompositions of $f$ using $(q,r,\delta)$-atom $\mathfrak{a}_n$ supported on the cube $Q^n$. 
\end{enumerate}

It's straightward to see that, in the two cases, $\left\|\cdot\right\|_{\mathcal{H}_{\mathrm{loc},fin}^{(q,p)}}$ defines a quasi-norm on $\mathcal{H}_{\mathrm{loc},fin}^{(q,p)}$.

\begin{remark} \label{remarqdecofinieloc}
With the above assumptions, for any $(\textbf{a},Q)\in\mathcal{A}_{\mathrm{loc}}(q,r,\delta)$, 
\begin{eqnarray*}
\left\|\textbf{a}\right\|_{\mathcal{H}_{\mathrm{loc},fin}^{(q,p)}}\leq\left\|\left(\frac{1}{\left\|\chi_{_{Q}}\right\|_q}\right)^{\eta}\chi_{_{Q}}\right\|_{\frac{q}{\eta},\frac{p}{\eta}}^{\frac{1}{\eta}}=\frac{1}{\left\|\chi_{_{Q}}\right\|_q}\left\|\chi_{_{Q}}\right\|_{q,p},
\end{eqnarray*}
since $\textbf{a}=1\textbf{a}$ is a finite decomposition of $(q,r,\delta)$-atoms. 
\end{remark}

Also, we will need the following lemma whose proof is similar to the one of \cite{AbFt}, Lemma 4.8, p. 1923. We omit details.
\begin{lem}\label{lemmodifvra1loc}
Let $f\in\mathcal{H}_{\mathrm{loc}}^{(q,p)}$ such that $\left\|f\right\|_{\mathcal{H}_{\mathrm{loc}}^{(q,p)}}=1$ and $\text{supp}(f)\subset B(0,R)$ with $R>1$. Then, for all $x\notin B(0,4R)$,
\begin{eqnarray*}
\mathcal{M}_{\mathrm{loc}}^0(f)(x)\leq C(\varphi,d,q,p)R^{-\frac{d}{p}},
\end{eqnarray*}
where $C(\varphi,d,q,p)>0$ is a constant independent of $f$ and $x$.
\end{lem}
 
Now, we can give our result which is the following.
 
\begin{thm} \label{thafondamfiniloc}
Let $\delta\geq\left\lfloor d\left(\frac{1}{q}-1\right)\right\rfloor$ be an integer and $\max\left\{p,1\right\}<r\leq+\infty$. 
\begin{enumerate}
\item If $r<+\infty$, we fix $0<\eta<q$. Then $\left\|\cdot\right\|_{\mathcal{H}_{\mathrm{loc}}^{(q,p)}}$ and $\left\|\cdot\right\|_{\mathcal{H}_{\mathrm{loc},fin}^{(q,p)}}$ are equivalent on $\mathcal{H}_{\mathrm{loc},fin}^{(q,p)}$. \label{thafondamfini01loc}
\item If $r=+\infty$, we fix $0<\eta\leq 1$. Then $\left\|\cdot\right\|_{\mathcal{H}_{\mathrm{loc}}^{(q,p)}}$ and $\left\|\cdot\right\|_{\mathcal{H}_{\mathrm{loc},fin}^{(q,p)}}$ are equivalent on $\mathcal{H}_{\mathrm{loc},fin}^{(q,p)}\cap\mathcal{C}(\mathbb{R}^d)$, where $\mathcal{C}(\mathbb{R}^d)$ denotes the space of continuous complex values functions on $\mathbb{R}^d$. \label{thafondamfini02loc}
\end{enumerate}
\end{thm}
\begin{proof}
We first prove the part (\ref{thafondamfini01loc}). We consider an integer $\delta\geq\left\lfloor d\left(\frac{1}{q}-1\right)\right\rfloor$, $\max\left\{1,p\right\}<r<+\infty$ and $0<\eta<q$. Let $f\in\mathcal{H}_{\mathrm{loc},fin}^{(q,p)}$. From Theorem \ref{thafondamloc}, we deduce that   
\begin{eqnarray*}
\left\|f\right\|_{\mathcal{H}_{\mathrm{loc}}^{(q,p)}}\lsim \left\|f\right\|_{\mathcal{H}_{\mathrm{loc},fin}^{(q,p)}}.
\end{eqnarray*}
It remains to show that 
\begin{eqnarray}
\left\|f\right\|_{\mathcal{H}_{\mathrm{loc},fin}^{(q,p)}}\lsim\left\|f\right\|_{\mathcal{H}_{\mathrm{loc}}^{(q,p)}}. \label{thafondamfini03loc}
\end{eqnarray} 
By homogeneity of the quasi-norm $\left\|\cdot\right\|_{\mathcal{H}_{\mathrm{loc},fin}^{(q,p)}}$, we may assume that $\left\|f\right\|_{\mathcal{H}_{\mathrm{loc}}^{(q,p)}}=1$. Since $f$ is a finite linear combination of $(q,r,\delta)$-atoms, there exists a real $R>1$ such that $\text{supp}(f)\subset B(0,R)$. Hence
\begin{eqnarray}
\mathcal{M}_{\mathrm{loc}}^0(f)(x)\leq C_{\varphi,d,q,p}R^{-\frac{d}{p}}, \label{rev1thafondamfini03loc}
\end{eqnarray}
for all $x\notin B(0,4R)$, by Lemma \ref{lemmodifvra1loc}. For each $j\in\mathbb{Z}$, we set $$\mathcal{O}^j:=\left\{x\in\mathbb{R}^d:\ \mathcal{M}_{\mathrm{loc}}^0(f)(x)>2^j\right\}\cdot$$ Denote by $j'$ the largest integer $j$ such that $2^{j}<C_{\varphi,d,q,p}R^{-\frac{d}{p}}$. Then, 
\begin{eqnarray}
\mathcal{O}^j\subset B(0,4R),\label{thafondamfini1loc}
\end{eqnarray}
for all $j>j'$, by (\ref{rev1thafondamfini03loc}). Furthermore, since $f\in L^r$, by the proof (first approach) of Theorem \ref{fondamentaltheoloc}, there exist a sequence $\left\{\left(\mathfrak a_{j,k},Q_{j,k}\right)\right\}_{(j,k)\in\mathbb{Z}\times\mathbb{Z}_{+}}$ in $\mathcal{A}_{\mathrm{loc}}(q,r,\delta)$ and a sequence of scalars $\left\{\lambda_{j,k}\right\}_{(j,k)\in\mathbb{Z}\times\mathbb{Z}_{+}}$ such that $f=\sum_{j=-\infty}^{+\infty}\sum_{k\geq 0}\lambda_{j,k}\mathfrak a_{j,k}$ almost everywhere and in $\mathcal{S'}$, with  
\begin{equation}
\text{supp}(\mathfrak a_{j,k})\subset\mathcal{O}^j,\ \ |\lambda_{j,k}\mathfrak a_{j,k}|\lsim 2^j\text{ a.e and }\ Q_{j,k}=C_0 Q_{j,k}^{\ast}\ , \label{rev2thafondamfini03loc}
\end{equation}
for each $j\in\mathbb{Z}$, where the family of cubes $\left\{Q_{j,k}^{\ast}\right\}_{k\geq 0}$ satisfies  
\begin{eqnarray}
\underset{k\geq 0}\bigcup Q_{j,k}^{\ast}=\mathcal{O}^j\text{ and }\ \underset{k\geq 0}\sum\chi_{_{Q_{j,k}^{\ast}}}\lsim 1. \label{rev3thafondamfini03loc}
\end{eqnarray}
Also, 
\begin{eqnarray}
\left\|\sum_{j=-\infty}^{+\infty}\underset{k\geq 0}\sum\left(\frac{|\lambda_{j,k}|}{\left\|\chi_{_{Q_{j,k}}}\right\|_{q}}\right)^{\eta}\chi_{_{Q_{j,k}}}\right\|_{\frac{q}{\eta},\frac{p}{\eta}}^{\frac{1}{\eta}}\lsim\left\|f\right\|_{\mathcal{H}_{\mathrm{loc}}^{(q,p)}}.\label{rev4thafondamfini03loc}
\end{eqnarray}
Set $h:=\sum_{j\leq j'}\sum_{k\geq 0}\lambda_{j,k}\mathfrak a_{j,k}$ and $\ell:=\sum_{j>j'}\sum_{k\geq 0}\lambda_{j,k}\mathfrak a_{j,k}$ , where the series converge almost everywhere and in $\mathcal{S'}$. We have $f=h+\ell$ and $\text{supp}(\ell)\subset\bigcup_{j>j'}\mathcal{O}^j\subset B(0,4R)$, by (\ref{rev2thafondamfini03loc}) and (\ref{thafondamfini1loc}). It follows that $\text{supp}(h)\subset B(0,4R)$, since $f=\ell=0$ sur $\mathbb{R}^d\backslash{B(0,4R)}$. Also, with (\ref{rev2thafondamfini03loc}) and (\ref{rev3thafondamfini03loc}), arguing as in the proof of \cite{AbFt}, Theorem 4.9, p. 1924, we have 
\begin{eqnarray*}
\left\|\ell\right\|_{r}\lsim\left\|\mathcal{M}_{\mathrm{loc}}^0(f)\right\|_{r}\approx\left\|f\right\|_{r}<+\infty,
\end{eqnarray*}
since $f\in L^r$, $r>1$ and $N\geq\left\lfloor\frac{d}{q}\right\rfloor+1>\left\lfloor\frac{d}{r}\right\rfloor+1$. Hence $\ell\in L^r$ and the series $\sum_{j>j'}\sum_{k\geq 0}\lambda_{j,k}\mathfrak a_{j,k}$ converges to $\ell$ in $L^r$, by the Lebesgue dominated convergence theorem in $L^r$. Thus, $h=f-\ell\in L^r$. Moreover, since $\sum_{k\geq 0}\chi_{_{\text{supp}(\mathfrak a_{j,k})}}\lsim 1$ (see Remark \ref{remqdecoatloc0} (\ref{thafondamfini300})), arguing as in the proof of \cite{AbFt}, Theorem 4.9, p. 1924, we obtain $|h(x)|\leq C_{\varphi,d,q,p,\delta}R^{-\frac{d}{p}}$, for almost all $x\in\mathbb{R}^d$, by the choice of $j'$. Setting $$C_{2}:=\left(C_{\varphi,d,q,p,\delta}R^{-\frac{d}{p}}\right)^{-1}\left\|\chi_{_{Q(0,8R)}}\right\|_{q}^{-1},$$ we have $|C_{2}h(x)|\leq\left\|\chi_{_{Q(0,8R)}}\right\|_{q}^{-1}$, for almost all $x\in\mathbb{R}^d$. Therefore, $(C_{2}h,Q(0,8R))\in\mathcal{A}_{\mathrm{loc}}(q,\infty,\delta)\subset\mathcal{A}_{\mathrm{loc}}(q,r,\delta)$, since $\text{supp}(h)\subset B(0,4R)\subset Q(0,8R)$. 

Now, we rewrite $\ell$ as a finite linear combination of $(q,r,\delta)$-atoms. For every positive integer $i$, we set $$F_i:=\left\{(j,k)\in\mathbb{Z}\times\mathbb{Z}_{+}: j>j', |j|+k\leq i\right\}$$ and define the finite sum $\ell_i$ by $\ell_i=\sum_{(j,k)\in F_i}\lambda_{j,k}\mathfrak a_{j,k}$. The convergence of the series $\sum_{j>j'}\sum_{k\geq 0}\lambda_{j,k}\mathfrak a_{j,k}$ to $\ell$ in $L^r$ implies that there exists an integer $i_0>0$ such that $\left\|\ell-\ell_{i_0}\right\|_{r}\leq|Q(0,8R)|^{\frac{1}{r}-\frac{1}{q}}$. Hence $(\ell-\ell_{i_0},Q(0,8R))\in\mathcal{A}_{\mathrm{loc}}(q,r,\delta)$, since $\text{supp}(\ell-\ell_{i_0})\subset B(0,4R)\subset Q(0,8R)$. Thus, $\ell=(\ell-\ell_{i_0})+\ell_{i_0}$ is a finite linear combination of $(q,r,\delta)$-atoms, and
\begin{eqnarray*}
f=h+\ell={C_{2}}^{-1}(C_{2}h)+(\ell-\ell_{i_0})+\ell_{i_0} 
\end{eqnarray*}
is a finite decomposition of $f$ as $(q,r,\delta)$-atoms. Therefore, arguing as in the proof of \cite{AbFt}, Theorem 4.9, p. 1924, we obtain  
\begin{eqnarray*}
\left\|f\right\|_{\mathcal{H}_{\mathrm{loc},fin}^{(q,p)}}\lsim I_1+I_2+I_3\lsim 1,
\end{eqnarray*}
where $$I_1=\left\|\left(\frac{{C_{2}}^{-1}}{\left\|\chi_{_{Q(0,8R)}}\right\|_{q}}\right)^{\eta}\chi_{_{Q(0,8R)}}\right\|_{\frac{q}{\eta},\frac{p}{\eta}}^{\frac{1}{\eta}},\  I_2=\left\|\left(\frac{1}{\left\|\chi_{_{Q(0,8R)}}\right\|_{q}}\right)^{\eta}\chi_{_{Q(0,8R)}}\right\|_{\frac{q}{\eta},\frac{p}{\eta}}^{\frac{1}{\eta}}$$ and $$I_3=\left\|\underset{(j,k)\in F_{i_0}}\sum\left(\frac{|\lambda_{j,k}|}{\left\|\chi_{_{Q_{j,k}}}\right\|_{q}}\right)^{\eta}\chi_{_{Q_{j,k}}}\right\|_{\frac{q}{\eta},\frac{p}{\eta}}^{\frac{1}{\eta}}.$$ 
This implies (\ref{thafondamfini03loc}) which completes the proof of (\ref{thafondamfini01loc}).  
 
For the part (\ref{thafondamfini02loc}), we fix $0<\eta\leq 1$. Let $f\in\mathcal{H}_{\mathrm{loc},fin}^{(q,p)}\cap\mathcal{C}(\mathbb{R}^d)$. As in (\ref{thafondamfini01loc}), we have $\left\|f\right\|_{\mathcal{H}_{\mathrm{loc}}^{(q,p)}}\lsim\left\|f\right\|_{\mathcal{H}_{\mathrm{loc},fin}^{(q,p)}}$, by Theorem \ref{thafondaloc}. For the reverse inequality, we also may assume that $\left\|f\right\|_{\mathcal{H}_{\mathrm{loc}}^{(q,p)}}=1$, by homogeneity. We have $f\in L_{\mathrm{loc}}^1$. Thus, as in the proof of (\ref{thafondamfini01loc}), there exist a sequence $\left\{\left(\mathfrak a_{j,k},Q_{j,k}\right)\right\}_{(j,k)\in\mathbb{Z}\times\mathbb{Z}_{+}}$ in $\mathcal{A}_{\mathrm{loc}}(q,\infty,\delta)$ and a sequence of scalars $\left\{\lambda_{j,k}\right\}_{(j,k)\in\mathbb{Z}\times\mathbb{Z}_{+}}$ such that $f=\sum_{j=-\infty}^{+\infty}\sum_{k\geq 0}\lambda_{j,k}\mathfrak a_{j,k}$ almost everywhere and in  $\mathcal{S'}$, and (\ref{rev2thafondamfini03loc}), (\ref{rev3thafondamfini03loc}) and (\ref{rev4thafondamfini03loc}) hold. Notice that each $(q,\infty,\delta)$-atom $\mathfrak a_{j,k}$ is continuous by examining its definition (see (\ref{operamaxi7ajreloc}) and (\ref{operamaxi17revuajloc})), since $f$ is continuous. Also, since $\left\|\Phi\right\|_{1}\leq\mathfrak{N}_{N}(\Phi)\leq 1$ for any $\Phi\in\mathcal{F}_{N}$, it's easy to see that 
\begin{eqnarray}
\mathcal{M}_{\mathrm{loc}}^0(f)(x)\leq\left\|f\right\|_{\infty}. \label{thafondamfini5loc}
\end{eqnarray}
for all $x\in\mathbb{R}^d$. Then, it follows from (\ref{thafondamfini5loc}) that $\mathcal{O}^j=\emptyset$, for all $j\in\mathbb{Z}$ such that $\left\|f\right\|_{\infty}\leq 2^j$. Moreover, as in the part (\ref{thafondamfini01loc}), there exists a real $R>1$ such that $\text{supp}(f)\subset B(0,R)$ so that (\ref{rev1thafondamfini03loc}) holds. Let $h$ and $\ell$ be the functions as defined in the part (\ref{thafondamfini01loc}). Arguing as in (\ref{thafondamfini01loc}), we obtain that $\text{supp}(\ell)\subset B(0,4R)$ and $(C_{2}'h,Q(0,8R))\in\mathcal{A}_{\mathrm{loc}}(q,\infty,\delta)$ with $C_{2}':=\left(C_{\varphi,d,q,p,\delta}R^{-\frac{d}{p}}\right)^{-1}\left\|\chi_{_{Q(0,8R)}}\right\|_{q}^{-1}$. Denote by $j''$ the largest integer $j$ such that $2^{j}<\left\|f\right\|_{\infty}$. Then 
\begin{eqnarray*}
\ell:=\sum_{j>j'}\sum_{k\geq 0}\lambda_{j,k}\mathfrak a_{j,k}=\sum_{j'<j\leq j''}\sum_{k\geq 0}\lambda_{j,k}\mathfrak a_{j,k}.
\end{eqnarray*}		
Let $\epsilon>0$. Since $f$ is uniformly continuous, there exists a real $\gamma>0$ such that, for all $x,y\in\mathbb{R}^d$, $|x-y|<\gamma$ implies that $|f(x)-f(y)|<\epsilon$. Without loss of generality, we may assume that $\gamma<1$. Write $\ell=\ell_1^{\epsilon}+\ell_2^{\epsilon}$, with $$\ell_1^{\epsilon}:=\underset{(j,k)\in G_1}\sum\lambda_{j,k}\mathfrak a_{j,k}\ \ \text{ and }\ \ \ell_2^{\epsilon}:=\underset{(j,k)\in G_2}\sum\lambda_{j,k}\mathfrak a_{j,k}\ ,$$ where $$G_1:=\left\{(j,k)\in\mathbb{Z}\times\mathbb{Z}_{+}: \text{diam}(Q_{j,k})\geq\gamma,\  j'<j\leq j''\right\}$$ and $$G_2:=\left\{(j,k)\in\mathbb{Z}\times\mathbb{Z}_{+}: \text{diam}(Q_{j,k})<\gamma,\  j'<j\leq j''\right\}.$$ Using (\ref{rev2thafondamfini03loc}), the construction of cubes $Q_{j,k}^\ast$ ($Q_{j,k}^\ast=c\tilde{Q}_{j,k}$, for some $1<c<5/4$, where the family of cubes $\left\{\tilde{Q}_{j,k}\right\}_{k\geq0}$ corresponds to the Whitney decomposition of $\mathcal{O}^j$) and (\ref{thafondamfini1loc}), we see that $\text{card}(G_1)<+\infty$. Therefore, $\ell_1^{\epsilon}$ is continuous. 
 
For any $(j,k)\in G_2$ and $x\in Q_{j,k}$, we clearly have $|f(x)-f(x_{j,k})|<\epsilon$, where $x_{j,k}$ stands for the center of $Q_{j,k}$. Set $$\tilde{f}(x):=[f(x)-f(x_{j,k})]\chi_{_{Q_{j,k}}}(x) \text{ and } \tilde{c}_{j,k}(x):=c_{j,k}(x)-f(x_{j,k}),$$ for all $x\in\mathbb{R}^d$, where $c_{j,k}$ is the polynomial defined by (\ref{operamaxi40re1}). We have $[\tilde{f}(x)-\tilde{c}_{j,k}(x)]\eta_{j,k}(x)=[f(x)-c_{j,k}(x)]\eta_{j,k}(x)$, for all $x\in\mathbb{R}^d$, where $\eta_{j,k}$ is the function in (\ref{operamaxi40re1}), since $\text{supp}(\eta_{j,k})\subset Q_{j,k}^\ast\subset Q_{j,k}$. Hence 
\begin{eqnarray*}
\int_{\mathbb{R}^d}[\tilde{f}(x)-\tilde{c}_{j,k}(x)]\eta_{j,k}(x)\mathfrak{p}(x)dx=0, 
\end{eqnarray*}
for all $\mathfrak{p}\in\mathcal{P_{\delta}}$, by (\ref{operamaxi40re1}). Since $|\tilde{f}(x)|<\epsilon$ for all $x\in\mathbb{R}^d$ implies that $\mathcal{M}_{\mathrm{loc}}^0(\tilde{f})(x)\leq\epsilon$ for all $x\in\mathbb{R}^d$, and $\ell^{j,k}:=\ell(Q_{j,k}^\ast)<\frac{\gamma}{C_0\sqrt{d}}<1$, then by Lemma \ref{lemdecompoat4loc} (\ref{operamaxi8loc}), we have 
\begin{eqnarray}
\sup_{x\in\mathbb{R}^d}|\tilde{c}_{j,k}(x)\eta_{j,k}(x)|\lsim\sup_{y\in\mathbb{R}^d}\mathcal{M}_{\mathrm{loc}}^0(\tilde{f})(y)\lsim\epsilon. \label{thafondamfini6loc}
\end{eqnarray}
Let $i\geq0$ be an integer. Denote by $\tilde{c}_{k}^{i}$ the polynomial of $\mathcal{P_{\delta}}$ such that 
\begin{eqnarray*}
\int_{\mathbb{R}^d}[\tilde{f}(x)-\tilde{c}_{j+1,i}(x)]\eta_{j,k}(x)\mathfrak{p}(x)\eta_{j+1,i}(x)dx=\int_{\mathbb{R}^d}\tilde{c}_{k}^{i}(x)\mathfrak{p}(x)\eta_{j+1,i}(x)dx,
\end{eqnarray*}
for all $\mathfrak{p}\in\mathcal{P_{\delta}}$, where $\tilde{c}_{j+1,i}(x)=c_{j+1,i}(x)-f(x_{j,k})$. The above expression means that $\tilde{c}_{k}^{i}$ is the orthogonal projection of $(\tilde{f}-\tilde{c}_{j+1,i})\eta_{j,k}$ on $\mathcal{P_{\delta}}$ with respect to the norm (\ref{normeprecis}). Since $\text{supp}(\eta_{j,k})\subset Q_{j,k}^\ast\subset Q_{j,k}$, we have $[\tilde{f}(x)-\tilde{c}_{j+1,i}(x)]\eta_{j,k}(x)=[f(x)-c_{j+1,i}(x)]\eta_{j,k}(x)$, for all $x\in\mathbb{R}^d$. Hence $\tilde{c}_{k}^{i}=c_{k}^{i}$, by (\ref{operamaxi40re3}). Then, we infer from Lemma \ref{lemdecompoat4loc} (\ref{operamaxi7loc}) that 
\begin{eqnarray}
\sup_{y\in\mathbb{R}^d}|\tilde{c}_{k}^{i}(y)\eta_{j+1,i}(y)|\lsim\sup_{y\in\mathbb{R}^d}\mathcal{M}_{\mathrm{loc}}^0(\tilde{f})(y)\lsim\epsilon. \label{thafondamfini7loc}
\end{eqnarray} 
Endeed, since $\tilde{c}_{k}^{i}=c_{k}^{i}$, we have $\tilde{c}_{k}^{i}=0$, if $Q_{j,k}^\ast\cap Q_{j+1,i}^\ast=\emptyset$, by (\ref{operamaxi5}). In this case, (\ref{thafondamfini7loc}) is clearly verified. If $Q_{j,k}^\ast\cap Q_{j+1,i}^\ast\neq\emptyset$, then $\text{diam}(Q_{j+1,i}^\ast)\leq C\text{diam}(Q_{j,k}^\ast)$, by Lemma \ref{lemdecompoat3} (\ref{operamaxi6}). Thus, 
\begin{eqnarray*}
\ell^{j+1,i}:=\ell(Q_{j+1,i}^\ast)\leq\frac{C}{\sqrt{d}}\text{diam}(Q_{j,k}^\ast)<\frac{C}{C_0\sqrt{d}}\gamma<1,
\end{eqnarray*}
since $1<C<C_0$ (see Lemma \ref{lemdecompoat3} (\ref{operamaxi6})). Moreover,
\begin{eqnarray}
\int_{\mathbb{R}^d}[\tilde{f}(x)-\tilde{c}_{j+1,i}(x)]\eta_{j+1,i}(x)\mathfrak{p}(x)dx=0,\label{recjkiloc}
\end{eqnarray}
for all $\mathfrak{p}\in\mathcal{P_{\delta}}$. Endeed, $Q_{j,k}^\ast\cap Q_{j+1,i}^\ast\neq\emptyset$ implies that $\text{supp}(\eta_{j+1,i})\subset Q_{j+1,i}^\ast\subset C_0Q_{j,k}^\ast=:Q_{j,k}$, by lemma \ref{lemdecompoat3} (\ref{operamaxi6}). Hence $[\tilde{f}(x)-\tilde{c}_{j+1,i}(x)]\eta_{j+1,i}(x)=[f(x)-c_{j+1,i}(x)]\eta_{j+1,i}(x)$, for all $x\in\mathbb{R}^d$, which implies (\ref{recjkiloc}), by (\ref{operamaxi40re1}). Therefore, we can apply Lemma \ref{lemdecompoat4loc} (\ref{operamaxi7loc}) to obtain (\ref{thafondamfini7loc}). 

Also, when $Q_{j,k}^\ast\cap Q_{j+1,i}^\ast\neq\emptyset$, we have
\begin{eqnarray}
\sup_{y\in\mathbb{R}^d}|\tilde{c}_{j+1,i}(y)\eta_{j+1,i}(y)|\lsim\sup_{y\in\mathbb{R}^d}\mathcal{M}_{\mathrm{loc}}^0(\tilde{f})(y)\lsim\epsilon,
\label{thafondamfini8loc}
\end{eqnarray} 
by Lemma \ref{lemdecompoat4loc} (\ref{operamaxi8loc}), since (\ref{recjkiloc}) holds and $\ell^{j+1,i}<1$. Moreover, since $\ell^{j,k}<1$, it follows from the definition of $\lambda_{j,k}\mathfrak a_{j,k}$ (see (\ref{operamaxi7ajreloc}) and (\ref{operamaxi17revuajloc})) that  
\begin{eqnarray*}
\lambda_{j,k}\mathfrak a_{j,k}&=&(f-c_{j,k})\eta_{j,k}-\underset{i\in E_1^{j+1}}\sum(f-c_{j+1,i})\eta_{j+1,i}\eta_{j,k} + \underset{i\in E_1^{j+1}}\sum c_{k}^{i}\eta_{j+1,i}\\
&=&\tilde{f}\eta_{j,k}(1-\underset{i\in E_1^{j+1}}\sum\eta_{j+1,i})-\tilde{c}_{j,k}\eta_{j,k}+\eta_{j,k}\underset{i\in E_1^{j+1}}\sum\tilde{c}_{j+1,i}\eta_{j+1,i}+\underset{i\in E_1^{j+1}}\sum\tilde{c}_{k}^{i}\eta_{j+1,i}.
\end{eqnarray*}
Also, by $0\leq\underset{i\in E_1^{j+1}}\sum\eta_{j+1,i}(x)\leq\underset{i\geq 0}\sum\eta_{j+1,i}(x)=\chi_{_{\mathcal{O}^{j+1}}}(x)\leq1$, we have  
\begin{eqnarray*}
\left|\tilde{f}(x)\eta_{j,k}(x)\left(1-\underset{i\in E_1^{j+1}}\sum\eta_{j+1,i}(x)\right)\right|\leq|\tilde{f}(x)|\leq\epsilon, 
\end{eqnarray*}
for all $x\in Q_{j,k}$ and $(j,k)\in G_2$. From this together with (\ref{thafondamfini6loc}), (\ref{thafondamfini7loc}); (\ref{thafondamfini8loc}) and the fact that $\underset{i\geq 0}\sum\chi_{_{Q_{j+1,i}^{\ast}}}\lsim 1$, it follows that  
\begin{eqnarray*}
|\lambda_{j,k}\mathfrak a_{j,k}(x)|\lsim\epsilon,
\end{eqnarray*}
for all $x\in Q_{j,k}$ and $(j,k)\in G_2$. Thus, using the fact that $\underset{k\geq 0}\sum\chi_{_{\text{supp}(\mathfrak a_{j,k})}}\lsim 1$ (see Remark \ref{remqdecoatloc0} (\ref{thafondamfini300})), we obtain  
\begin{eqnarray}
|\ell_2^{\epsilon}(x)|\leq C_3\sum_{j'<j\leq j''}\epsilon=C_3(j''-j')\epsilon, \label{recjki1loc}
\end{eqnarray}
for all $x\in\mathbb{R}^d$, where $C_3>0$ is a constant independent of $x$, $f$; $k$ and $\epsilon$. We deduce from this estimate that $\ell$ is continuous. Therefore, $h=f-\ell$ is continuous and $C_2'h$ is a continuous $(q,\infty,\delta)$-atom.

Now, we find a finite decomposition of $\ell$ as continuous $(q,\infty,\delta)$-atoms by using again the splitting $\ell=\ell_1^{\epsilon}+\ell_2^{\epsilon}$. Endeed, it is clear that, for all $\epsilon>0$, $\ell_1^{\epsilon}$ is a finite linear combination of continuous $(q,\infty,\delta)$-atoms. Also, $\ell_2^{\epsilon}=\ell-\ell_1^{\epsilon}$ is continuous and $\text{supp}\left(\ell_2^{\epsilon}\right)\subset B(0,4R)\subset Q(0,8R)$. Hence taking $\epsilon=[C_3(j''-j')]^{-1}|Q(0,8R)|^{-1/q}$, we have $|\ell_2^{\epsilon}(x)|\leq|Q(0,8R)|^{-1/q}$, by (\ref{recjki1loc}), so that $(\ell_2^{\epsilon},Q(0,8R))\in\mathcal{A}_{\mathrm{loc}}(q,\infty,\delta)$. This establishes the finite  decomposition of $\ell$ as  continuous $(q,\infty,\delta)$-atoms.   

Thus, $f=h+\ell={C_2'}^{-1}(C_2'h)+\ell_1^{\epsilon}+\ell_2^{\epsilon}$ is a finite decomposition of $f$ as continuous $(q,\infty,\delta)$-atoms and as in the part (\ref{thafondamfini01loc}), we obtain $$\left\|f\right\|_{\mathcal{H}_{\mathrm{loc},fin}^{(q,p)}}\lsim 1.$$ This finishes the proof of the part (\ref{thafondamfini02loc}) and hence of Theorem \ref{thafondamfiniloc}.
\end{proof}

We give a useful result of density. 

\begin{lem}\label{finilemmodifvra1loc}
Let $\mathcal{H}_{\mathrm{loc},fin}^{(q,p)}$ be the subspace of $\mathcal{H}_{\mathrm{loc}}^{(q,p)}$ consisting of finite linear combinations of $(q,\infty,\delta)$-atoms. Then, $\mathcal{H}_{\mathrm{loc},fin}^{(q,p)}\cap\mathcal{C}^{\infty}(\mathbb{R}^d)$ is a dense subspace of $\mathcal{H}_{\mathrm{loc},fin}^{(q,p)}$ in the quasi-norm $\left\|\cdot\right\|_{\mathcal{H}_{\mathrm{loc}}^{(q,p)}}$. 
\end{lem}
\begin{proof}
We use some arguments of the proof of \cite{YDYS}, Theorem 6.4. Let $f\in\mathcal{H}_{\mathrm{loc},fin}^{(q,p)}$. Then, $f$ is a finite linear combination of $(q,\infty,\delta)$-atoms. Therefore, there exists a real $R>0$ such that $\text{supp}(f)\subset B(0,R)$. Furthermore, for all $0<t\leq1$, we have  
\begin{eqnarray}
f\ast\varphi_t\in\mathcal{C}^{\infty}(\mathbb{R}^d)\ \text{ and }\ \text{supp}(f\ast\varphi_t)\subset B(0,R+1) \label{rev1finilemmodifvra1loc}.
\end{eqnarray}
Assume that $f=\sum_{n=0}^j{\lambda}_n\textbf{a}_n$ with $\left\{(\textbf{a}_n, Q^n)\right\}_{n=0}^j\subset\mathcal{A}_{\mathrm{loc}}(q,\infty,\delta)$ and $\left\{\lambda_n\right\}_{n=0}^j\subset\mathbb{C}$. Then, for all $0<t\leq1$,
\begin{eqnarray*}
f\ast\varphi_t=\sum_{n=0}^j{\lambda}_n(\textbf{a}_n\ast\varphi_t).
\end{eqnarray*}
Now, for any $0<t\leq1$ and $n\in\left\{0,1,\ldots,j\right\}$, we prove that there exists a constant $c_{t,n}>0$ such that $c_{t,n}(\textbf{a}_n\ast\varphi_t)$ is a $(q,\infty,\delta)$-atom lying in $\mathcal{C}^{\infty}(\mathbb{R}^d)$, which implies that for any $0<t\leq1$,
\begin{eqnarray*}
f\ast\varphi_t\in\mathcal{H}_{\mathrm{loc},fin}^{(q,p)}\cap\mathcal{C}^{\infty}(\mathbb{R}^d).
\end{eqnarray*}
For $n\in\left\{0,1,\ldots,j\right\}$, we have $\text{supp}(\textbf{a}_n)\subset Q^n:=Q(x^n,\ell_n)$. Then $\text{supp}(\textbf{a}_n\ast\varphi_t)\subset \widetilde{Q}^n:=Q(x^n,\ell_n+2t)$. Moreover, $\textbf{a}_n\ast\varphi_t\in\mathcal{C}^{\infty}$ and
\begin{eqnarray*}
\left\|\textbf{a}_n\ast\varphi_t\right\|_{\infty}\leq\left\|\varphi\right\|_{1}\left\|\textbf{a}_n\right\|_{\infty}\leq\left(\left\|\varphi\right\|_{1}\frac{|Q^n|^{-\frac{1}{q}}}{|\widetilde{Q}^n|^{-\frac{1}{q}}}\right)|\widetilde{Q}^n|^{-\frac{1}{q}}.
\end{eqnarray*}
Also, for all $\beta\in\mathbb{Z}_{+}^d$ with $|\beta|\leq\delta$, $\int_{\mathbb{R}^d}x^{\beta}\textbf{a}_n(x)dx=0$ implies that 
\begin{eqnarray*}
\int_{\mathbb{R}^d}x^{\beta}(\textbf{a}_n\ast\varphi_t)(x)dx=0.
\end{eqnarray*}
Thus, with $c_{t,n}:=\left(\left\|\varphi\right\|_{1}\frac{|Q^n|^{-\frac{1}{q}}}{|\widetilde{Q}^n|^{-\frac{1}{q}}}\right)^{-1}=\left\|\varphi\right\|_{1}^{-1}\left(\frac{|Q^n|}{|\widetilde{Q}^n|}\right)^{\frac{1}{q}}$, $c_{t,n}(\textbf{a}_n\ast\varphi_t)$ is a $(q,\infty,\delta)$-atom lying in $\mathcal{C}^{\infty}(\mathbb{R}^d)$. 

Likewise, $\text{supp}(f-f\ast\varphi_t)\subset B(0,R+1)$, by (\ref{rev1finilemmodifvra1loc}), and $f-f\ast\varphi_t$ has the same vanishing moments as $f$. Let $1<s<\infty$. We have 
\begin{eqnarray}
\lim_{t\rightarrow 0}\left\|f-f\ast\varphi_t\right\|_{s}\rightarrow0, \label{clairmaintloc} 
\end{eqnarray}
since $f\in L^s$ and $\int_{\mathbb{R}^d}\varphi(x)dx=1$. Without loss of generality, we may assume that $\left\|f-f\ast\varphi_t\right\|_{s}>0$, when $t$ is small enough. Set $$c_t:=\left\|f-f\ast\varphi_t\right\|_{s}|Q(0,2(R+1))|^{1/q-1/s}$$ and $\mathfrak a_t:=(c_t)^{-1}(f-f\ast\varphi_t)$. Then, $\mathfrak a_t$ is a $(q,s,\delta)$-atom, $f-f\ast\varphi_t=c_t\mathfrak a_t$ and $c_t\rightarrow0$ as $t\rightarrow0$, by (\ref{clairmaintloc}). Thus, $\left\|f-f\ast\varphi_t\right\|_{\mathcal{H}_{\mathrm{loc}}^{(q,p)}}\lsim c_t\rightarrow 0$ as $t\rightarrow 0$.
The proof of Lemma \ref{finilemmodifvra1loc} is complete. 
\end{proof}

If $\mathcal{H}_{\mathrm{loc},fin}^{(q,p)}$ is the subspace of $\mathcal{H}_{\mathrm{loc}}^{(q,p)}$ consisting of finite linear combinations of $(q,\infty,\delta)$-atoms, then it follows from Lemma \ref{finilemmodifvra1loc} that $\mathcal{H}_{\mathrm{loc},fin}^{(q,p)}\cap\mathcal{C}(\mathbb{R}^d)$ is dense in $\mathcal{H}_{\mathrm{loc}}^{(q,p)}$ with respect to the quasi-norm $\left\|\cdot\right\|_{\mathcal{H}_{\mathrm{loc}}^{(q,p)}}$, since $\mathcal{H}_{\mathrm{loc},fin}^{(q,p)}$ is dense in $\mathcal{H}_{\mathrm{loc}}^{(q,p)}$ in the quasi-norm $\left\|\cdot\right\|_{\mathcal{H}_{\mathrm{loc}}^{(q,p)}}$.

\subsection{Molecular Decomposition} 

Our definition of molecule for $\mathcal{H}_{\mathrm{loc}}^{(q,p)}$ spaces is the one of $\mathcal{H}^{(q,p)}$ spaces (see \cite{AbFt}, Definition 5.1, p. 1931) except that, when $|Q|\geq1$, the condition (3) in \cite{AbFt}, Definition 5.1, is not requiered. Thus, a $(q,r,\delta)$-atom for $\mathcal{H}_{\mathrm{loc}}^{(q,p)}$ is a $(q,r,\delta)$-molecule for $\mathcal{H}_{\mathrm{loc}}^{(q,p)}$. We denote by $\mathcal{M}\ell_{\mathrm{loc}}(q,r,\delta)$ the set of all $(\textbf{m},Q)$ such that $\textbf{m}$ is a $(q,r,\delta)$-molecule centered at $Q$ (for $\mathcal{H}_{\mathrm{loc}}^{(q,p)}$).

We have the following decomposition theorem as an application of Theorem \ref{fondamentaltheoloc}.

\begin{thm} \label{fondamentaltheomolecloc}
Let $1<r\leq+\infty$ and $\delta\geq\left\lfloor d\left(\frac{1}{q}-1\right)\right\rfloor$ be an integer. Then, for every $f\in\mathcal{H}_{\mathrm{loc}}^{(q,p)}$, there exist a sequence $\left\{(\textbf{m}_n, Q^n)\right\}_{n\geq 0}$ in $\mathcal{M}\ell_{\mathrm{loc}}(q,r,\delta)$ and a sequence of scalars $\left\{{\lambda}_n\right\}_{n\geq 0}$ such that 
\begin{eqnarray*}
f=\sum_{n\geq 0}{\lambda}_n\textbf{m}_n\ \ \text{in}\ \mathcal{S'}\ \text{and}\ \mathcal{H}_{\mathrm{loc}}^{(q,p)},
\end{eqnarray*}
and, for all $\eta>0$,
\begin{eqnarray*}
\left\|\sum_{n\geq 0}\left(\frac{|\lambda_n|}{\left\|\chi_{_{Q^n}}\right\|_{q}}\right)^{\eta}\chi_{_{Q^{n}}}\right\|_{\frac{q}{\eta},\frac{p}{\eta}}^{\frac{1}{{\eta}}}\lsim\left\|f\right\|_{\mathcal{H}_{\mathrm{loc}}^{(q,p)}}.
\end{eqnarray*}
\end{thm}

Also, we have the two following reconstruction theorems which are the local analogues of \cite{AbFt}, Theorems 5.2 and 5.3, pp. 1931-1932.  

\begin{thm} \label{thafondammolec1loc}
Let $0<\eta\leq 1$ and $\delta\geq\left\lfloor d\left(\frac{1}{q}-1\right)\right\rfloor$ be an integer. Then, for all sequences $\left\{(\textbf{m}_n, Q^n)\right\}_{n\geq 0}$ in $\mathcal{M}\ell_{\mathrm{loc}}(q,\infty,\delta)$ and all sequences of scalars $\left\{{\lambda}_n\right\}_{n\geq 0}$ satisfying 
\begin{eqnarray}
\left\|\sum_{n\geq 0}\left(\frac{|\lambda_n|}{\left\|\chi_{_{Q^n}}\right\|_{q}}\right)^{\eta}\chi_{_{Q^{n}}}\right\|_{\frac{q}{\eta},\frac{p}{\eta}}<+\infty, \label{inversetheor2molec1loc}
\end{eqnarray}
the series $f:=\sum_{n\geq 0}{\lambda}_n\textbf{m}_n$ converges in $\mathcal{S'}$ and $\mathcal{H}_{\mathrm{loc}}^{(q,p)}$, with 
\begin{eqnarray*}
\left\|f\right\|_{\mathcal{H}_{\mathrm{loc}}^{(q,p)}}\lsim\left\|\sum_{n\geq 0}\left(\frac{|\lambda_n|}{\left\|\chi_{_{Q^n}}\right\|_{q}}\right)^{\eta}\chi_{_{Q^{n}}}\right\|_{\frac{q}{\eta},\frac{p}{\eta}}^{\frac{1}{\eta}}.
\end{eqnarray*}
\end{thm}
\begin{proof}
Let $\left\{(\textbf{m}_n, Q^n)\right\}_{n\geq 0}$ be a sequence of elements of $\mathcal{M}\ell_{\mathrm{loc}}(q,\infty,\delta)$ and $\left\{{\lambda}_n\right\}_{n\geq 0}$ be a sequence of scalars satisfying (\ref{inversetheor2molec1loc}). Consider an integer $j\geq 0$. Set $J_1=\left\{n\in\left\{0,1,\ldots,j\right\}:\ |Q^n|<1\right\}$ and $J_2=\left\{n\in\left\{0,1,\ldots,j\right\}:\ |Q^n|\geq1\right\}$. We have 
\begin{eqnarray*}
\mathcal{M}_{{\mathrm{loc}}_{_{0}}}\left(\sum_{n\in J_1}\lambda_n\textbf{m}_n\right)(x)\lsim\sum_{n\in J_1}|\lambda_n|\left(\mathfrak{M}(\textbf{m}_n)(x)\chi_{_{\widetilde{Q^n}}}(x)+\frac{\left[\mathfrak{M}(\chi_{_{Q^{n}}})(x)\right]^{u}}{\left\|\chi_{_{Q^{n}}}\right\|_q}\right),
\end{eqnarray*}
for all $x\in\mathbb{R}^d$, where $\widetilde{Q^n}=2\sqrt{d}Q^n$ and $u=\frac{d+\delta+1}{d}$ , by \cite{NEYS}, (5.2), p. 3698, and  
\begin{eqnarray*}
\mathcal{M}_{{\mathrm{loc}}_{_{0}}}\left(\sum_{n\in J_2}\lambda_n\textbf{m}_n\right)(x)\lsim\sum_{n\in J_2}|\lambda_n|\left(\mathfrak{M}(\textbf{m}_n)(x)\chi_{_{\widetilde{Q^n}}}(x)+\frac{\left[\mathfrak{M}(\chi_{_{Q^{n}}})(x)\right]^{u}}{\left\|\chi_{_{Q^{n}}}\right\|_q}\right),
\end{eqnarray*}
for all $x\in\mathbb{R}^d$, where $\widetilde{Q^n}=4\sqrt{d}Q^n$ and $u=\frac{d+\delta+1}{d}$ , by \cite{NEYS1}, (7.6), p. 955. Hence   
\begin{eqnarray*}
\mathcal{M}_{{\mathrm{loc}}_{_{0}}}\left(\sum_{n=0}^j\lambda_n\textbf{m}_n\right)(x)&\leq&\mathcal{M}_{{\mathrm{loc}}_{_{0}}}\left(\sum_{n\in J_1}\lambda_n\textbf{m}_n\right)(x)+\mathcal{M}_{{\mathrm{loc}}_{_{0}}}\left(\sum_{n\in J_2}\lambda_n\textbf{m}_n\right)(x)\\
&\lsim&\sum_{n=0}^j|\lambda_n|\left(\mathfrak{M}(\textbf{m}_n)(x)\chi_{_{\widetilde{Q^n}}}(x)+\frac{\left[\mathfrak{M}(\chi_{_{Q^{n}}})(x)\right]^{u}}{\left\|\chi_{_{Q^{n}}}\right\|_q}\right),
\end{eqnarray*}
for all $x\in\mathbb{R}^d$, where $\widetilde{Q^n}=4\sqrt{d}Q^n$ and $u=\frac{d+\delta+1}{d}\cdot$ Thus, 
\begin{eqnarray*}
\mathcal{M}_{{\mathrm{loc}}_{_{0}}}\left(\sum_{n=0}^j\lambda_n\textbf{m}_n\right)(x)\lsim\sum_{n=0}^j|\lambda_n|\left(\frac{\left[\mathfrak{M}(\chi_{_{Q^{n}}})(x)\right]^{u}}{\left\|\chi_{_{Q^{n}}}\right\|_q}\right),
\end{eqnarray*}
for all $x\in\mathbb{R}^d$, by the proof of \cite{AbFt}, Theorem 5.2, p. 1931. And we end as in the proof of \cite{AbFt}, Theorem 4.3, pp. 1914-1915.  
\end{proof}

\begin{thm} \label{thafondammolec2loc}
Let $\max\left\{p,1\right\}<r<+\infty$, $0<\eta<q$ and $\delta\geq\left\lfloor d\left(\frac{1}{q}-1\right)\right\rfloor$ be an integer. Then, for all sequences $\left\{(\textbf{m}_n, Q^n)\right\}_{n\geq 0}$ in $\mathcal{M}\ell_{\mathrm{loc}}(q,r,\delta)$ and all sequences of scalars $\left\{{\lambda}_n\right\}_{n\geq 0}$ such that 
\begin{eqnarray}
\left\|\sum_{n\geq 0}\left(\frac{|\lambda_n|}{\left\|\chi_{_{Q^n}}\right\|_{q}}\right)^{\eta}\chi_{_{Q^{n}}}\right\|_{\frac{q}{\eta},\frac{p}{\eta}}<+\infty, \label{inversetheor2molecloc}
\end{eqnarray}
the series $f:=\sum_{n\geq 0}{\lambda}_n\textbf{m}_n$ converges in $\mathcal{S'}$ and $\mathcal{H}_{\mathrm{loc}}^{(q,p)}$, with 
\begin{eqnarray*}
\left\|f\right\|_{\mathcal{H}_{\mathrm{loc}}^{(q,p)}}\lsim\left\|\sum_{n\geq 0}\left(\frac{|\lambda_n|}{\left\|\chi_{_{Q^n}}\right\|_{q}}\right)^{\eta}\chi_{_{Q^{n}}}\right\|_{\frac{q}{\eta},\frac{p}{\eta}}^{\frac{1}{\eta}}.
\end{eqnarray*}
\end{thm}
\begin{proof}
Let $\left\{(\textbf{m}_n, Q^n)\right\}_{n\geq 0}$ be a sequence of elements of $\mathcal{M}\ell_{\mathrm{loc}}(q,r,\delta)$ and $\left\{{\lambda}_n\right\}_{n\geq 0}$ be a sequence of scalars satisfying (\ref{inversetheor2molecloc}). Consider an integer $j\geq 0$. Proceeding as in the proof of Theorem \ref{thafondammolec1loc}, we have  
\begin{eqnarray*}
\mathcal{M}_{{\mathrm{loc}}_{_{0}}}\left(\sum_{n=0}^j\lambda_n\textbf{m}_n\right)(x)\lsim\sum_{n=0}^j|\lambda_n|\left(\mathfrak{M}(\textbf{m}_n)(x)\chi_{_{\widetilde{Q^n}}}(x)+\frac{\left[\mathfrak{M}(\chi_{_{Q^{n}}})(x)\right]^{u}}{\left\|\chi_{_{Q^{n}}}\right\|_q}\right),
\end{eqnarray*}
for all $x\in\mathbb{R}^d$, where $\widetilde{Q^n}=4\sqrt{d}Q^n$ and $u=\frac{d+\delta+1}{d}\cdot$ Thus, arguing as in the proof of \cite{AbFt}, Theorem 5.3, p. 1932, we obtain 
\begin{eqnarray*}
\left\|\mathcal{M}_{{\mathrm{loc}}_{_{0}}}\left(\sum_{n=0}^j{\lambda}_n\textbf{m}_n\right)\right\|_{q,p}\lsim\left\|\sum_{n=0}^j\left(\frac{|\lambda_n|}{\left\|\chi_{_{Q^n}}\right\|_{q}}\right)^{\eta}\chi_{_{Q^{n}}}\right\|_{\frac{q}{\eta},\frac{p}{\eta}}^{\frac{1}{\eta}}.
\end{eqnarray*}
And we end as in the proof of \cite{AbFt}, Theorem 4.3, pp. 1914-1915.  
\end{proof}

\section{Characterization of the dual space of $\mathcal{H}_{\mathrm{loc}}^{(q,p)}$}

In this section, we consider an integer $\delta\geq\left\lfloor d\left(\frac{1}{q}-1\right)\right\rfloor$. Let $1<r\leq+\infty$. We denote by $L_{\mathrm{comp}}^r(\mathbb{R}^d)$ the subspace of $L^r$- functions with compact support, and, for all cube $Q$, $L^r(Q)$ stands for the subspace of $L^r$-functions supported in $Q$. If $f\in L_{\mathrm{loc}}^1$ and $Q$ is a cube, then, by Riesz's representation theorem, there exists an unique polynomial of $\mathcal{P_{\delta}}$ (we recall that $\mathcal{P_{\delta}}:=\mathcal{P_{\delta}}(\mathbb{R}^d)$ is the space of polynomial functions of degree at most $\delta$) that we denote by $P_Q^{\delta}(f)$ such that, for all $\mathfrak{q}\in\mathcal{P_{\delta}}$,
\begin{eqnarray}
\int_{Q}\left[f(x)-P_Q^{\delta}(f)(x)\right]\mathfrak{q}(x)dx=0. \label{Campanato1}
\end{eqnarray}

By following \cite{YDYS}, Definition 7.1, p. 51, we introduce the following definition.

\begin{defn}\label{defindudualloc} $\left({\mathcal{L}}_{r,\phi,\delta}^{\mathrm{loc}}(\mathbb{R}^d)\right)$. Let $1\leq r\leq+\infty$ and $\phi: \mathcal{Q}\rightarrow (0,+\infty)$ be a function. The space ${\mathcal{L}}_{r,\phi,\delta}^{\mathrm{loc}}(\mathbb{R}^d)$ is the set of all $f\in L_{\mathrm{loc}}^r$ such that $\left\|f\right\|_{{\mathcal{L}}_{r,\phi,\delta}^{\mathrm{loc}}}<+\infty$, where 
\begin{eqnarray*}
\left\|f\right\|_{{\mathcal{L}}_{r,\phi,\delta}^{\mathrm{loc}}}:&=&\sup_{\underset{|Q|\geq1}{Q\in\mathcal{Q}}}\frac{1}{\phi(Q)}\left(\frac{1}{|Q|}\int_{Q}|f(x)|^rdx\right)^{\frac{1}{r}}\\
&+&\sup_{\underset{|Q|<1}{Q\in\mathcal{Q}}}\frac{1}{\phi(Q)}\left(\frac{1}{|Q|}\int_{Q}\left|f(x)-P_Q^{\delta}(f)(x)\right|^rdx\right)^{\frac{1}{r}},
\end{eqnarray*}
when $r<+\infty$, and
\begin{eqnarray*}
\left\|f\right\|_{\mathcal{L}_{r,\phi,\delta}^{\mathrm{loc}}}:=\sup_{\underset{|Q|\geq1}{Q\in\mathcal{Q}}}\frac{1}{\phi(Q)}\left\|f\right\|_{L^{\infty}(Q)}+\sup_{\underset{|Q|<1}{Q\in\mathcal{Q}}}\frac{1}{\phi(Q)}\left\|f-P_Q^{\delta}(f)\right\|_{L^{\infty}(Q)}, 
\end{eqnarray*}
when $r=+\infty$.
\end{defn}
 
For simplicity, we just denote $\mathcal{L}_{r,\phi,\delta}^{\mathrm{loc}}(\mathbb{R}^d)$ by $\mathcal{L}_{r,\phi,\delta}^{\mathrm{loc}}$. We consider the function $\phi_{1}:\mathcal{Q}\rightarrow(0,+\infty)$ 
defined by 
\begin{eqnarray}
\phi_{1}(Q)=\frac{\left\|\chi_{_{Q}}\right\|_{q,p}}{|Q|}\ ,
\label{Campanatonolocnloc}
\end{eqnarray} 
for all $Q\in\mathcal{Q}$. Also, we set, for every $T\in\left(\mathcal{H}^{(q,p)}\right)^{\ast}$, $$\left\|T\right\|:=\left\|T\right\|_{\left(\mathcal{H}^{(q,p)}\right)^{\ast}}=\sup_{\underset{\left\|f\right\|_{\mathcal{H}^{(q,p)}}\leq 1}{f\in\mathcal{H}^{(q,p)}}}|T(f)|$$ and $\mathcal{L}_{r,\phi_{1},\delta}:=\mathcal{L}_{r,\phi_{1},\delta}(\mathbb{R}^d)$ the set of all $f\in L_{\mathrm{loc}}^r$ such that $\left\|f\right\|_{\mathcal{L}_{r,\phi_{1},\delta}}<+\infty$ (see \cite{AbFt2}, Definition 3.2). In \cite{AbFt2}, we obtained the following duality result for $\mathcal{H}^{(q,p)}$, with $0<q\leq p\leq 1$. 
 
\begin{thm} [\cite{AbFt2}, Theorem 3.3] \label{theoremdual}
Suppose that $0<q\leq p\leq 1$. Let $1<r\leq+\infty$. Then, $\left(\mathcal{H}^{(q,p)}\right)^{\ast}$ is isomorphic to $\mathcal{L}_{r',\phi_1,\delta}$ with equivalent norms, where $\frac{1}{r}+\frac{1}{r'}=1$. More precisely, we have the following assertions:  
\begin{enumerate}
\item Let $g\in\mathcal{L}_{r',\phi_1,\delta}$. Then, the mapping 
$$T_g:f\in\mathcal{H}_{fin}^{(q,p)}\longmapsto\int_{\mathbb{R}^d}g(x)f(x)dx,$$ where $\mathcal{H}_{fin}^{(q,p)}$ is the subspace of $\mathcal{H}^{(q,p)}$ consisting of finite linear combinations of $(q,r,\delta)$-atoms, extends to a unique continuous linear functional on $\mathcal{H}^{(q,p)}$ such that
\begin{eqnarray*}
\left\|T_g\right\|\leq C\left\|g\right\|_{\mathcal{L}_{r',\phi_{1},\delta}}, 
\end{eqnarray*}
where $C>0$ is a constant independent of $g$. \label{dualpoint1}
\item Conversely, for any $T\in\left(\mathcal{H}^{(q,p)}\right)^{\ast}$, there exists $g\in\mathcal{L}_{r',\phi_1,\delta}$ such that $$T(f)=\int_{\mathbb{R}^d}g(x)f(x)dx,\ \text{ for all }\ f\in\mathcal{H}_{fin}^{(q,p)},$$ and 
\begin{eqnarray*}
\left\|g\right\|_{\mathcal{L}_{r',\phi_{1},\delta}}\leq C\left\|T\right\|, 
\end{eqnarray*}
where $C>0$ is a constant independent of $T$. \label{dualpoint2}
\end{enumerate}
\end{thm}
 
Now, for all $T\in\left(\mathcal{H}_{\mathrm{loc}}^{(q,p)}\right)^{\ast}$, for simplicity, we set $$\left\|T\right\|:=\left\|T\right\|_{\left(\mathcal{H}_{\mathrm{loc}}^{(q,p)}\right)^{\ast}}=\sup_{\underset{\left\|f\right\|_{\mathcal{H}_{\mathrm{loc}}^{(q,p)}}\leq 1}{f\in\mathcal{H}_{\mathrm{loc}}^{(q,p)}}}|T(f)|.$$ With Definition \ref{defindudualloc} and Theorem \ref{theoremdual}, we have our following characterization of the dual of $\mathcal{H}_{\mathrm{loc}}^{(q,p)}$, when $0<q\leq p\leq 1$. 

\begin{thm} \label{theoremdualloc}
Suppose that $0<q\leq p\leq 1$. Let $1<r\leq+\infty$. Then, $\left(\mathcal{H}_{\mathrm{loc}}^{(q,p)}\right)^{\ast}$ is isomorphic to $\mathcal{L}_{r',\phi_1,\delta}^{\mathrm{loc}}$ with equivalent norms, where $\frac{1}{r}+\frac{1}{r'}=1$. More precisely, we have the following assertions: 
\begin{enumerate}
\item Let $g\in\mathcal{L}_{r',\phi_1,\delta}^{\mathrm{loc}}$. Then, the mapping
$$T_g:f\in\mathcal{H}_{\mathrm{loc},fin}^{(q,p)}\longmapsto\int_{\mathbb{R}^d}g(x)f(x)dx,$$ where $\mathcal{H}_{\mathrm{loc},fin}^{(q,p)}$ is the subspace of $\mathcal{H}_{\mathrm{loc}}^{(q,p)}$ consisting of finite linear combinations of $(q,r,\delta)$-atoms, extends to a unique continuous linear functional on $\mathcal{H}_{\mathrm{loc}}^{(q,p)}$ such that
\begin{eqnarray*}
\left\|T_g\right\|\leq C\left\|g\right\|_{\mathcal{L}_{r',\phi_{1},\delta}^{\mathrm{loc}}}, 
\end{eqnarray*}
 where $C>0$ is a constant independent of $g$. \label{1theoremdualloc}
\item Conversely, for any $T\in\left(\mathcal{H}_{\mathrm{loc}}^{(q,p)}\right)^{\ast}$, there exists $g\in\mathcal{L}_{r',\phi_1,\delta}^{\mathrm{loc}}$ such that $$T(f)=\int_{\mathbb{R}^d}g(x)f(x)dx,\ \text{ for all }\ f\in\mathcal{H}_{\mathrm{loc},fin}^{(q,p)}\ ,$$
and 
\begin{eqnarray*}
\left\|g\right\|_{\mathcal{L}_{r',\phi_{1},\delta}^{\mathrm{loc}}}\leq C\left\|T\right\|, 
\end{eqnarray*}
 where $C>0$ is a constant independent of $T$. \label{2theoremdualloc}
\end{enumerate}
\end{thm}
\begin{proof} 
We borrow some ideas from \cite{YDYS}, Theorem 7.5. We distinguish two cases: $1<r<+\infty$ and $r=+\infty$. 

Suppose that $1<r<+\infty$. First, we prove (\ref{1theoremdualloc}). Fix $0<\eta<q$. Let $g\in\mathcal{L}_{r',\phi_1,\delta}^{\mathrm{loc}}$ , where $\frac{1}{r}+\frac{1}{r'}=1$. Consider the mapping $T_g$ defined on $\mathcal{H}_{\mathrm{loc},fin}^{(q,p)}$ by $$T_g(f)=\int_{\mathbb{R}^d}g(x)f(x)dx,\ \forall\ f\in\mathcal{H}_{\mathrm{loc},fin}^{(q,p)}.$$ It is easy to verify that $T_g$ is well defined and linear. Let $f\in\mathcal{H}_{\mathrm{loc},fin}^{(q,p)}$. Then, there exist a finite sequence $\left\{\left(a_n,Q^n\right)\right\}_{n=0}^m$ of elements of $\mathcal{A}_{\mathrm{loc}}(q,r,\delta)$ and a finite sequence of scalars $\left\{\lambda_n\right\}_{n=0}^m$ such that $f=\sum_{n=0}^m\lambda_n a_n$. Set $J_1=\left\{n\in\left\{0,1,\ldots,m\right\}:\ |Q^n|<1\right\}$ and $J_2=\left\{n\in\left\{0,1,\ldots,m\right\}:\ |Q^n|\geq1\right\}$. We have $$|T_g(f)|\leq I_1+I_2,$$ with $$I_1=\left|\int_{\mathbb{R}^d}\left(\sum_{n\in J_1}\lambda_n a_n(x)\right)g(x)dx\right| \text{ and } I_2=\left|\int_{\mathbb{R}^d}\left(\sum_{n\in J_2}\lambda_n a_n(x)\right)g(x)dx\right|.$$ With Proposition \ref{InverseHoldMink3}, by arguing as in the proof of Theorem \ref{theoremdual}, we obtain  
\begin{eqnarray*}
I_1&\leq&\left\|\sum_{n\in J_1}\left(\frac{|\lambda_n|}{\left\|\chi_{_{Q^n}}\right\|_q}\right)^{\eta}\chi_{_{Q^n}}\right\|_{\frac{q}{\eta},\frac{p}{\eta}}^{\frac{1}{\eta}}\left[\sup_{\underset{|Q|<1}{Q\in\mathcal{Q}}}\frac{1}{\phi_{1}(Q)}\left(\frac{1}{|Q|}\int_{Q}\left|g(x)-P_{Q}^{\delta}(g)(x)\right|^{r'}dx\right)^{\frac{1}{r'}}\right]\\
&\leq&\left\|\sum_{n=0}^m\left(\frac{|\lambda_n|}{\left\|\chi_{_{Q^n}}\right\|_q}\right)^{\eta}\chi_{_{Q^n}}\right\|_{\frac{q}{\eta},\frac{p}{\eta}}^{\frac{1}{\eta}}\left[\sup_{\underset{|Q|<1}{Q\in\mathcal{Q}}}\frac{1}{\phi_{1}(Q)}\left(\frac{1}{|Q|}\int_{Q}\left|g(x)-P_{Q}^{\delta}(g)(x)\right|^{r'}dx\right)^{\frac{1}{r'}}\right].
\end{eqnarray*}
Also, 
\begin{eqnarray*}
I_2&\leq&\left\|\sum_{n\in J_2}\left(\frac{|\lambda_n|}{\left\|\chi_{_{Q^n}}\right\|_q}\right)^{\eta}\chi_{_{Q^n}}\right\|_{\frac{q}{\eta},\frac{p}{\eta}}^{\frac{1}{\eta}}\left[\sup_{\underset{|Q|\geq1}{Q\in\mathcal{Q}}}\frac{1}{\phi_{1}(Q)}\left(\frac{1}{|Q|}\int_{Q}\left|g(x)\right|^{r'}dx\right)^{\frac{1}{r'}}\right]\\
&\leq&\left\|\sum_{n=0}^m\left(\frac{|\lambda_n|}{\left\|\chi_{_{Q^n}}\right\|_q}\right)^{\eta}\chi_{_{Q^n}}\right\|_{\frac{q}{\eta},\frac{p}{\eta}}^{\frac{1}{\eta}}\left[\sup_{\underset{|Q|\geq1}{Q\in\mathcal{Q}}}\frac{1}{\phi_{1}(Q)}\left(\frac{1}{|Q|}\int_{Q}\left|g(x)\right|^{r'}dx\right)^{\frac{1}{r'}}\right].
\end{eqnarray*}
Hence 
\begin{eqnarray*}
|T_g(f)|\leq\left\|f\right\|_{\mathcal{H}_{\mathrm{loc},fin}^{(q,p)}}\left\|g\right\|_{{\mathcal{L}}_{r',\phi_{1},\delta}^{\mathrm{loc}}}\lsim \left\|f\right\|_{\mathcal{H}_{\mathrm{loc}}^{(q,p)}}\left\|g\right\|_{{\mathcal{L}}_{r',\phi_{1},\delta}^{\mathrm{loc}}},
\end{eqnarray*}
by Theorem \ref{thafondamfiniloc}. This shows that $g\in\left(\mathcal{H}_{\mathrm{loc}}^{(q,p)}\right)^{\ast}$ and 
\begin{eqnarray*}
\left\|g\right\|:=\left\|T_g\right\|\lsim\left\|g\right\|_{\mathcal{L}_{r',\phi_{1},\delta}^{\mathrm{loc}}}, 
\end{eqnarray*}
since $\mathcal{H}_{\mathrm{loc},fin}^{(q,p)}$ is dense in $\mathcal{H}_{\mathrm{loc}}^{(q,p)}$ with respect to the quasi-norm $\left\|\cdot\right\|_{\mathcal{H}_{\mathrm{loc}}^{(q,p)}}$. 
 
Now, prove (\ref{2theoremdualloc}). Let $T\in\left(\mathcal{H}_{\mathrm{loc}}^{(q,p)}\right)^{\ast}$. Fix $0<\eta<q$. Let $Q$ be a cube such that $\ell(Q)\geq1$. We first prove that  
\begin{eqnarray}
\left(\mathcal{H}_{\mathrm{loc}}^{(q,p)}\right)^{\ast}\subset\left(L^r(Q)\right)^{\ast}.\label{Hanbancloc}
\end{eqnarray}
For any $f\in L^r(Q)\backslash\left\{0\right\}$, we set $$a(x):=\left\|f\right\|_{L^r(Q)}^{-1}|Q|^{\frac{1}{r}-\frac{1}{q}}f(x)\ ,$$ for all $x\in\mathbb{R}^d$. Clearly, $(a,Q)\in\mathcal{A}_{\mathrm{loc}}(q,r,\delta)$. Thus, $f\in\mathcal{H}_{\mathrm{loc},fin}^{(q,p)}\subset\mathcal{H}_{\mathrm{loc}}^{(q,p)}$ and
\begin{eqnarray*}
\left\|f\right\|_{\mathcal{H}_{\mathrm{loc}}^{(q,p)}}=\left\|f\right\|_{L^r(Q)}|Q|^{\frac{1}{q}-\frac{1}{r}}\left\|a\right\|_{\mathcal{H}_{\mathrm{loc}}^{(q,p)}}\lsim|Q|^{\frac{1}{q}-\frac{1}{r}}\left\|f\right\|_{L^r(Q)}. 
\end{eqnarray*}
Hence
\begin{eqnarray}
|T(f)|\leq\left\|T\right\|\left\|f\right\|_{\mathcal{H}_{\mathrm{loc}}^{(q,p)}}\lsim |Q|^{\frac{1}{q}-\frac{1}{r}}\left\|T\right\|\left\|f\right\|_{L^r(Q)}, \label{Campanato5loc}
\end{eqnarray}
for all $f\in L^r(Q)$. Thus, $T$ is a continuous linear functional on $L^r(Q)$, with 
\begin{eqnarray*}
\left\|T\right\|_{\left(L^r(Q)\right)^{\ast}}:=\sup_{\underset{\left\|f\right\|_{L^r(Q)}\leq 1}{f\in L^r(Q)}}|T(f)|\lsim|Q|^{\frac{1}{q}-\frac{1}{r}}\left\|T\right\|.
\end{eqnarray*}
This proves (\ref{Hanbancloc}).

Therefore, since $\left(L^{r}(Q)\right)^{\ast}$ is isomorphic to $L^{r'}(Q)$, there exists $g^Q\in L^{r'}(Q)$ such that 
\begin{eqnarray}
T(f)=\int_{Q}f(x)g^Q(x)dx, \text{ for all } f\in L^{r}(Q), \label{Campanato6loc}
\end{eqnarray}
Now, for every cube $Q$ such that $\ell(Q)\geq1$, we consider the fonction $g^Q$ as defined above. Let $\left\{Q_n\right\}_{n\geq 1}$ be an increasing sequence of cubes which converges to $\mathbb{R}^d$ and $\ell(Q_1)\geq1$. For each cube $Q_n$, we have 
\begin{eqnarray}
T(f)=\int_{Q_n}f(x)g^{Q_n}(x)dx, 
\label{Campanato8loc}
\end{eqnarray}
for all $f\in L^r(Q_n)$, by (\ref{Campanato6loc}).

Now, we construct a function $g\in L_{\mathrm{loc}}^{r'}(\mathbb{R}^d)$ such that
\begin{eqnarray}
T(f)=\int_{Q_n}f(x)g(x)dx, \label{Campanato9loc}
\end{eqnarray}
for all $f\in L^r(Q_n)$ and all $n\geq1$. First, assume that $f\in L^r(Q_1)$. Then, 
\begin{eqnarray*}
T(f)=\int_{Q_1}f(x)g^{Q_1}(x)dx,
\end{eqnarray*}
by (\ref{Campanato8loc}). Since $L^r(Q_1)\subset L^r(Q_2)$, also we have
\begin{eqnarray*}
T(f)=\int_{Q_2}f(x)g^{Q_2}(x)dx=\int_{Q_1}f(x)g^{Q_2}(x)dx, 
\end{eqnarray*}
by (\ref{Campanato8loc}). Thus, for all $f\in L^r(Q_1)$,
\begin{eqnarray*}
\int_{Q_1}f(x)\left[g^{Q_1}(x)-g^{Q_2}(x)\right]dx=0.
\end{eqnarray*}
This implies that $g^{Q_1}(x)=g^{Q_2}(x)$, for almost all $x\in Q_1$. Thus, after changing values of $g^{Q_1}$ (or $g^{Q_2}$) on a set of measure zero, we have $g^{Q_1}(x)=g^{Q_2}(x)$, for all $x\in Q_1$. Arguing as above, we obtain
\begin{eqnarray}
g^{Q_n}(x)=g^{Q_{n+1}}(x), \label{Campanato12loc}
\end{eqnarray}
for all $x\in Q_n$ and all $n\geq1$. Set 
\begin{eqnarray}
g_1(x):=g^{Q_1}(x),\ \text{ if }\ x\in Q_1 \label{Campanato13loc}
\end{eqnarray}
and $$g_{n+1}(x):=\left\{\begin{array}{lll}g^{Q_n}(x),&\text{ if }&x\in Q_n\ ,\\ 
g^{Q_{n+1}}(x),&\text{ if }&x\in Q_{n+1}\backslash{Q_n}\ ,\end{array}\right.$$ for all $n\geq 1$. We have 
\begin{eqnarray}
g_{n+1}(x)=g^{Q_{n+1}}(x), \label{Campanato14loc}
\end{eqnarray}
for all $x\in Q_{n+1}$ and all $n\geq 1$, by (\ref{Campanato12loc}). With (\ref{Campanato13loc}) and (\ref{Campanato14loc}), we define the function $g$ on $\mathbb{R}^d$ by $$g(x):=g_n(x)=g^{Q_n}(x)\ ,\ \text{ if }\ x\in Q_n\ ,$$  
for all $n\geq 1$. Then, $g\in L_{\mathrm{loc}}^{r'}$, since $g_n\in L_{\mathrm{loc}}^{r'}$, for all $n\geq 1$, and  
\begin{eqnarray*}
\int_{Q_{n}}f(x)g(x)dx=\int_{Q_{n}}f(x)g_{n}(x)dx=T(f),
\end{eqnarray*}
for all $f\in L^r(Q_{n})$, for all $n\geq 1$, by (\ref{Campanato8loc}). Thus, the function $g$ satisfies (\ref{Campanato9loc}).

To end, we show that $g\in\mathcal{L}_{r',\phi_1,\delta}^{\mathrm{loc}}$ and
\begin{eqnarray}
T(f)=\int_{\mathbb{R}^d}f(x)g(x)dx, \label{Campanato15loc}  
\end{eqnarray}
for all $f\in\mathcal{H}_{\mathrm{loc},fin}^{(q,p)}$. Let $f\in\mathcal{H}_{\mathrm{loc},fin}^{(q,p)}$. We have $f\in L_{\mathrm{comp}}^r(\mathbb{R}^d)$, since $\mathcal{H}_{\mathrm{loc},fin}^{(q,p)}\subset L_{\mathrm{comp}}^r(\mathbb{R}^d)$. Thus, there exists an integer $n\geq 1$ such that $f\in L^r(Q_n)$. Hence (\ref{Campanato15loc}) holds, by (\ref{Campanato9loc}). 

Now, we prove that $g\in\mathcal{L}_{r',\phi_1,\delta}^{\mathrm{loc}}$. Let $Q$ be a cube with $\ell(Q)\geq1$, and $f\in L^r(Q)$ such that $\left\|f\right\|_{L^r(Q)}\leq 1$. Set $$a(x):=|Q|^{\frac{1}{r}-\frac{1}{q}}f(x),$$ for all $x\in\mathbb{R}^d$. We have $(a,Q)\in\mathcal{A}_{\mathrm{loc}}(q,r,\delta)$. Hence 
\begin{eqnarray*}
T(a)=\int_{\mathbb{R}^d}a(x)g(x)dx=\int_{Q}a(x)g(x)dx,
\end{eqnarray*}
by (\ref{Campanato15loc}). Since $T\in\left(\mathcal{H}_{\mathrm{loc}}^{(q,p)}\right)^{\ast}$, we have 
\begin{eqnarray*}
\left|\int_{Q}a(x)g(x)dx\right|=|T(a)|&\leq&\left\|T\right\|\left\|a\right\|_{\mathcal{H}_{\mathrm{loc}}^{(q,p)}}\\
&\lsim&\left\|a\right\|_{\mathcal{H}_{\mathrm{loc},fin}^{(q,p)}}\left\|T\right\| \lsim \frac{1}{\left\|\chi_{_{Q}}\right\|_q}\left\|\chi_{_{Q}}\right\|_{q,p}\left\|T\right\|,
\end{eqnarray*}
by Theorem \ref{thafondamfiniloc} and Remark \ref{remarqdecofinieloc}. Thus, for all $f\in L^r(Q)$ with $\left\|f\right\|_{L^r(Q)}\leq 1$, we have 
\begin{eqnarray*}
\left|\int_{Q}f(x)g(x)dx\right|&\lsim&|Q|^{\frac{1}{q}-\frac{1}{r}}\frac{1}{\left\|\chi_{_{Q}}\right\|_q}\left\|\chi_{_{Q}}\right\|_{q,p}\left\|T\right\|\\
&=&C|Q|^{-\frac{1}{r}}\left\|\chi_{_{Q}}\right\|_{q,p}\left\|T\right\|.
\end{eqnarray*}
This implies that 
\begin{eqnarray*}
\frac{1}{\phi_1(Q)}\left(\frac{1}{|Q|}\int_{Q}\left|g(x)\right|^{r'}dx\right)^{\frac{1}{r'}}\lsim\left\|T\right\|.
\end{eqnarray*}
Therefore, 
\begin{eqnarray}
\sup_{\underset{|Q|\geq1}{Q\in\mathcal{Q}}}\frac{1}{\phi_{1}(Q)}\left(\frac{1}{|Q|}\int_{Q}\left|g(x)\right|^{r'}dx\right)^{\frac{1}{r'}}\lsim \left\|T\right\|. \label{Campanato16loc} 
\end{eqnarray}
Moreover, since $\mathcal{H}^{(q,p)}\subset\mathcal{H}_{\mathrm{loc}}^{(q,p)}$ and $\left\|f\right\|_{\mathcal{H}_{\mathrm{loc}}^{(q,p)}}\leq\left\|f\right\|_{\mathcal{H}^{(q,p)}}$ for all $f\in\mathcal{H}^{(q,p)}$, we have $\left(\mathcal{H}_{\mathrm{loc}}^{(q,p)}\right)^{\ast}\subset\left(\mathcal{H}^{(q,p)}\right)^{\ast}$ and $T|_{\mathcal{H}^{(q,p)}}\in\left(\mathcal{H}^{(q,p)}\right)^{\ast}$. Since (\ref{Campanato15loc}) holds for all $f\in\mathcal{H}_{fin}^{(q,p)}$, from Theorem \ref{theoremdual} (\ref{dualpoint2}), we deduce that $g\in\mathcal{L}_{r',\phi_{1},\delta}$ and 
\begin{eqnarray*}
\left\|g\right\|_{\mathcal{L}_{r',\phi_{1},\delta}}\lsim\left\|T|_{\mathcal{H}^{(q,p)}}\right\|_{\left(\mathcal{H}^{(q,p)}\right)^{\ast}}\lsim\left\|T\right\|.
\end{eqnarray*} 
Hence 
\begin{equation}
\sup_{\underset{|Q|<1}{Q\in\mathcal{Q}}}\frac{1}{\phi_{1}(Q)}\left(\frac{1}{|Q|}\int_{Q}\left|g(x)-P_{Q}^{\delta}(g)(x)\right|^{r'}dx\right)^{\frac{1}{r'}}\lsim\left\|T\right\|.\label{Campanato17loc} 
\end{equation}
Combining (\ref{Campanato16loc}) and (\ref{Campanato17loc}), we obtain $\left\|g\right\|_{\mathcal{L}_{r',\phi_{1},\delta}^{\mathrm{loc}}}\lsim\left\|T\right\|$ and $g\in\mathcal{L}_{r',\phi_{1},\delta}^{\mathrm{loc}}$. This finishes the proof of (\ref{2theoremdualloc}) and hence, the proof in Case $1<r<+\infty$.

The proof in Case $r=+\infty$ is similar to the one of Theorem \ref{theoremdual}, we omit details. This completes the proof of Theorem \ref{theoremdualloc}.
\end{proof}

\begin{remark}
\begin{enumerate}
\item The space $\mathcal{L}_{r,\phi_{1},\delta}^{\mathrm{loc}}$ is independent of $1\leq r<+\infty$ and the integer $\delta\geq\left\lfloor d\left(\frac{1}{q}-1\right)\right\rfloor$, by Theorem \ref{theoremdualloc}.
\item When $q=p=1$, we have $\mathcal{H}_{\mathrm{loc}}^{(q,p)}=\mathcal{H}_{\mathrm{loc}}^1$ and $\phi_{1}\equiv1$. Thus, taking $r=+\infty$ and $\delta=0$ in Theorem \ref{theoremdualloc}, we obtain $\mathcal{L}_{1,\phi_{1},0}^{\mathrm{loc}}=\mathrm{bmo}(\mathbb{R}^d)$ the dual of $\mathcal{H}_{\mathrm{loc}}^1$ introduced by Goldberg in \cite{DGG} and which is the set of all $f\in L_{\mathrm{loc}}^1$ such that  
\begin{eqnarray*}
\left\|f\right\|_{\mathrm{bmo}}:=\sup_{\underset{|Q|\geq1}{Q\in\mathcal{Q}}}\frac{1}{|Q|}\int_{Q}|f(x)|dx+\sup_{\underset{|Q|<1}{Q\in\mathcal{Q}}}\frac{1}{|Q|}\int_{Q}\left|f(x)-f_Q\right|dx<+\infty,
\end{eqnarray*}
where $f_Q=\frac{1}{|Q|}\int_{Q}f(x)dx$.
\item When $q=p<1$, we have $\mathcal{H}_{\mathrm{loc}}^{(q,p)}=\mathcal{H}_{\mathrm{loc}}^q$ and $\phi_{1}(Q)=|Q|^{\frac{1}{q}-1}$, for all $Q\in\mathcal{Q}$. Thus, taking $r=+\infty$ and $\delta=\left\lfloor d\left(\frac{1}{q}-1\right)\right\rfloor$ in Theorem \ref{theoremdualloc}, we obtain $\mathcal{L}_{1,\phi_{1},\delta}^{\mathrm{loc}}=\Lambda_{d\left(\frac{1}{q}-1\right)}$ the dual of $\mathcal{H}_{\mathrm{loc}}^q$ defined by Goldberg in \cite{DGG}, Theorem 5, p. 40. 
\item Also, when $q=p<1$, by taking $1<r<+\infty$, we obtain $\mathcal{L}_{r',\phi_{1},\delta}^{\mathrm{loc}}=\mathrm{bmo^{r'}_{\rho,\omega}}(\mathbb{R}^d)$ the dual of $h_{\omega}^{\Phi}(\mathbb{R}^d)=\mathcal{H}_{\mathrm{loc}}^q$ defined in \cite{YDYS}, in the case where $\Phi(t)=t^q$ and $\omega\equiv1$, with $\rho(t)=t^{-1}/\Phi^{-1}(t^{-1})$, $0<t<+\infty$. 
\item As in \cite{AbFt2}, Remark 3.4, we do not know to characterize the dual space of $\mathcal{H}_{\mathrm{loc}}^{(q,p)}$, whenever $0<q\leq1<p<+\infty$.
\end{enumerate}
\end{remark}

\section{Boundedness of Pseudo-differential Operators}
 
In this section, unless otherwise specified, we assume that $0<q\leq 1$ and $q\leq p<+\infty$. In \cite{DGG}, Theorem 4, D. Goldberg showed that pseudo-differential operators of a certain class ($\mathcal{S}^0$) are bounded on $\mathcal{H}_{\mathrm{loc}}^q$. This result has been extended to other generalizations of $\mathcal{H}_{\mathrm{loc}}^q$ spaces (see \cite{YDYS}, \cite{YDYSB}). We prove that this result extends to $\mathcal{H}_{\mathrm{loc}}^{(q,p)}$ spaces. To see this, we recall the definition of a pseudo-differential operator and some results we will need. For more informations about pseudo-differential operators, the reader can refer to \cite{JD}, \cite{MRVT}, \cite{MA}, \cite{MET}, \cite{HT1} and \cite{MWW}.

\begin{defn}[$\mathcal{S}_{\rho,\sigma}^\mu$] Let $\mu\in\mathbb{R}$, $0\leq\rho\leq1$ and $0\leq\sigma\leq1$. $\mathcal{S}_{\rho,\sigma}^\mu$ is the collection of all complex-values $\mathcal{C}^{\infty}$ functions $\psi(x,\xi)$ in $\mathbb{R}^d\times\mathbb{R}^d$, such that, for all multi-indices $\alpha=(\alpha_1,\ldots,\alpha_d)$ and $\beta=(\beta_1,\ldots,\beta_d)$ there exists a positive number $C_{\alpha\beta}$ with 
\begin{eqnarray*}
|\partial_{x}^\alpha\partial_{\xi}^\beta\psi(x,\xi)|\leq C_{\alpha,\beta}(1+|\xi|)^{\mu+\sigma|\alpha|-\rho|\beta|},\ \ x\in\mathbb{R}^d,\ \xi\in\mathbb{R}^d.
\end{eqnarray*}
\end{defn}

Every element of $\mathcal{S}_{\rho,\sigma}^\mu$ is called symbol. Let $\psi(x,\xi)\in\mathcal{S}_{\rho,\sigma}^\mu$. The operator $T$ defined by  
\begin{eqnarray}
T(f)(x)=\int_{\mathbb{R}^d}\psi(x,\xi)e^{2\pi ix\xi}\widehat{f}(\xi)d\xi, \label{pseudodif1}
\end{eqnarray}
for all $f\in\mathcal{S}$ and all $x\in\mathbb{R}^d$, is called a $\mathcal{S}_{\rho,\sigma}^\mu$ pseudo-differential operator. For every $\mathcal{S}_{\rho,\sigma}^\mu$ pseudo-differential operator $T$, (\ref{pseudodif1}) also writes  
\begin{eqnarray}
T(f)(x)=\int_{\mathbb{R}^d}K(x,x-y)f(y)dy, \label{pseudodif2}
\end{eqnarray}
where 
\begin{eqnarray}
K(x,z)=\int_{\mathbb{R}^d}\psi(x,\xi)e^{2\pi iz\xi}d\xi. \label{pseudodif3}
\end{eqnarray}
For (\ref{pseudodif2}) and (\ref{pseudodif3}), See \cite{JD}, Chap. 5, pp. 113-114 or \cite{MA}, Chap.VI, (8), p. 235; (43), p. 250. When $\mu=\sigma=0$ and $\rho=1$, we just denote by $\mathcal{S}^0$ the class $\mathcal{S}_{1,0}^0$. 

In the rest of this section, we are interested in $\mathcal{S}^0$ pseudo-differential operators. It's well known that, when $\psi\in\mathcal{S}^0$, for every $f\in L^2$, (\ref{pseudodif2}) is valid for almost all $x\notin\mathrm{supp}(f)$, and $T$ is a bounded operator on $L^r$, for any $1<r<+\infty$ (see \cite{JD}, Chap. 5, pp. 113-114 and \cite{MA}, Chap.VI, Proposition 4, p. 250). Also, we have the following

\begin{lem}[\cite{DGG}, Lemma 6] \label{lempseudodif} 
Let $T$ be a $\mathcal{S}^0$ pseudo-differential operator. If $\phi\in\mathcal{S}$, then $T_t(f)=\phi_t\ast T(f)$ has a symbol $\psi_t$ which satisfies 
\begin{eqnarray}
|\partial_{x}^\alpha\partial_{\xi}^\beta\psi_t(x,\xi)|\leq C_{\alpha,\beta}(1+|\xi|)^{-|\beta|},\ x\in\mathbb{R}^d,\ \xi\in\mathbb{R}^d, \label{pseudodif4}
\end{eqnarray}
and a kernel $K_t(x,z)=FT_\xi\psi_t(x,\xi)$ which satisfies
\begin{eqnarray}
|\partial_{x}^\alpha\partial_{z}^\beta K_t(x,z)|\leq C_{\alpha,\beta}|z|^{-d-|\beta|},\ x\in\mathbb{R}^d,\ z\in\mathbb{R}^d\backslash\left\{0\right\}, \label{pseudodif5}
\end{eqnarray}
where $C_{\alpha,\beta}>0$ is a constant independent of $t$ if $0<t\leq1$.
\end{lem}

Our main result in this section is the following theorem.

\begin{thm} \label{theopseudodif}
Let $T$ be a $\mathcal{S}^0$ pseudo-differential operator. Then, $T$ extends to a bounded operator from $\mathcal{H}_{\mathrm{loc}}^{(q,p)}$ to $\mathcal{H}_{\mathrm{loc}}^{(q,p)}$, for all $0<q\leq1$. 
\end{thm}
\begin{proof}
Let $\delta\geq\left\lfloor d\left(\frac{1}{q}-1\right)\right\rfloor$ be an integer and $r>\max\left\{2;p\right\}$ be a real. Consider the set $\mathcal{A}_{\mathrm{loc}}(q,r,\delta)$. Let $f\in\mathcal{H}_{\mathrm{loc},fin}^{(q,p)}$. Then, there exist a finite sequence $\left\{(\textbf{a}_n, Q^n)\right\}_{n=0}^j$ of elements of $\mathcal{A}_{\mathrm{loc}}(q,r,\delta)$ and  a finite sequence of scalars $\left\{{\lambda}_n\right\}_{n=0}^j$ such that $f=\sum_{n=0}^j\lambda_n\textbf{a}_n$. Set $\widetilde{Q^n}:=4\sqrt{d}Q^n$, $n\in\left\{0,1,\ldots,j\right\}$, and denote by $x_n$ and $\ell_n$ respectively the center and side-length of $Q^n$. We have 
\begin{eqnarray*}
\mathcal{M}_{{\mathrm{loc}}_{_{0}}}(T(f))(x)\leq\sum_{n=0}^j|\lambda_n|\left[\mathcal{M}_{{\mathrm{loc}}_{_{0}}}(T(\textbf{a}_n))(x)\chi_{_{\widetilde{Q^n}}}(x)+\mathcal{M}_{{\mathrm{loc}}_{_{0}}}(T(\textbf{a}_n))(x)\chi_{_{\mathbb{R}^d\backslash{\widetilde{Q^n}}}}(x)\right],
\end{eqnarray*}
for all $x\in\mathbb{R}^d$. Hence
\begin{eqnarray*}
\left\|\mathcal{M}_{{\mathrm{loc}}_{_{0}}}(T(f))\right\|_{q,p}\lsim I+J, 
\end{eqnarray*}
with $$I=\left\|\sum_{n=0}^j|\lambda_n|\mathcal{M}_{{\mathrm{loc}}_{_{0}}}(T(\textbf{a}_n))\chi_{_{\widetilde{Q^n}}}\right\|_{q,p}\ \text{ and }\ J=\left\|\sum_{n=0}^j|\lambda_n|\mathcal{M}_{{\mathrm{loc}}_{_{0}}}(T(\textbf{a}_n))\chi_{_{\mathbb{R}^d\backslash{\widetilde{Q^n}}}}\right\|_{q,p}.$$ Fix $0<\eta<q$. Arguing as in the proof of \cite{AbFt2}, Theorem 4.20, we obtain 
\begin{eqnarray*}
I\leq\left\|\sum_{n=0}^j|\lambda_n|^{\eta}\left(\mathcal{M}_{{\mathrm{loc}}_{_{0}}}(T(\textbf{a}_n))\chi_{_{\widetilde{Q^n}}}\right)^{\eta}\right\|_{\frac{q}{\eta},\frac{p}{\eta}}^{\frac{1}{\eta}}\lsim \left\|\sum_{n=0}^j\left(\frac{|\lambda_n|}{\left\|\chi_{_{Q^n}}\right\|_{q}}\right)^{\eta}\chi_{_{Q^n}}\right\|_{\frac{q}{\eta},\frac{p}{\eta}}^{\frac{1}{\eta}}.
\end{eqnarray*}
To estimate $J$, it suffices to prove that 
\begin{eqnarray}
\mathcal{M}_{{\mathrm{loc}}_{_{0}}}(T(\textbf{a}_n))(x)\lsim\frac{\left[\mathfrak{M}(\chi_{_{Q^n}})(x)\right]^{\vartheta}}{\left\|\chi_{_{Q^n}}\right\|_q}\ , \label{applicattheo17loc}
\end{eqnarray}
for all $x\notin\widetilde{Q^n}$, where $\vartheta=\frac{d+\delta+1}{d}\cdot$ To show (\ref{applicattheo17loc}), we distinguish two cases: Case $\ell_n<1$ and Case $\ell_n\geq1$.

Case: $\ell_n<1$. Then, the atom $\textbf{a}_n$ satisfies the vanishing condition. Consider $x\notin\widetilde{Q^n}$. Let $0<t\leq1$. With Lemma \ref{lempseudodif} (\ref{pseudodif5}), by arguing as in the proof of \cite{AbFt2}, Theorem 4.20, we obtain   
\begin{eqnarray*}
|(T(\textbf{a}_n)\ast\varphi_t)(x)|=\left|\int_{Q^n}K_t(x,x-z)\textbf{a}_n(z)dz\right|\lsim \frac{\left[\mathfrak{M}(\chi_{_{Q^n}})(x)\right]^{\vartheta}}{\left\|\chi_{_{Q^n}}\right\|_q}\ ,
\end{eqnarray*}
with $\vartheta=\frac{d+\delta+1}{d}\cdot$ Hence $$\mathcal{M}_{{\mathrm{loc}}_{_{0}}}(T(\textbf{a}_n))(x)\lsim\frac{\left[\mathfrak{M}(\chi_{_{Q^n}})(x)\right]^{\vartheta}}{\left\|\chi_{_{Q^n}}\right\|_q}\ ,$$ with $\vartheta=\frac{d+\delta+1}{d}$ , when $\ell_n<1$.

Case: $\ell_n\geq1$. Consider $x\notin\widetilde{Q^n}$. Let $0<t\leq1$. We have 
\begin{eqnarray*}
|(T(\textbf{a}_n)\ast\varphi_t)(x)|\leq\int_{Q^n}|K_t(x,x-z)||\textbf{a}_n(z)|dz.
\end{eqnarray*}
Also, for all $z\in Q^n$, we have 
\begin{eqnarray}
1<|x-z|\approx|x-x_n|. \label{pseudodif7}
\end{eqnarray} 
Combining (\ref{pseudodif7}) with Lemma \ref{lempseudodif} (\ref{pseudodif4}), we have, by \cite{MA}, (9), p. 235 (also see \cite{MET}, (0.5.5), p. 16), for all integer $M>0$, the existence of a constant $C_M>0$ independent of $t$, such that  
\begin{eqnarray}
|K_t(x,x-z)|\leq C_M|x-z|^{-M} \leq C(M)|x-x_n|^{-M}. \label{pseudodif8} 
\end{eqnarray}
Taking $M=d+\delta+1$, it follows from (\ref{pseudodif8}) that 
\begin{eqnarray*}
|(T(\textbf{a}_n)\ast\varphi_t)(x)|&\leq&C(M)|x-x_n|^{-M}\int_{Q^n}|\textbf{a}_n(z)|dz\\
&\lsim&\frac{1}{\left\|\chi_{_{Q^n}}\right\|_q}\frac{\ell_n^d}{|x-x_n|^{d+\delta+1}}\\
&\lsim&\frac{1}{\left\|\chi_{_{Q^n}}\right\|_q}\frac{\ell_n^{d+\delta+1}}{|x-x_n|^{d+\delta+1}}\lsim\frac{\left[\mathfrak{M}(\chi_{_{Q^n}})(x)\right]^{\vartheta}}{\left\|\chi_{_{Q^n}}\right\|_q}\ , 
\end{eqnarray*}
with $\vartheta=\frac{d+\delta+1}{d}$, since $\ell_n\geq1$ implies that $\ell_n^d\leq\ell_n^{d+\delta+1}$. Hence
$$\mathcal{M}_{{\mathrm{loc}}_{_{0}}}(T(\textbf{a}_n))(x)\lsim\frac{\left[\mathfrak{M}(\chi_{_{Q^n}})(x)\right]^{\vartheta}}{\left\|\chi_{_{Q^n}}\right\|_q}\ ,$$ with $\vartheta=\frac{d+\delta+1}{d}$ , when $\ell_n\geq1$.

Combining these two cases, we obtain (\ref{applicattheo17loc}).\\ With (\ref{applicattheo17loc}), we obtain 
\begin{eqnarray*}
J\lsim \left\|\sum_{n=0}^j\left(\frac{|\lambda_n|}{\left\|\chi_{_{Q^n}}\right\|_{q}}\right)^{\eta}\chi_{_{Q^{n}}}\right\|_{\frac{q}{\eta},\frac{p}{\eta}}^{\frac{1}{\eta}}.
\end{eqnarray*}
Finally, 
\begin{eqnarray*}
\left\|\mathcal{M}_{{\mathrm{loc}}_{_{0}}}(T(f))\right\|_{q,p}\lsim\left\|\sum_{n=0}^j\left(\frac{|\lambda_n|}{\left\|\chi_{_{Q^n}}\right\|_{q}}\right)^{\eta}\chi_{_{Q^n}}\right\|_{\frac{q}{\eta},\frac{p}{\eta}}^{\frac{1}{\eta}}.
\end{eqnarray*}
Thus, 
\begin{eqnarray*}
\left\|T(f)\right\|_{\mathcal{H}_{\mathrm{loc}}^{(q,p)}}=\left\|\mathcal{M}_{{\mathrm{loc}}_{_{0}}}(T(f))\right\|_{q,p}\lsim\left\|f\right\|_{\mathcal{H}_{\mathrm{loc},fin}^{(q,p)}}\lsim\left\|f\right\|_{\mathcal{H}_{\mathrm{loc}}^{(q,p)}},
\end{eqnarray*}
by Theorem \ref{thafondamfiniloc}. Therefore, $T$ is bounded from $\mathcal{H}_{\mathrm{loc},fin}^{(q,p)}$ to $\mathcal{H}_{\mathrm{loc}}^{(q,p)}$ and the density of $\mathcal{H}_{\mathrm{loc},fin}^{(q,p)}$ in $\mathcal{H}_{\mathrm{loc}}^{(q,p)}$ with respect to the quasi-norm $\left\|\cdot\right\|_{\mathcal{H}_{\mathrm{loc}}^{(q,p)}}$ yields the result. 
\end{proof}

\begin{cor} \label{theopseudodifbis}
Let $T$ be a $\mathcal{S}^0$ pseudo-differential operator. Then, $T$ extends to a bounded operator from $\mathcal{H}_{\mathrm{loc}}^{(q,p)}$ to $\mathcal{H}_{\mathrm{loc}}^{(q_1,p_1)}$, if $q_1\leq q$ and $p\leq p_1$, for all $0<q\leq1$. 
\end{cor}
\begin{proof}
Corollary \ref{theopseudodifbis} follows from Proposition \ref{ratotute} and Theorem \ref{theopseudodif}. 
\end{proof}

As consequence of the proof of Theorem \ref{theopseudodif}, every convolution operator $T$ as in \cite{AbFt2}, Theorem 4.20, which satisfies in addition a certain condition, can be extended to a bounded operator from $\mathcal{H}_{\mathrm{loc}}^{(q,p)}$ to $\mathcal{H}_{\mathrm{loc}}^{(q,p)}$. More precisely, we have the following result.

\begin{prop}\label{propopseudodif} 
Let $T$ be a convolution operator, $T(f)=K\ast f,\  f\in\mathcal{S}$, where the distribution $K$ satisfies the following conditions:
\begin{enumerate}
\item $|\widehat{K}(x)|\leq A$, \label{pseudodif9} 
\item $|\partial^\beta{K}(x)|\leq B|x|^{-d-|\beta|}$, for all $|\beta|\leq\left\lfloor d\left(\frac{1}{q}-1\right)\right\rfloor+1$; \label{pseudodif10} 
\item $|K(x)|\leq C_K|x|^{-\frac{d}{q}-\gamma}$, with $\gamma\geq1$, \label{pseudodif11} 
\end{enumerate}
where $A>0$, $B>0$ and $C_K>0$ are constants independent of $x$ and $\beta$. Then, $T$ extends to a bounded operator from $\mathcal{H}_{\mathrm{loc}}^{(q,p)}$ to $\mathcal{H}_{\mathrm{loc}}^{(q,p)}$.
\end{prop}
\begin{proof}
Conditions (\ref{pseudodif9}) and (\ref{pseudodif10}) imply that $T$ is bounded on $L^r$, for any $1<r<+\infty$, by  \cite{AbFt2}, Remark 4.14 and \cite{JD}, Theorem 5.1. Set $\delta=\left\lfloor d\left(\frac{1}{q}-1\right)\right\rfloor\cdot$ Let $r>\max\left\{2;p\right\}$ be a real. Consider the set $\mathcal{A}_{\mathrm{loc}}(q,r,\delta)$. Let $f\in\mathcal{H}_{\mathrm{loc},fin}^{(q,p)}$. Then, there exist a finite sequence $\left\{(\textbf{a}_n, Q^n)\right\}_{n=0}^j$ in $\mathcal{A}_{\mathrm{loc}}(q,r,\delta)$ and a finite sequence of scalars $\left\{{\lambda}_n\right\}_{n=0}^j$ such that $f=\sum_{n=0}^j\lambda_n\textbf{a}_n$. Set $\widetilde{Q^n}:=4\sqrt{d}Q^n$, $n\in\left\{0,1,\ldots,j\right\}$, and denote by $x_n$ and $\ell_n$ respectively the center and side-length of $Q^n$. We have
\begin{eqnarray*}
\left\|\mathcal{M}_{{\mathrm{loc}}_{_{0}}}(T(f))\right\|_{q,p} \lsim I+J,
\end{eqnarray*}
with $$I=\left\|\sum_{n=0}^j|\lambda_n|\mathcal{M}_{{\mathrm{loc}}_{_{0}}}(T(\textbf{a}_n))\chi_{_{\widetilde{Q^n}}}\right\|_{q,p}\ \text{and }\ J=\left\|\sum_{n=0}^j|\lambda_n|\mathcal{M}_{{\mathrm{loc}}_{_{0}}}(T(\textbf{a}_n))\chi_{_{\mathbb{R}^d\backslash{\widetilde{Q^n}}}}\right\|_{q,p}.$$ Fix $0<\eta<q$. As in the proof of Theorem \ref{theopseudodif}, we have 
\begin{eqnarray*}
I\lsim \left\|\sum_{n=0}^j\left(\frac{|\lambda_n|}{\left\|\chi_{_{Q^n}}\right\|_{q}}\right)^{\eta}\chi_{_{Q^n}}\right\|_{\frac{q}{\eta},\frac{p}{\eta}}^{\frac{1}{\eta}}.
\end{eqnarray*}
To estimate $J$, it suffices to show that 
\begin{eqnarray}
\mathcal{M}_{{\mathrm{loc}}_{_{0}}}(T(\textbf{a}_n))(x)\lsim\frac{\left[\mathfrak{M}(\chi_{_{Q^n}})(x)\right]^{\vartheta}}{\left\|\chi_{_{Q^n}}\right\|_q}\ , \label{applicattheo17loc1}
\end{eqnarray}
for all $x\notin\widetilde{Q^n}$, where $\vartheta=\frac{d+\delta+1}{d}\cdot$ To prove (\ref{applicattheo17loc1}), we distinguish two cases, as in the proof of Theorem \ref{theopseudodif}.

Case: $\ell_n<1$. Consider $x\notin\widetilde{Q^n}$. For all $0<t\leq1$, we put $$K^{(t)}:=K\ast\varphi_t.$$ We have 
\begin{eqnarray}
\sup_{0<t\leq1}|\widehat{K^{(t)}}(z)|\leq\left\|\widehat{\varphi}\right\|_{\infty}A\ , \label{applicattheo18loc}
\end{eqnarray}
for all $z\in\mathbb{R}^d$, and  
\begin{eqnarray}
\sup_{0<t\leq1}|\partial^{\beta}_z K^{(t)}(z)|\leq \frac{C_{\varphi,A,B,\delta,d}}{|z|^{d+|\beta|}}\ , \label{applicattheo19loc}
\end{eqnarray}
for all $z\neq0$ and all multi-indices $\beta$ with $|\beta|\leq\delta+1$. The proof of (\ref{applicattheo18loc}) and (\ref{applicattheo19loc}) is similar to the one of (4.41) and (4.42) of the proof of \cite{AbFt2}, Theorem 4.20, we omit details. With (\ref{applicattheo19loc}), we obtain $$\mathcal{M}_{{\mathrm{loc}}_{_{0}}}(T(\textbf{a}_n))(x)\lsim\frac{\left[\mathfrak{M}(\chi_{_{Q^n}})(x)\right]^{\vartheta}}{\left\|\chi_{_{Q^n}}\right\|_q}\ ,$$ where $\vartheta=\frac{d+\delta+1}{d}\cdot$

Case: $\ell_n\geq1$. Consider $x\notin\widetilde{Q^n}$. For all $0<t\leq1$, we have  
\begin{eqnarray*}
|(T(\textbf{a}_n)\ast\varphi_t)(x)|\leq\int_{Q^n}|K^{(t)}(x-z)||\textbf{a}_n(z)|dz.
\end{eqnarray*}
Furthermore, for all $z\in Q^n$, we have
\begin{eqnarray}
\frac{3}{2}<|x-z|\approx|x-x_n| \label{pseudodif12}
\end{eqnarray}
and 
\begin{eqnarray}
|K^{(t)}(x-z)|\leq C(\varphi,d,q,\gamma,C_K)|x-z|^{-\frac{d}{q}-\gamma}, \label{pseudodif13}
\end{eqnarray}
where $C(\varphi,d,q,\gamma,C_K)>0$ is a constant independent of $t$, $x$ and $z$. We prove a general version of (\ref{pseudodif13}), namely,
\begin{eqnarray}
|K^{(t)}(z)|\leq C(\varphi,d,q,\gamma,C_K)|z|^{-\frac{d}{q}-\gamma}, \label{pseudodif14}
\end{eqnarray}
for all $|z|\geq\frac{3}{2}\cdot$ For the proof of (\ref{pseudodif14}), consider $z\in\mathbb{R}^d$ such that $|z|\geq\frac{3}{2}\cdot$ Then, 
\begin{eqnarray*}
K^{(t)}(z)=\int_{\mathbb{R}^d}K(z-y)\varphi_t(y)dy=\int_{|y|<t}K(z-y)\varphi_t(y)dy,
\end{eqnarray*}
since $\text{supp}(\varphi)\subset B(0,1)$ and $|z-y|\geq|z|-|y|>\frac{3}{2}-t>0$, for all $|y|<t$. Hence 
\begin{eqnarray*}
|K^{(t)}(z)|\leq\int_{|y|<t}\frac{C_K}{|z-y|^{\frac{d}{q}+\gamma}}|\varphi_t(y)|dy,
\end{eqnarray*}
by (\ref{pseudodif11}). But, for all $|y|<t$, we have 
\begin{eqnarray*}
|z-y|\geq|z|-|y|>|z|-\frac{2}{3}|z|=\frac{1}{3}|z|,
\end{eqnarray*}
since $|z|\geq\frac{3}{2}\geq\frac{3}{2}t>\frac{3}{2}|y|$. Therefore, 
\begin{eqnarray*}
|K^{(t)}(z)|&\leq&3^{\frac{d}{q}+\gamma}\frac{C_K}{|z|^{\frac{d}{q}+\gamma}}\int_{|y|<t}|\varphi_t(y)|dy\\
&\leq&C(\varphi,d,q,\gamma,C_K)|z|^{-\frac{d}{q}-\gamma},
\end{eqnarray*}
with $C(\varphi,d,q,\gamma,C_K):=3^{\frac{d}{q}+\gamma}\left\|\varphi\right\|_{1}C_K>0$. This establishes  (\ref{pseudodif14}). Then, (\ref{pseudodif13}) immediately follows from (\ref{pseudodif14}), by (\ref{pseudodif12}).\\
With (\ref{pseudodif12}) and (\ref{pseudodif13}), by proceeding as in the proof of Theorem \ref{theopseudodif}, we obtain   
\begin{eqnarray*}
|(T(\textbf{a}_n)\ast\varphi_t)(x)|&\lsim &\frac{1}{\left\|\chi_{_{Q^n}}\right\|_q}\frac{\ell_n^d}{|x-x_n|^{\frac{d}{q}+\gamma}}\lsim \frac{1}{\left\|\chi_{_{Q^n}}\right\|_q}\frac{\ell_n^{d+\delta+1}}{|x-x_n|^{\frac{d}{q}+\gamma}}\\
&\lsim &\frac{1}{\left\|\chi_{_{Q^n}}\right\|_q}\frac{\ell_n^{d+\delta+1}}{|x-x_n|^{d+\delta+1}}\lsim\frac{\left[\mathfrak{M}(\chi_{_{Q^n}})(x)\right]^{\vartheta}}{\left\|\chi_{_{Q^n}}\right\|_q}\ ,
\end{eqnarray*}
with $\vartheta=\frac{d+\delta+1}{d}$ , since $d+\delta+1\leq\frac{d}{q}+\gamma$ and $|x-x_n|>1$. Hence $$\mathcal{M}_{{\mathrm{loc}}_{_{0}}}(T(\textbf{a}_n))(x)\lsim\frac{\left[\mathfrak{M}(\chi_{_{Q^n}})(x)\right]^{\vartheta}}{\left\|\chi_{_{Q^n}}\right\|_q}\ ,$$ with $\vartheta=\frac{d+\delta+1}{d}\cdot$ 

Combining these two cases, we obtain (\ref{applicattheo17loc1}).\\ With (\ref{applicattheo17loc1}), we end as in the proof of Theorem  \ref{theopseudodif}.
\end{proof}

In Proposition \ref{propopseudodif}, the hypothesis on $\gamma$ can be weakened by taking, in (\ref{pseudodif11}), 
\begin{eqnarray}
\gamma>d\left(\frac{1}{q}-1\right)-\left\lfloor d\left(\frac{1}{q}-1\right)\right\rfloor\cdot \label{pseudodif16}
\end{eqnarray}
In fact, when $\gamma\geq1$, we clearly have (\ref{pseudodif16}), since $$0\leq d\left(\frac{1}{q}-1\right)-\left\lfloor d\left(\frac{1}{q}-1\right)\right\rfloor<1.$$ We are going to prove Proposition \ref{propopseudodif} under Condition (\ref{pseudodif16}). The proof is similar to the previous, we only present points which change.

\begin{proof}
With the same elements of the previous proof, only the estimation of $J$ presents some modifications. Endeed, in Case : $\ell_n\geq1$, we have $$\mathcal{M}_{{\mathrm{loc}}_{_{0}}}(T(\textbf{a}_n))(x)\leq C(\varphi,d,q,\gamma,C_K)\frac{\left[\mathfrak{M}(\chi_{_{Q^n}})(x)\right]^{\vartheta}}{\left\|\chi_{_{Q^n}}\right\|_q}\ ,$$ with $\vartheta=\frac{d+\delta+\gamma}{d}\cdot$ Thus, the expression of $\vartheta$ is no longer the same in the two cases (Case: $\ell_n<1$ and Case: $\ell_n\geq1$). And hence, we have no longer (\ref{applicattheo17loc1}), for all $x\notin\widetilde{Q^n}$, $n\in\left\{0,1,\ldots,j\right\}$. To overcome the problem, we write 
\begin{eqnarray*}
\sum_{n=0}^j|\lambda_n|\mathcal{M}_{{\mathrm{loc}}_{_{0}}}(T(\textbf{a}_n))\chi_{_{\mathbb{R}^d\backslash{\widetilde{Q^n}}}}&=&\sum_{n:\ \ell_n<1}|\lambda_n|\mathcal{M}_{{\mathrm{loc}}_{_{0}}}(T(\textbf{a}_n))\chi_{_{\mathbb{R}^d\backslash{\widetilde{Q^n}}}}\\
&+&\sum_{n:\ \ell_n\geq1}|\lambda_n|\mathcal{M}_{{\mathrm{loc}}_{_{0}}}(T(\textbf{a}_n))\chi_{_{\mathbb{R}^d\backslash{\widetilde{Q^n}}}}.
\end{eqnarray*}
Thus, $J\lsim J_1+J_2$, with $J_1=\left\|\sum_{n:\ \ell_n<1}|\lambda_n|\mathcal{M}_{{\mathrm{loc}}_{_{0}}}(T(\textbf{a}_n))\chi_{_{\mathbb{R}^d\backslash{\widetilde{Q^n}}}}\right\|_{q,p}$ and $J_2=\left\|\sum_{n:\ \ell_n\geq1}|\lambda_n|\mathcal{M}_{{\mathrm{loc}}_{_{0}}}(T(\textbf{a}_n))\chi_{_{\mathbb{R}^d\backslash{\widetilde{Q^n}}}}\right\|_{q,p}$. We have   
\begin{eqnarray*}
J_1\lsim\left\|\sum_{n:\ \ell_n<1}|\lambda_n|\frac{\left[\mathfrak{M}(\chi_{_{Q^n}})\right]^{\vartheta}}{\left\|\chi_{_{Q^n}}\right\|_q}\right\|_{q,p}&\lsim& \left\|\sum_{n:\ \ell_n<1}\left(\frac{|\lambda_n|}{\left\|\chi_{_{Q^n}}\right\|_{q}}\right)^{\eta}\chi_{_{Q^{n}}}\right\|_{\frac{q}{\eta},\frac{p}{\eta}}^{\frac{1}{\eta}}\\
&\lsim& \left\|\sum_{n=0}^j\left(\frac{|\lambda_n|}{\left\|\chi_{_{Q^n}}\right\|_{q}}\right)^{\eta}\chi_{_{Q^{n}}}\right\|_{\frac{q}{\eta},\frac{p}{\eta}}^{\frac{1}{\eta}},
\end{eqnarray*}
where $\vartheta=\frac{d+\delta+1}{d}\cdot$ Likewise,  
\begin{eqnarray*}
J_2\lsim\left\|\sum_{n:\ \ell_n\geq1}|\lambda_n|\frac{\left[\mathfrak{M}(\chi_{_{Q^n}})\right]^{\vartheta}}{\left\|\chi_{_{Q^n}}\right\|_q}\right\|_{q,p}\lsim\left\|\sum_{n=0}^j\left(\frac{|\lambda_n|}{\left\|\chi_{_{Q^n}}\right\|_{q}}\right)^{\eta}\chi_{_{Q^{n}}}\right\|_{\frac{q}{\eta},\frac{p}{\eta}}^{\frac{1}{\eta}},
\end{eqnarray*}
where $\vartheta=\frac{d+\delta+\gamma}{d}\cdot$ Hence 
\begin{eqnarray*}
J\lsim\left\|\sum_{n=0}^j\left(\frac{|\lambda_n|}{\left\|\chi_{_{Q^n}}\right\|_{q}}\right)^{\eta}\chi_{_{Q^{n}}}\right\|_{\frac{q}{\eta},\frac{p}{\eta}}^{\frac{1}{\eta}}.
\end{eqnarray*}
And we end as in the previous proof.
\end{proof}

Another consequence of Theorem \ref{theopseudodif} is that $\mathcal{H}_{\mathrm{loc}}^{(q,p)}$ ($0<q\leq1$, $q\leq p$) is stable under multiplication by the Schwartz class $\mathcal{S}$. Endeed, let $\phi\in\mathcal{S}$. The operator $T$ defined by  
\begin{eqnarray*}
T(f)(x)=\phi(x)f(x),\ x\in\mathbb{R}^d, 
\end{eqnarray*}
for all $f\in\mathcal{S}$, is clearly a $\mathcal{S}^0$ pseudo-differential operator. Therefore, by Theorem \ref{theopseudodif}, $T$ extends to a bounded operator on $\mathcal{H}_{\mathrm{loc}}^{(q,p)}$. Thus, we can define by extension the product of an element $f\in\mathcal{H}_{\mathrm{loc}}^{(q,p)}$ and a function $\phi\in\mathcal{S}$ such that the product also denoted by $\phi f\in\mathcal{H}_{\mathrm{loc}}^{(q,p)}$. 

This result extends the one of \cite{DGG} to $\mathcal{H}_{\mathrm{loc}}^{(q,p)}$ spaces, for $0<q\leq1$ and $q\leq p$.

Let $T_\psi$ be the $\mathcal{S}^0$ pseudo-differential operator associated with $\psi\in\mathcal{S}^0$. By \cite{MA}, (4), p. 233, 
\begin{eqnarray}
\left\langle T_\psi f,g \right\rangle=\left\langle f,(T_\psi)^{\ast}g \right\rangle, \label{pseudodif18}
\end{eqnarray}
for all $f,g\in\mathcal{S}$, where $\left\langle f,g \right\rangle$ denotes $\int_{\mathbb{R}^d}f(x)\overline{g}(x)dx$ and $(T_\psi)^{\ast}$ is the adjoint of $T_\psi$. Also, there exists $\psi^{\ast}\in\mathcal{S}^0$ such that $(T_\psi)^{\ast}=T_{\psi^{\ast}}$, by \cite{MA}, Proposition, p. 259. Combining these facts with Theorems \ref{theopseudodif} and \ref{theoremdualloc}, we have the following result.

\begin{cor}
Suppose that $0<q\leq p\leq1$. Let $\delta\geq\left\lfloor d\left(\frac{1}{q}-1\right)\right\rfloor$ be an integer, $1<r\leq+\infty$ and $T$ be a $\mathcal{S}^0$ pseudo-differential operator. Then, there exists a positive constant $C$ such that, for all $f\in\mathcal{L}_{r',\phi_1,\delta}^{\mathrm{loc}}$ ,
\begin{eqnarray*}
\left\|T(f)\right\|_{\mathcal{L}_{r',\phi_{1},\delta}^{\mathrm{loc}}}\leq C\left\|f\right\|_{\mathcal{L}_{r',\phi_{1},\delta}^{\mathrm{loc}}}.
\end{eqnarray*}
\end{cor}

\end{document}